\documentclass[a4paper,10pt]{amsart}
\usepackage[utf8]{inputenc}
\usepackage{amsthm}
\usepackage{amsmath}
\usepackage{bm}
\usepackage{enumitem}
\usepackage{amsfonts}
\usepackage{amssymb}
\usepackage{mathtools}
\usepackage{appendix}
\usepackage{mathrsfs}
\usepackage{setspace}
\usepackage{comment}
\usepackage{xcolor}
\usepackage{todonotes}
\usepackage{thmtools} 
\usepackage[foot]{amsaddr}
\RequirePackage[a4paper,top=2.54cm,bottom=2.54cm,left=1.90cm,right=1.90cm,%
                headsep=1em,includehead,includefoot]{geometry}
\usepackage[pdfdisplaydoctitle,colorlinks,breaklinks,urlcolor=blue,linkcolor=blue,citecolor=blue]{hyperref} 
\usepackage[nameinlink,capitalise]{cleveref}

\newcommand{\C}{\mathbb{C}}

\renewcommand{\H}{\mathcal{H}}

\newcommand{\LL}{\mathcal{L}}

\newcommand{\N}{\mathbb{N}}

\newcommand{\R}{\mathbb{R}}

\newcommand{\T}{\mathbb{T}}

\newcommand{\Z}{\mathbb{Z}}

\DeclareMathOperator{\grad}{\nabla}

\renewcommand{\epsilon}{\varepsilon}
\renewcommand{\setminus}{\smallsetminus}
\newcommand{\eps}{\epsilon}

\newcommand{\one}{\bm{1}}

\newcommand{\norm}[1]{\left\|#1\right\|}

\newcommand{\expt}[2][]{\mathbb{E}_{#1}\left[#2\right]}
\newcommand{\mB}[1]{\overline{B^{\eps}_{#1}}}
\newcommand{\oB}[1]{{{B}^{\eps}_{#1}}^{\prime}}
\newcommand{\mBcomp}[2]{\overline{B^{#2,\eps}_{#1}}}
\newcommand{\oBcomp}[2]{{{B}^{#2,\eps}_{#1}}^{\prime}}
\newcommand{\mA}[1]{\overline{A^{\eps}_{#1}}}
\newcommand{\oA}[1]{{{A}^{\eps}_{#1}}^{\prime}}
\newcommand{\mAcomp}[2]{\overline{A^{#2,\eps}_{#1}}}
\newcommand{\oAcomp}[2]{{{A}^{#2,\eps}_{#1}}^{\prime}}
\newtheorem{theorem}{Theorem}[section]
\newtheorem{definition}[theorem]{Definition}
\newtheorem{hypothesis}[theorem]{Hypothesis}
\newtheorem{corollary}[theorem]{Corollary}
\newtheorem{lemma}[theorem]{Lemma}
\newtheorem{proposition}[theorem]{Proposition}

\theoremstyle{remark}
\newtheorem{remark}[theorem]{Remark}

\numberwithin{equation}{section}
\newcommand{\sumkj}{\sum_{\substack{k\in \Z^{3,N}_0\\ j\in \{1,2\}}}}
\newcommand{\sumkmeanj}{\sum_{\substack{k\in \Z^{3,N}_0\\k_3=0,\quad j\in \{1,2\}}}}
\newcommand{\sumkmeanjuno}{\sum_{\substack{k\in \Z^{3,N}_0\\k_3=0,\quad j=1}}}
\newcommand{\sumkmeanjdue}{\sum_{\substack{k\in \Z^{3,N}_0\\k_3=0,\quad j=2}}}
\newcommand{\sumkmeancov}{\sum_{\substack{k\in \Z^{3,N}_0\\ k_3=0}}}
\newcommand{\sumkoscj}{\sum_{\substack{k\in \Z^{3,N}_0\\k_3\neq0,\quad j\in \{1,2\}}}}
\newcommand{\skj}{\sigma_{k,j}^{\eps}}
\newcommand{\skjcomp}[1]{\sigma_{k,j}^{#1,\eps}}
\newcommand{\smkj}{\sigma_{-k,j}^{\eps}}
\newcommand{\smkjcomp}[1]{\sigma_{-k,j}^{#1,\eps}}
\newcommand{\sk}[1]{\sigma_{k,#1}^{\eps}}
\newcommand{\skcomp}[2]{\sigma_{k,#2}^{#1,\eps}}
\newcommand{\smk}[1]{\sigma_{-k,#1}^{\eps}}
\newcommand{\smkcomp}[2]{\sigma_{-k,#2}^{#1,\eps}}
\newcommand{\tkj}{\theta_{k,j}^{\eps}}
\newcommand{\tkjcomp}[1]{\theta_{k,#1}^{\eps}}
\newcommand{\curl}{\operatorname{curl}}
\newcommand{\sumH}{\sum_{\substack{k_H\in \Z^2_0\\ N\leq \lvert k_H\rvert \leq 2N}}}
\newcommand{\ke}{\kappa_{\eps}}

\newenvironment{acknowledgements}{%
  \begin{abstract}
}{%
  \end{abstract}
}
\newenvironment{funding}{%
  \begin{abstract}
}{%
  \end{abstract}
}

\title[It\^o-Stratonovich Diffusion Limit for 3D Thin Domain]{On the It\^o-Stratonovich Diffusion Limit for the Magnetic Field in a 3D Thin Domain}
\author[F. Butori]{Federico Butori}
\address{Scuola Normale Superiore, Piazza dei Cavalieri, 7, 56126 Pisa, Italia}
\email{\href{mailto:federico.butori at sns.it}{federico.butori at sns.it}}
\author[F. Flandoli]{Franco Flandoli}
\address{Scuola Normale Superiore, Piazza dei Cavalieri, 7, 56126 Pisa, Italia}
\email{\href{mailto:franco.flandoli at sns.it}{franco.flandoli at sns.it}}
\author[E. Luongo]{Eliseo Luongo}
\address{Scuola Normale Superiore, Piazza dei Cavalieri, 7, 56126 Pisa, Italia}
\email{\href{mailto:eliseo.luongo at sns.it}{eliseo.luongo at sns.it}}
\date\today
\date\today
\allowdisplaybreaks
\begin{document}
\begin{abstract}
We introduce a stochastic model for a passive magnetic field in a three dimensional thin domain. The velocity field, white in time and modelling phenomenologically a turbulent fluid, acts on the magnetic field as a transport-stretching noise. We prove, in a quantitative way, that, in the simultaneous scaling limit of the thickness of the thin layer and the separation of scales, the mean on the thin direction of the magnetic field is close to the solution of the equation of the magnetic field with additional dissipation. In certain choice of noises with correlation between their components, without mirror symmetry and with a non zero mean helicity, we identify an alpha-term, in addition to the extra dissipation term. However, it does not produce dynamo; consequently, we extend a no-dynamo theorem to thin layers.
\end{abstract}

\maketitle

\section{Introduction}\label{sec introduction}
We consider a linear vector-valued SPDE in the unknown $B_{t}^{\varepsilon
}=B_{t}^{\varepsilon}\left(  x\right)  \in\mathbb{R}^{3}$, $x\in
\mathbb{T}_{\varepsilon}^{3}=\mathbb{T}^{2}\times\mathbb{R}/(2\varepsilon
\pi\mathbb{Z})$ (a model of thin domain), $t\in\left[  0,T\right]  $, of
advection type, loosely written as (see the rigorous formulation in \autoref{sec main result} below)
\begin{align}\label{intro app eq}
\begin{cases}
\partial_{t}B_{t}^{\varepsilon}+\nabla\times\left(  v_{t}^{\varepsilon}\times
B_{t}^{\varepsilon}\right)    & =\eta\Delta B_{t}^{\varepsilon}\\
\operatorname{div}B_{t}^{\varepsilon}  & =0\\
B_{t}^{\varepsilon}|_{t=0}  & =B_{0}^{\varepsilon}%
\end{cases}
\end{align}
where $v_{t}^{\varepsilon}=v_{t}^{\varepsilon}\left(  x\right)  \in
\mathbb{R}^{3}$ is a white noise in time, with suitable space covariance
structure, described in detail in \autoref{subsect description noise}. The multiplications
$v_{t}^{\varepsilon}\cdot\nabla B_{t}^{\varepsilon}$ and $B_{t}^{\varepsilon
}\cdot \nabla v_{t}^{\varepsilon}$ appearing in the development of the term
$\nabla\times\left(  v_{t}^{\varepsilon}\times B_{t}^{\varepsilon}\right)  $
are understood in Stratonovich sense; then the system is rewritten, as
explained in \autoref{sec main result}, as an It\^{o} system with the It\^{o}-Stratonovich corrector.

We prove that the vertical average $\mB{t}=M_{\varepsilon}B_{t}^{\varepsilon}$, $M_{\varepsilon}$ given by \autoref{mean fluct operator} below, converges in a suitable distributional topology to the solution
$\overline{B}_{t}$ of the system%
\begin{align}\label{intro limit eq}
\begin{cases}
\partial_{t}\overline{B}_{t} & =\left(  \eta+\eta_{T}\right)  \Delta\overline{B}_{t}+\nabla\times\left(  \mathcal{A}\overline{B}%
_{t}\right)  \\
\operatorname{div}\overline{B}_{t}  & =0\\
\overline{B}_{t}|_{t=0}  & =\overline{B}_{0}%
\end{cases}
\end{align}
where $\overline{B}_{0}$ is the limit of $M_{\varepsilon}B_{0}^{\varepsilon}$.
Here $\eta_{T}$ is an ``eddy magnetic diffusivity'', linked to the
space-covariance of $v_{t}^{\varepsilon}\left(  x\right)  $; the additional
term $\eta_{T}\Delta\overline{B}_{t}$ is often called the beta-term in the
literature devoted to equations for the magnetic field driven by a turbulent
velocity field (see \autoref{subsect physics} below). $\mathcal{A}$ is
a matrix; the term $\nabla\times\left(  \mathcal{A}\overline{B}_{t}\right)  $
is usually called the alpha-term. The convergence result is stated in detail
by \autoref{main Theorem} and the formulation given above is reconstructed in \autoref{appB}. 

We classify such a theorem as \textit{It\^{o}-Stratonovich stochastic
diffusion limit}, in analogy with the deterministic and stochastic theory of
diffusion limits based on homogenization techniques (see for instance
\cite{MajdaKra}, \cite{Komor}). The approach however is different, strongly
based on the It\^{o}-Stratonovich corrector, initiated by \cite{galeati2020convergence} (see also
\cite{FlaLuoWaseda} and references therein), from which we give this name.

\autoref{main Theorem}, in the particular geometry of the thin domain is new and
particularly challenging both from the viewpoint of the proof (we devote
\autoref{sec strategy proof} to explain the strategy and its difficulties) and the Physics
behind (see \autoref{subsect physics}). The closest previous result in
this direction is \cite{flandoli2023boussinesq}, a 2D-2C (two-dimensional, three component)
model, where a number of ideas used also here have been already introduced.
However, the thin domain treated here is a truly 3D-3C model, which reduces to
2D-2C only in the limit. The control of the limit is demanding, especially for
the noise part, the new one compared to the classical references on thin
domain equations. This is the first truly 3D problem where the technique of
It\^{o}-Stratonovich stochastic diffusion limit is developed. The main source
of difficulty is the stochastic stretching term $B_{t}^{\varepsilon}%
\cdot\nabla v_{t}^{\varepsilon}$; without it, as in \cite{FlLuo21}, the
control of the limit is easier. The 2D-3C problem of \cite{flandoli2023boussinesq}, being 2D
from the beginning, eliminates part of the difficulty. The treatment of the
stretching term in full generality, namely not in the particular case of thin
domains and under our assumption described in \autoref{subsect description noise}, is an open problem.\\

The content of the paper is the following one. In 
\autoref{subsect physics} we describe some ideas that motivated us to treat this
problem. In \autoref{sec functional analysis main results} we introduce rigorously our framework. In particular, we define the structure of the noise, $v^{\eps}_t$, in \autoref{subsect description noise}, while we state our main result in \autoref{sec main result} and describe the main ideas behind its proof in \autoref{sec strategy proof}. The remainder of the paper is dedicated to proving our main result. In \autoref{section a priori estimates} we prove estimates for $B^{\eps}_t$ uniformly in the thickness of the domain, then we show the convergence of $\mB{t}$ to $\overline{B}_t$ in \autoref{sec proof main thm} . We conclude the paper with two appendices. In \autoref{appendix ito strat corr} we provide an explicit representation of the It\^o-Stratonovich corrector in \eqref{intro app eq}. In \autoref{appB} we compute an average of the helicity of $v^{\eps}_t$, highlighting its connection to the alpha-term in \eqref{intro limit eq}.
\subsection{Comments on the Physical motivation\label{subsect physics}}
The model introduced above for $B_{t}^{\varepsilon}$ is a classical model for
a magnetic field driven by a velocity field $v_{t}^{\varepsilon}$, with some
degree of resistivity (namely closedness to ideal) given by $\frac{1}{\eta}$. We refer to the classical monographs \cite{moffatt1978field}, \cite{parker1955hydromagnetic} for a detailed discussion on the topics highlighted in this section.
This model is used in Plasma Physics to investigate several phenomena, like
reconnection and dynamo effect. We shall deal with reconnection in a future
work under the more specific geometry of toroidal plasma. Here let us comment
about the model and result under the usual view of dynamo research.

The step from the equation for $B_{t}^{\varepsilon}$ to a \textquotedblleft
mean field" equation like the one above for $\overline{B}_{t}$ is a classical
step in the Physics literature on this topic. The usual procedure consists in
taking averages of $B_{t}^{\varepsilon}$, getting a Reynolds type tensor which
mixes bilinearly the randomness of $B_{t}^{\varepsilon}$ and $v_{t}%
^{\varepsilon}$, and close the equation for the average of $B_{t}%
^{\varepsilon}$ by suitable closure assumptions on the average of the bilinear
term. Compared to these derivations, we do not aim to get new final mean field
equations, since those indicated in Plasma Physics literature are already deep
and compared with reality; our aim is to justify them. Indeed, their validity
is far from trivial, presumably false in full generality, for substantial
reasons and not just for a question of mathematical rigour. Let us explain why.

In our approach, the It\^{o}-Stratonovich corrector appearing when we rewrite
the term $\nabla\times\left(  v_{t}^{\varepsilon}\times B_{t}^{\varepsilon
}\right)  $ from the Stratonovich form to the It\^{o} one, corresponds
precisely to the mean field term (or sum of terms) discovered in the
literature under the most natural closure formulae. But in our rigorous
derivation there is still the It\^{o} term, which should disappear in the
limit when $\varepsilon\rightarrow0$. There is no obvious reason why it should
disappear, and in general it contains fluctuations which may be very important
and alter the limit result. In scalar and 2D models, where stretching of
vector quantities does not appear, it is possible to prove that the It\^{o}
term is negligible in the limit (see for instance \cite{galeati2020convergence}, \cite{FlaLuoWaseda}, \cite{flandoli2021scaling}, \cite{flandoli2021quantitative}, \cite{carigi2023dissipation}, \cite{flandoli20232d}, \cite{flandoli2023reduced}%
). In 3D, the problem is extremely difficult. We start to throw some light
with this work in a truly 3D model, thanks to the asymptotic simplifications
of the thin domain; previously some basic ingredients have been understood for
the 2D-3C model of \cite{flandoli2023boussinesq}.

Therefore, to summarize the previous discussion, our aim is making rigorous,
when true, the mean field theory of a magnetic field $B_{t}^{\varepsilon}$
advected by a turbulent velocity field $v_{t}^{\varepsilon}$. We have
identified here a framework where the mean field theory is correct.

Let us now come to some of the usual conclusions of the (previously heuristic)
mean field theory and let us compare our results with them. The final mean
field system for $\overline{B}_{t}$ contains two new terms, the alpha and beta
terms. Let us start from the beta term $\eta_{T}\Delta\overline{B}_{t}$. Let
us recall the following sentence from Krause e R\"{a}dler \cite[Page 12]{krause2016mean}:
\textquotedblleft homogeneous isotropic mirror symmetric turbulence only
influences the decay rate of the mean magnetic fields, which is enhanced in
almost all cases of physical interest\textquotedblright. When the field
$v_{t}^{\varepsilon}$ is mirror symmetric, the average elicity of
$v_{t}^{\varepsilon}$ (defined in \autoref{appB} below in the case of our white
noise field $v_{t}^{\varepsilon}$) is zero, which then implies that
$\mathcal{A}=0$, namely the alpha-term $\nabla\times\left(  \mathcal{A}%
\overline{B}_{t}\right)  $ is absent. In this case, the only mean-field effect
of the turbulent field $v_{t}^{\varepsilon}$ is to produce the beta-term
$\eta_{T}\Delta\overline{B}_{t}$, which enhances the decay rate, as stated by \cite[Page 12]{krause2016mean}. Therefore our result rigorously confirms that prediction, in
the particular case of thin domains. 

Concerning the other additional term, namely the alpha-term $\nabla
\times\left(  \mathcal{A}\overline{B}_{t}\right)  $, the story is more
complicated. In fully 3D cases such a term may be responsible for the growth
of the magnetic field by stretching and it is advocated as one of the sources
of dynamo, a very important observed phenomenon which is related to the
maintenance of the magnetic field in certain planets and stars. However, in
our case the term $\nabla\times\left(  \mathcal{A}\overline{B}_{t}\right)  $,
although present, cannot produce an increase of the magnetic field. This is
due to the specific algebraic structure, described in \autoref{sec main result}. The matrix
$\mathcal{A}$\ is degenerate and propagates only certain information of
$\overline{B}_{t}$ (those coming from the vertical part of the vector
potential) which do not create a loop suitable for increase (the vertical part
of the vector potential may only dissipate). Therefore, the conclusion of our
work is that in the thin domain geometry we may produce an alpha-term but not
a dynamo. This is coherent with conclusions of the Plasma Physics literature
which state that dynamo does not appear when the magnetic field depends only
on two variables, see for instance \cite[Section 3.5]{gilbert2003dynamo}. This happens only in the limit of the thin domain, not for
the approximating equation, but the final lack of dynamo remains true.
\section{Functional Setting and Main Results}\label{sec functional analysis main results}
In \autoref{notation sec} we introduce the notation we follow along the paper. In \autoref{preliminaries subsec} we introduce some functional analysis tools we employ in order to carry on rigorously our analysis described in \autoref{sec introduction}. We describe the noise we employ in \autoref{subsect description noise} providing some insights in the meaning of the assumptions and comparisons with the existing literature. We state our main result in \autoref{sec main result} and describe the main ideas behind its proof in \autoref{sec strategy proof} for the convenience of the reader. 
\subsection{Notation}\label{notation sec}
In the following we denote by $\mathbb{T}^2=\mathbb{R}^2/(2\pi\mathbb{Z})^2$ (resp. $\mathbb{T}^3=\mathbb{R}^2/(2\pi\mathbb{Z})^3$) the two (resp. three) dimensional torus. Moreover, for each $\eps>0$ let us denote by \begin{align*}
\T^3_\eps=\T^2\times\mathbb{R}/(2\eps\pi\mathbb{Z}) 
\end{align*} the thin layer three dimensional torus. In \autoref{preliminaries subsec} we focus on function spaces on periodic domains which we in general denote by $\mathcal{D}$ varying between $\T^2,\ \T^3,\ \T^3_{\eps}$. \\
Similarly we denote by $\Z^3_0=\Z^3\setminus \{0\}$ (resp. $\Z^2_0=\Z^2\setminus \{0\}$) the three (resp. two) dimensional lattice made by points with integer coordinates. Moreover, for each $\eps>0$ let us denote by \begin{align*}
    \Z^{3,\eps^{-1}}_0=\Z^2\times \eps^{-1}\Z \setminus\{0\}
\end{align*}
the thin layer three dimensional lattice. We introduce a partition of $\Z^{3,\eps^{-1}}_0=\Gamma^{\eps}_+\cup \Gamma^{\eps}_-$ such that 
\begin{align*}
    \Gamma^{\eps}_+=\{k\in \Z^{3,\eps^{-1}}_0: k_3>0 \vee (k_3=0 \land k_2>0) \vee (k_3=k_2=0 \land k_1>0)\},\quad \Gamma^{\eps}_{-}=-\Gamma^{\eps}_+.
\end{align*}\\ 
If $v=(v_1,v_2,v_3)^t:\mathcal{D}\rightarrow \R^3$ we will sometimes denote by $Dv=(\nabla v)^t$, $v_H=(v_1,v_2)^t$ in the following.\\
For each $j\geq 0,\ l>1$ we denote by 
\begin{align*}
    \zeta^N_{H,j}= N^{j-2}\sumH \frac{1}{\lvert k_H\rvert^j},\quad \zeta_{H,j}=\int_{1\leq\lvert x\rvert\leq 2}\frac{dx_1dx_2}{\lvert x\rvert^j}=2\pi \begin{cases}
        \log 2 & \text{ if } j=2,\\
        \frac{2^{2-j}-1}{2-j} & \text{ if } j\neq 2
    \end{cases} , \quad \zeta_{l}=\frac{1}{2}\sum_{k\in \Z\setminus\{0\} }\frac{1}{\lvert k\rvert^l}<+\infty.
\end{align*}
In particular, by the approximation properties of Riemann sums, for each $j\geq 0$ it holds
\begin{align*}
 \zeta_{H,j}^N= \zeta_{H,j}N^{2-j}+O(N^{1-j}).
\end{align*}
For two vector fields $X$ and $Y$ , we write $\mathcal{L}_X Y$ for the Lie derivative $\mathcal{L}_X Y=X\cdot\nabla Y -Y\cdot\nabla X.$ We denote by $\{e_1,e_2,e_3\}$ the canonical basis of $\R^3$.\\
We denote by $[M, N]$ the quadratic covariation process,
which is defined for any couple $M,\ N$ of square integrable real semimartingales and we extend it by bilinearity to the analogue complex valued processes.\\
Let $a,b$ be two positive numbers, then we write $a\lesssim b$ if there exists a positive constant $C$ such that $a \leq C b$ and $a\lesssim_\xi b$ when we want to highlight the dependence of the constant $C$ on a parameter $\xi$. If $x\in \C$ we denote by $x^*$ its complex conjugate. If $\mathcal{X}$ is a tensor we denote by $\norm{\mathcal{X}}_{HS}$ its Hilbert-Schmidt norm. Lastly, if $x\in \R^2$ we denote by $x^{\perp}=(-x_2,x_1)^t.$
\subsection{Preliminaries}\label{preliminaries subsec}
\subsubsection{Function Spaces in the Periodic Setting}\label{function spaces periodic}
Let us start setting some classical notation before describing the main contributions of this work. We refer to monographs \cite{ FlaLuoWaseda, Marchioro1994,pazy2012semigroups,temam1995navier,temam2001navier, trie1983fun,trie1995fun} for a complete discussion on the topics shortly recalled in this section.\\ Let $ \left(H^{s,p}(\mathcal{D}),\norm{\cdot}_{H^{s,p}(\mathcal{D})}\right),\ s\in\mathbb{R},\ p\in (1,+\infty)$ be the Bessel spaces of periodic functions. In case of $p=2$, we simply write $H^{s}(\mathcal{D})$ in place of $H^{s,2}(\mathcal{D})$ and we denote by $\langle \cdot,\cdot\rangle_{H^s(\mathcal{D})}$ the corresponding scalar products. When dealing with functions having $\T^2$ as a domain we neglect the domain in the definition of norm and inner product, i.e. we write $\norm{\cdot}_{H^{s,p}}$ in place of $\norm{\cdot}_{H^{s,p}(\T^2)}$.
In case of $s=0$, we write $L^2(\mathcal{D})$ instead of $H^0(\mathcal{D})$ and we neglect the subscript in the notation for the norm and the inner product when considering functions having $\T^2$ as a domain. Instead, in case of functions having $\T^3_\eps$ as a domain we write $\norm{\cdot}_{\eps}$ (resp. $\langle \cdot,\cdot\rangle_{\eps}$)  in place of $\norm{\cdot}_{L^2(\T^3_\eps)}$ (resp. $\langle \cdot,\cdot\rangle_{L^2(\T^3_\eps)}$). With some abuse of notation, for $s>0$, we denote the duality between $H^{-s}(\mathcal{D})$ and $H^s(\mathcal{D})$ by $\langle\cdot,\cdot\rangle_{L^2(\mathcal{D})}$.\\
We denote by $\Dot{H}^{s,p}(\mathcal{D})$ (resp. $\Dot{H}^{s}(\mathcal{D}),\ \Dot{L^2}(\mathcal{D})$) the closed subspace of $H^{s,p}(\mathcal{D})$ (resp. $H^{s}(\mathcal{D}),\ L^2(\mathcal{D})$) made by zero mean functions with norm inducted by $H^{s,p}(\mathcal{D})$ (resp. $H^{s}(\mathcal{D}),\ L^2(\mathcal{D})$).
With some abuse of notation we denote the Laplacian with the same symbol, $\Delta$, either when considering functions with domain $\T^2, \T^3$ or $\T^3_{\eps}$, i.e. we  denote by \begin{align*}
    \Delta: \mathsf{D}(\Delta)\subseteq \Dot{L^2}(\mathcal{D})\rightarrow \Dot{L}^2(\mathcal{D}),
\end{align*}
where $\mathsf{D}(\Delta)=\Dot{H}^2(\mathcal{D})$. We follow the same notation when considering the orthogonal projection of $L^2(\mathcal{D})$ on $\Dot{L}^2(\mathcal{D})$, which we denote by $Q:L^2(\mathcal{D})\rightarrow \Dot{L}^2(\mathcal{D})$. The restriction (resp. extension) of $Q$ to $H^s(\mathcal{D})$ (resp. $H^{-s}(\mathcal{D})$) for $s>0$ is the orthogonal projection on $\Dot H^s(\mathcal{D})$ (resp. $\Dot H^{-s}(\mathcal{D})$). It is well known that ${\Delta}$ is the infinitesimal generator of analytic semigroup of negative type that we denote by $e^{t\Delta}:\Dot L^2(\mathcal{D})\rightarrow \Dot L^2(\mathcal{D})$ and moreover for each $\alpha\in \R$, $\mathsf{D}((-\Delta)^{\alpha})$ can be identified with $\Dot H^{2\alpha}(\mathcal{D})$. It is well known that the family $\{\frac{1}{2\pi\sqrt{2\pi\eps}}e^{ik\cdot x}\}_{k\in \Z^{3,\eps^{-1}}_0}$ (resp. $\{\frac{e^{ik\cdot x}}{2\pi}\}_{k\in \Z^2_0}$) is a complete orthogonal systems of $\Dot{H}^s(\T^3_{\eps})$ (resp. $\Dot{H}^s(\T^2)$) made by eigengunctions of $-\Delta$. Moreover it is orthonormal in $\Dot L^2(\T^3_{\eps})$ (resp. $\Dot L^2(\T^2)$). We recall for the convenience of the reader some properties of the Heat Semigroup we will exploit in the following: \begin{lemma}\label{Properties semigroup}
Let $q\in \Dot{H}^{\alpha}(\mathcal{D}),\ \alpha\in \mathbb{R}$. Then:
\begin{enumerate}[label=\roman*)]
    \item for any $\varphi\geq 0,$ it holds $\lVert e^{t {\Delta}}q\rVert_{\Dot{H}^{\alpha+\varphi}}\leq C_{\varphi}t^{-\varphi/2}\lVert q\rVert_{\Dot{H}^{\alpha}}$ for some constant increasing in $\varphi$;
    \item for any $\varphi\in [0,2],$ it holds $\lVert \left( I-e^{t{\Delta}}\right)q\rVert_{\Dot{H}^{\alpha-\varphi}}\lesssim_{\varphi} t^{\varphi/2}\lVert q\rVert_{\Dot{H}^{\alpha}}$.
\end{enumerate}
\end{lemma}

Similarly, for $d\in\{2,3\}$ we introduce the Bessel spaces of zero mean vector fields
\begin{align*}
    {H}^{s,p}(\mathcal{D};\mathbb{R}^d)&= \{u=(u_1,\dots, u_d)^t:\ u_1,\dots, u_d\in H^{s,p}(\mathcal{D})\},\\  \langle u, v\rangle_{{H}^s(\mathcal{D};\R^{d})}&=\sum_{j=1}^d\langle u_j,v_j\rangle_{H^s}, \quad \text{for } s\in \mathbb{R}.
\end{align*}
Again, in case of $s=0$ we write ${L}^2(\mathcal{D};\R^d)$ instead of ${H}^0(\mathcal{D};\R^d)$ and we neglect the subscript in the notation for the norm and the scalar product. Similarly we denote by $\Dot H^s(\mathcal{D};\R^d)$ (resp. $\Dot L^2(\mathcal{D};\R^d)$) the closed subspace of $H^s(\mathcal{D};\R^d)$ (resp. $L^2(\mathcal{D};\R^d)$) made by zero mean functions. In case of $\mathcal{D}=\T^{3}_{\eps},\ d=3$ we denote by $\mathbf{H}_{\eps}^{s}$ (resp. $\mathbf{L}^2_{\eps}$) the closed subspace of $H^s(\T^3_{\eps};\R^3)$ (resp. $L^2(\T^3_{\eps};\R^3)$) made by zero mean, divergence free vector fields with norm induces by $H^s(\T^3_{\eps};\R^3)$ (resp. $L^2(\T^3_{\eps};\R^3)$). 
With some abuse of notation we usually employ the same notation for norms and inner products that we introduced for scalar functions also for the corresponding vector fields, forgetting their range and we denote by \begin{align*}
    (-\Delta)^{\alpha} f=((-\Delta)^{\alpha} f_1,\dots, (-\Delta)^{\alpha} f_i)^t\quad \text{for }f\in \Dot{H}^{2\alpha}(\mathcal{D};{\R}^d).
\end{align*}

Having this notation in mind we recall the definition and the regularizing property of the Leray Projection and the Biot-Savart operator. We denote by \begin{align*}
    P_{\eps}:\Dot L^2(\T^3_{\eps};\R^3)\rightarrow \mathbf{L}^2_{\eps}
\end{align*} 
the Leray projection which, in the periodic setting, is an orthogonal projection even when acting between $\Dot H^s(\T^{3}_{\eps};\R^3)$ and $\mathbf{H}^s_{\eps}.$\\
If $v\in \mathbf{H}_{\eps}^{s}$ there exists a unique vector field $K_{\eps}[v]$ such that $\operatorname{curl}K_{\eps}[v]=v,\ \operatorname{div}K_{\eps}[v]=0$. Moreover the operator \begin{align*}
    K_{\eps}:\mathbf{H}_{\eps}^{s}\rightarrow \mathbf{H}_{\eps}^{s+1}
\end{align*} is linear and continuous for each $s\in \R$ and is given by $K_{\eps}[v]=(-\Delta)^{-1}\operatorname{curl}v$. In particular, if $v:\T^3_{\eps}\rightarrow\R^3,\ v=(v_1,v_2,v_3)^t,$ is independent from the third component, we have
\begin{align}\label{Biot savart independent from third component}
  K_{\eps}[v]=(-\nabla_H^{\perp}(-\Delta)^{-1}v^3,(-\Delta)^{-1}\nabla^{\perp}_H\cdot v^H)^t,  
\end{align}
where we denoted by $\nabla^{\perp}_H=(-\partial_2,\partial_1)$.

We conclude this subsection introducing some classical notation when dealing with stochastic processes taking values in separable Hilbert spaces. Let $Z$ be a separable Hilbert space, with associated norm $\| \cdot\|_{Z}$. We denote by $C_{\mathcal{F}}\left(  \left[  0,T\right]  ;Z\right) $ the space of weakly continuous adapted processes $\left(  X_{t}\right)  _{t\in\left[
0,T\right]  }$ with values in $Z$ such that
\[
\mathbb{E} \bigg[ \sup_{t\in\left[  0,T\right]  }\left\Vert X_{t}\right\Vert
_{Z}^{2}\bigg]  <\infty
\]
and by $L_{\mathcal{F}}^{p}\left(  0,T;Z\right),\ p\in [1,\infty),$ the space of progressively
measurable processes $\left(  X_{t}\right)  _{t\in\left[  0,T\right]  }$ with
values in $Z$ such that
\[
    \mathbb{E} \bigg[ \int_{0}^{T}\left\Vert X_{t}\right\Vert _{Z}^{p}dt \bigg]
<\infty.
\]
\subsubsection{Preliminaries on the Thin Layer}
As discussed in \autoref{sec introduction}, it is convenient to introduce the mean operator, $M_{\eps}$, and the fluctuations operator, $N_{\eps}$, in order to exploit the properties of working with a thin layer. Therefore, let us recall their definitions and some properties which we will use in the following. We refer to \cite{RaugelII,Raugel1, TemamZianeThin3d} for the proof of these results and some more discussions on the topic.
\begin{definition}\label{mean fluct operator}
We denote by $M_{\epsilon}:\Dot L^2(\T^3_{\eps})\rightarrow \Dot L^2(\T^3_{\eps})$ the mean operator in the third variable, i.e. \begin{align*}
        M_{\epsilon}\phi(x_1,x_2)=\frac{1}{2\pi\epsilon}\int_0^{2\pi\epsilon} \phi(x_1,x_2,x_3)dx_3.
    \end{align*}
    Similarly $N_{\epsilon}=I-M_{\epsilon}$, therefore $\int_0^{2\pi\epsilon} N_{\epsilon}\phi(x_1,x_2,x_3)dx_3=0.$   
\end{definition}
The operator $M_{\epsilon}$ and $N_{\epsilon}$ are projection either in $\Dot L^2(\T^3_\epsilon)$ and in $\Dot H^1(\T^3_\epsilon)$. More in general the following holds.
\begin{lemma}\label{properties of M N}$ $\\
\begin{enumerate}[label=\roman*)]
    \item The operators $M_{\eps},\ N_{\eps}$ are orthogonal projections in $\Dot H^s(\T^3_{\eps})$ for each $s\in \R.$ In particular:
    \begin{align}\label{property 1}
M_{\epsilon}N_{\epsilon}f=N_{\eps}M_{\eps}f=0\quad \forall f\in \Dot H^s(\T^3_{\eps}).
\end{align}
\item $M_{\eps},\ N_{\eps}$ commute with derivative and fractional powers of the Laplacian. In particular, $M_{\eps},\ N_{\eps}$ commute with $K_{\eps}$. 
\item For each $u,\ v,\ w:\T^3_{\eps}\rightarrow \R^3$ such that $\operatorname{div}u=\operatorname{div}v=\operatorname{div}w=0$ and $\langle u\cdot\nabla v,w\rangle_{\eps}$ is well defined, than we have
\begin{align}\label{property 3}
\langle u\cdot\nabla v,M_{\eps}w\rangle_{\eps}=\langle M_{\eps}u\cdot\nabla M_{\eps}v,M_{\eps}w\rangle_{\eps}+\langle N_{\eps}u\cdot\nabla N_{\eps}v,M_{\eps}w\rangle_{\eps}.    \end{align}
\begin{align}\label{property 4}
\langle u\cdot\nabla v,N_{\eps}w\rangle_{\eps}=\langle M_{\eps}u\cdot\nabla N_{\eps}v,N_{\eps}w\rangle_{\eps}+\langle N_{\eps}u\cdot\nabla M_{\eps}v,N_{\eps}w\rangle_{\eps}+\langle N_{\eps}u\cdot\nabla N_{\eps}v,N_{\eps}w\rangle_{\eps}. \end{align}
\item  For each $f\in \Dot H^{1}(\T^3_{\eps})$ the following Poincaré inequality holds \begin{align}\label{poincare thin layer}
        \norm{N_{\epsilon}v}_{{\epsilon}}\leq \epsilon \norm{\partial_{3}N_{\epsilon}v}_{{\epsilon}}.
    \end{align}
\end{enumerate}
\begin{remark}\label{remark fourier decomposition and splitting products}
In our periodic setting, the mean operator $M_{\eps}$ preserves Fourier modes which are independent from the third variable. Instead, the fluctuations operator $N_{\eps}$ preserves Fourier modes which depend from the third variable, i.e. if $f\in \Dot H^s({\T^{3}_\eps})$ then 
\begin{align*}
    f=\sum_{k\in \Z^{3,\eps^{-1}}_0}\langle f, e^{-ik\cdot x}\rangle_{\eps}\frac{e^{ik\cdot x}}{2\pi\eps}\end{align*}
    and \begin{align*}
    M_{\eps}f=\sum_{\substack{{k\in \Z^{3,\eps^{-1}}_0}\\ k_3=0}}\langle f, e^{-ik\cdot x}\rangle_{\eps}\frac{e^{ik\cdot x}}{8\pi\eps}\quad N_{\eps}f=\sum_{\substack{k\in \Z^{3,\eps^{-1}}_0\\ k_3\neq 0}}\langle f, e^{-ik\cdot x}\rangle_{\eps}\frac{e^{ik\cdot x}}{8\pi\eps}.
\end{align*}
Exploiting Fourier decomposition, one can easily prove that, more in general, if $u,\ v,\ w:\T^3_{\eps}\rightarrow \R$ are regular enough such that $uv\in  H^s(\T^3),\ s\in \R$ is well defined (e.g. $s=0,\ u,v\in \Dot L^4(\T^3_{\eps})$), then we have
\begin{align*}
    M_{\eps}(Q[uv])=Q[M_{\eps}u M_{\eps}v]+M_{\eps}\left(Q[N_{\eps}uN_{\eps}v]\right),\quad N_{\eps}(Q[uv])=M_{\eps}{u}N_{\eps}v+N_{\eps}uM_{\eps}v+N_{\eps}(Q[N_{\eps}uN_{\eps}v]).
\end{align*}
\end{remark}
\end{lemma}
\subsection{Description of the Stochastic System}\label{subsect description noise}
We move now to the description of the noise. We shall define it in a sort of expansion in the Fourier basis. For each $ k\in \Z^{3,\eps^{-1}}_0,\ j\in\{1,2\}$ we denote by $\skj(x)=\tkj  a_{k,j}e^{ik\cdot x}$, where $\{\frac{k}{\lvert k\rvert}, a_{k,1}, a_{k,2}\}$ is an orthonormal system of $\R^3$ for $k_3> 0$ and $a_{k,j}=a_{-k,j}$ if $k_3<0$, while the (complex) coefficients $\theta_{k,j}^\eps$ will be defined later. 
It is well known that the family $\{a_{k,j}\frac{1}{2\pi\sqrt{2\pi\eps}}e^{ik\cdot x}\}, \ {{k\in \Z^{3,\eps^{-1}}_0, \ j\in\{1,2\}}}$ is a complete orthogonal systems of $\mathbf{H}^s_{\eps}$ made by eigenfunctions of $-\Delta$. Moreover it is orthonormal in $\mathbf{L}^2_{\eps}$. 
Without losing of generality, we make a choice that will simplify some of the computations: when $k_3=0$ we choose $a_{k,1}=(-k_2/\lvert k\rvert,k_1/\lvert k\rvert,0),\ a_{k,2}=(0,0,1)$ if $k\in \Gamma_{+}^{\eps}$, $a_{k,j}=a_{-k,j}$ otherwise. 
Now we define the Brownian motions. Our construction aims at obtaining a non trivial space covariance structure, while retaining the description in the Fourier basis, since this choice will simplify many computations thanks to \autoref{remark fourier decomposition and splitting products}. The physical intuition behind our construction is discussed \autoref{rmk: noise motivation}. Begin with a correlation factor $\rho\in [-1,1]$. Then assume  $\left(\Omega,\mathcal{F},\mathcal{F}_t,\mathbb{P}\right)$ is a filtered probability space such that $(\Omega, \mathcal{F},\mathbb{P})$ is a complete probability space, $(\mathcal{F}_t)_{t\in [0,T]}$ is a right continuous filtration and $\mathcal{F}_0$ contains every $\mathbb{P}$ null subset of $\Omega$.
Now, for each $N\in \N$, $\{W_t^{k,j}\}, \ {{k\in \Z^{3,N}_0, \  j\in \{1,2\}}}$ is a sequence of complex-valued Brownian motions adapted to $\mathcal{F}_t$ 
    such that
$W^{-k,j}_t=(W^{k,j}_t)^*$ and 
\begin{align*}
    \expt{W^{k,j}_1,{W^{l,m}_1}^*}&=\begin{cases}
        2 & \textit{if } k=l, \ m=j\\
        2\rho &  \textit{if } k=l,\ k_3=0,\ m\neq j;
    \end{cases}
    \quad \left[W^{k,j}_\cdot, W^{l,m}_\cdot\right]_t=\begin{cases}
        2t & \textit{if } k=-l, \ m=j\\
        2\rho t &  \textit{if } k=-l,\ k_3=0,\ m\neq j.    
        \end{cases}
\end{align*}
Then we define \begin{align*}
    W^{\eps}_t=\sum_{\substack{k\in \Z^{3,N}_0\\ j\in \{1,2\}}}\skj W^{k,j}_t,
\end{align*} with $\skj=\tkj a_{k,j}e^{ik\cdot x}$. The role of the factor $\rho$ is to introduce some correlation between the vertical and horizontal components of the `2D' modes (those who have $k_3=0$) of the random field. 
With our choices, to ensure that $W^\eps$ takes values in a space of real valued functions, the coefficients must satisfy the condition \begin{equation}\label{symmetry of tks}
    {\tkj}^{*}=\theta_{-k,j}^\eps.\end{equation}
\begin{remark}
The complex Brownian motions $W^{k,j}_t$ required can be constructed in the following way: let $\{B_t^{k,j}\}_{\substack{k\in \Z^{3,N}_0\\ j\in \{1,2\}}}$ be a sequence of real, Brownian motions adapted to $\mathcal{F}_t$ such that
\begin{align*}
        \expt{B^{k,j}_t B^{l,m}_s}=\left(t\wedge s\right)\left(\delta_{j,m}\delta_{k,l}+\rho\delta_{\lvert j-m\rvert=1}\delta_{k,l}\delta_{k_3,0}\right)
    \end{align*}and then set
    \begin{align*}
        W^{k,j}_t=\begin{cases}
            B^{k,j}_t+iB^{-k,j}_t& \text{if } k\in \Gamma^{\eps}_+ \\
            B^{-k,j}_t-iB^{k,j}_t & \text{if } k\in \Gamma^{\eps}_-.
        \end{cases}
    \end{align*}
\end{remark}
Finally, we make an explicit choice of the coefficients $\tkj$:
\begin{align}\label{definition thetakj}
    \tkj &=\one_{\{N\leq \lvert k_H\rvert\leq 2N\} }\begin{cases}
        \frac{i}{C_{1,H}\lvert k\rvert}& \text{if } k_3=0, k\in\Gamma_{+}^{\eps},\ j=1\\
         \frac{-i}{C_{1,H}\lvert k\rvert}& \text{if } k_3=0, k\in\Gamma_{-}^{\eps},\ j=1\\
        \frac{1}{C_{2,H}\lvert k\rvert^{\gamma/2}}& \text{if } k_3=0,\ j=2\\
        \frac{1}{C_{V}\lvert k\rvert^{\beta/2}}& \text{if } k_3\neq 0.
    \end{cases}
\end{align}
 The parameters $\eps, N, C_{1,H}, C_{2,H}, C_{V},\beta,\gamma$ cannot be arbitrary. Indeed we will always assume in the following even if not specified:
\begin{hypothesis}\label{HP noise}$ $\\
\begin{itemize}
 \item The width of the thin layer is $2\pi\eps=\frac{2\pi}{N}$, i.e. $\eps=\frac{1}{N}$ for $N\in\N$.
     \item $\beta\geq 4,\ \gamma\geq 4.$ 
    \item $C_{1,H}>\sqrt{\frac{\zeta_{H,0}}{\eta}},\ C_{2,H},\ C_V>0$.
\end{itemize}
\end{hypothesis}
More general choices for the phase and the sign of our coefficients are allowed, although the computations and the notation get heavier, while the decay rate ($\beta, \gamma \ge 4$) is optimal in our setting. For this reason we decide to avoid this level of generality.
\begin{remark}\label{remark coefficients }
    Due to our choice of the coefficients $\tkj$ and the vectors $a_{k,j}$ it follows that for each $k$ such that $k_3=0,\ N\leq \lvert k\rvert \leq 2N$
    \begin{align*}
        \tkjcomp{1} a_{k,1}=\frac{i}{C_{1,H}}\frac{k^{\perp}}{\lvert k\rvert^2},\ \tkjcomp{2} a_{k,2}=\frac{1}{C_{2,H}}\frac{e_3}{\lvert k\rvert^{\gamma/2}}.
    \end{align*}
    From which it follows that the 2D modes of our noise naturally split in horizontal and vertical components
\begin{align*}
    \overline{W^{\eps,H}}:= \sum_{\substack{k\in \Z_0^{3, N} \\ {k_3=0, j=1} }}\sigma_{k,j}^{\eps}W_k^j \qquad \overline{W^{\eps,3}}:= \sum_{\substack{k\in \Z_0^{3, N} \\ {k_3=0, j=2} }}\sigma_{k,j}^{\eps}W_k^j
\end{align*}
Moreover the components satisfy the relations 
\begin{align*}
    \overline{W^{\eps,H}}(-x)=-\overline{W^{\eps,H}}(x), \qquad \overline{W^{\eps, 3}}(-x)=\overline{W^{\eps, 3}}(x).
\end{align*}
As shown in \cite{flandoli2023boussinesq}, this anisotropy of the noise is the main responsible for the appearance of the alpha-term in this kind of models. 
\end{remark}
The space covariance function associated to our noise is \begin{align}\label{def covariance}
Q^{\eps}(x,y)=\expt{W^{\eps}_1(x) \otimes {W^{\eps}_1(y)}^* }=Q^{\eps}_0(x,y)+\overline{Q^{\eps}_{\rho}}(x,y)=\overline{Q^{\eps}}(x,y)+\left(Q^{\eps}\right)'(x,y)+\overline{Q^{\eps}_{\rho}}(x,y)
\end{align} 
where we denoted by \begin{align*}
\overline{Q^{\eps}}(x,y)&=2\sumkmeanj \left\lvert \tkj\right\rvert^2 a_{k,j}\otimes a_{k,j}e^{ik\cdot(x-y)},\quad \left(Q^{\eps}\right)'(x,y)=2\sumkoscj \left\lvert \tkj\right\rvert^2 a_{k,j}\otimes a_{k,j}e^{ik\cdot(x-y)},\\
\overline{Q^{\eps}_{\rho}}(x,y)&=2\rho \sumkmeancov  \left(\tkjcomp{1}\left(\tkjcomp{2}\right)^* a_{k,1}\otimes a_{k,2}+\tkjcomp{2}\left(\tkjcomp{1}\right)^* a_{k,2}\otimes a_{k,1}\right)e^{ik\cdot(x-y)}.   
\end{align*}
Due to our choice, we have that $\overline{Q^{\eps}}(x,y),$ $\left(Q^{\eps}\right)'(x,y)$ and $\overline{Q^{\eps}_{\rho}}(x,y)$ are all translation invariant, therefore we simply write $Q^{\eps}(x-y)$ (resp. $Q^{\eps}_0(x-y),\ \overline{Q^{\eps}}(x-y),\ \left(Q^{\eps}\right)'(x-y),\ \overline{Q^{\eps}_{\rho}}(x-y)$) in place of  $Q^{\eps}(x,y)$ (resp. $Q^{\eps}_0(x,y),\ \overline{Q^{\eps}}(x,y),\ \left(Q^{\eps}\right)'(x,y),\ \overline{Q^{\eps}_{\rho}}(x,y)$). In particular $Q^{\eps}_0(x),\ \overline{Q^{\eps}}(x),\ \left(Q^{\eps}\right)'(x)$ are mirror symmetric, namely it holds 
\begin{align}\label{mirror symmetric property 1}
Q^{\eps}_0(x)=Q^{\eps}_0(-x),\ \overline{Q^{\eps}}(x)=\overline{Q^{\eps}}(-x),\ \left(Q^{\eps}\right)'(x)=\left(Q^{\eps}\right)'(-x).  
\end{align}
On the contrary $\overline{Q^{\eps}_{\rho}}(x)$ is an odd function, in particular $\overline{Q^{\eps}_{\rho}}(0)=0$. As a consequence, our noise satisfies the mirror symmetry property if and only if $\rho=0$ .
Let us understand better the behavior of $Q^{\eps}(0)$. Indeed, as discussed in \cite[Chapter 3]{FlaLuoWaseda}, it plays a crucial in order to provide the dissipation properties of our limit object. Let us start considering $\left(Q^{\eps}\right)'(0)$, we have:
\begin{align*}
\left(Q^{\eps}\right)'(0)=2\sum_{\substack{k\in \Z^{3,N}_0\\ N\leq\lvert k_H\rvert\leq 2N}}\frac{1}{C_V^2\lvert k\rvert^{\beta}}\left(I-\frac{k\otimes k}{\lvert k\rvert^2}\right). 
\end{align*}
Arguing as in \cite[Section 2.3]{galeati2020convergence}, it follows easily that $\left(Q^{\eps}\right)'(0)$ is a diagonal matrix, moreover it holds
\begin{align*}
\left(Q^{\eps}\right)'(0)=2(e_1\otimes e_1+e_2\otimes e_2)\eta^{\eps}_{V,T}+2(e_3\otimes e_3)\eta_{V,R}^\eps,
\end{align*}
where we denoted by 
\begin{align*}
\eta^{\eps}_{V,T}=\sum_{\substack{k\in \Z^{3,N}_0,\ k_3\neq 0\\ N\leq\lvert k_H\rvert\leq 2N}}\frac{1}{C_V^2\lvert k\rvert^{\beta}}\left(1-\frac{\lvert k_H\rvert^2}{2\lvert k\rvert^2}\right),\quad \eta_{V,R}^\eps=\sum_{\substack{k\in \Z^{3,N}_0,\ k_3\neq 0\\ N\leq\lvert k_H\rvert\leq 2N}}\frac{1}{C_V^2\lvert k\rvert^{\beta}}\left(1-\frac{\lvert k_3\rvert^2}{\lvert k\rvert^2}\right).     
\end{align*}
In particular both $\eta^{\eps}_{V,T}, \ \eta^{\eps}_{V,R}$ go to $0$ as $\eps\rightarrow 0$. Indeed it holds
\begin{align}\label{asymptotic coeffienct vertical noise}
\eta^{\eps}_{V,T}+\eta^{\eps}_{V,R}&\lesssim \sum_{\substack{k\in \Z^{3,N}_0, \ k_3\neq 0\\ N\leq\lvert k_H\rvert\leq 2N}}\frac{1}{\left(\lvert k_H\rvert^{2}+\lvert k_3\rvert^2\right)^{\beta/2}}\notag\\ &\leq \sumH \sum_{k_3\in \Z,\ k_3\neq 0} \frac{1}{\left(1+\lvert k_3\rvert^2\right)^{\beta/2}N^{\beta}}\notag\\ & \leq \zeta^N_{H,0}\zeta_{\beta}\frac{1}{N^{\beta-2}}=O(N^{2-\beta}).   
\end{align}
In order to treat $\overline{Q^{\eps}}(0)$ we introduce a further splitting of the sum. We write
\begin{align*}
\overline{Q^{\eps}}(x)=\overline{Q_T^{\eps}}(x)+\overline{Q_R^{\eps}}(x),    
\end{align*}
denoting by 
\begin{align*}
\overline{Q^{\eps}_T}(x)=2\sumkmeanjuno \left\lvert\tkj\right\rvert^2 a_{k,j}\otimes a_{k,j}e^{ik\cdot x},\quad \overline{Q^{\eps}_R}(x)=2\sumkmeanjdue \left\lvert\tkj\right\rvert^2 a_{k,j}\otimes a_{k,j}e^{ik\cdot x}.    
\end{align*}
Of course, $\overline{Q^{\eps}_T}(x), \overline{Q^{\eps}_R}(x)$ satisfy mirror symmetry property too. By definition it holds
\begin{align*}
\overline{Q^{\eps}_R}(0)=2\left(e_3\otimes e_3\right) \eta^{\eps}_{H,R},\quad \eta^{\eps}_{H,R}=\sumH \frac{1}{C_{2,H}^2\lvert k_H\rvert^{\gamma}}.
\end{align*}
Therefore also $\eta^{\eps}_{H,R}$ goes to $0$ as $\eps \rightarrow 0$. Indeed, we have 
\begin{align}\label{asymptotic coeffienct horizontal noise 1}
\eta^{\eps}_{H,R}=\frac{\zeta^N_{H,\gamma}N^{2-\gamma}}{C_{2,H}^2}=O(N^{2-\gamma}).    
\end{align}
Lastly we need to study the asymptotic behavior of $\overline{Q^{\eps}_T}(0)$ which is the one we will keep track in the scaling limit. Arguing as in \cite[Section 2.3]{galeati2020convergence}, it follows easily that also $\overline{Q^{\eps}_T}(0)$ is a diagonal matrix, moreover it holds
\begin{align*}
    \overline{Q^{\eps}_T}(0)=\frac{\zeta_{H,2}^N}{C_{1,H}^2}(e_1\otimes e_1+e_2\otimes e_2)=\frac{\zeta_{H,2}}{C_{1,H}^2}(e_1\otimes e_1+e_2\otimes e_2)+O(N^{-1}).
\end{align*}
We set \begin{align*}
\eta^{\eps}_T=\eta^{\eps}_{V,T}+\frac{\zeta^N_{H,2}}{C_{1,H}^2}=\frac{\zeta_{H,2}}{C_{1,H}^2}+O(N^{-1}), \quad \eta^{\eps}_R=\eta^{\eps}_{V,R}+\eta^{\eps}_{H,R}=O(N^{-4}),\quad \eta_T=\frac{\zeta_{H,2}}{C_{1,H}^2}.
\end{align*}

We state now some results on the covariance functions $Q^{\eps}(x),$ $Q^{\eps}_0(x),$ $\overline{Q_{\rho}^{\eps}}(x),$ $\overline{Q^{\eps}}(x),$ $\left(Q^{\eps}\right)'(x),$ $\overline{Q_T^{\eps}}(x),$ $\overline{Q_R^{\eps}}(x)$. The first ones are strongly related to the invariance by translation of the covariance functions, therefore holds for all of them. We give the detailed statement and proof only for $Q^{\eps}_0(x)$, the others being analogous.
\begin{lemma}\label{lemma properties covariance operator 1}
For each $l,m\in \{1,2,3\}$ it holds
\begin{align}\label{property derivatives 1}
    \left(\partial_{l} Q^{\eps}_0\right)(x-y)&=\partial_{x_l} Q^{\eps}_0(x-y)=2i\sumkj \left\lvert\tkj\right\rvert^2 k_l a_{k,j}\otimes a_{k,j}e^{ik\cdot(x-y)},\\
    \label{property derivatives 2}
    \left(\partial_{l}\partial_m Q^{\eps}_0\right)(x-y)&=-\partial_{x_l}\partial_{y_m}Q^{\eps}_0(x-y)=-2\sumkj \left\lvert\tkj\right\rvert^2 k_l k_m a_{k,j}\otimes a_{k,j}e^{ik\cdot(x-y)}.
\end{align}
In particular
\begin{align}\label{property derivatives 3}
    \left(\partial_{l} Q^{\eps}_0\right)(0)&=2i\sumkj \left\lvert\tkj\right\rvert^2 k_l a_{k,j}\otimes a_{k,j},\\
    \label{property derivatives 4}
    \left(\partial_{l}\partial_m Q^{\eps}_0\right)(0)&=-2\sumkj \left\lvert\tkj\right\rvert^2 k_l k_m a_{k,j}\otimes a_{k,j}.
\end{align}
\end{lemma}
\begin{proof}
$\left(\partial_{l} Q^{\eps}_0\right)(x-y)=\partial_{x_l} Q^{\eps}_0(x-y)$ and $\left(\partial_{l}\partial_m Q^{\eps}_0\right)(x-y)=-\partial_{x_l}\partial_{y_m}Q^{\eps}_0(x-y)$ by the chain rule. Then \eqref{property derivatives 1} and \eqref{property derivatives 2} follow from the definition of $Q^{\eps}_0(x-y).$
Setting $y=x$ in the right hand side of \eqref{property derivatives 1} (resp. \eqref{property derivatives 2}) we get immediately
\begin{align*}
     \left(\partial_{l} Q^{\eps}_0\right)(0)&=2i\sumkj \left\lvert\tkj\right\rvert^2 k_l a_{k,j}\otimes a_{k,j}\\ \text{(resp. } \left(\partial_{l}\partial_m Q^{\eps}_0\right)(0)&=-2\sumkj \left\lvert\tkj\right\rvert^2 k_l k_m a_{k,j}\otimes a_{k,j}\text{)}.
\end{align*}
\end{proof}
The other results we need on the covariance functions are strongly related to the mirror symmetry properties, therefore they hold true only for $Q^{\eps}_0(x), \overline{Q^{\eps}}(x),\ \left(Q^{\eps}\right)'(x),\ \overline{Q_T^{\eps}}(x),\ \overline{Q_R^{\eps}}(x)$. We give the detailed statement and proof only for $Q^{\eps}_0(x)$, the others being analogous.
\begin{lemma}\label{lemma covariance operator 2}
It holds
\begin{align}\label{property derivatives 5}
    \left(\partial_{l} Q^{\eps}_0\right)(0)&=2i\sumkj \left\lvert\tkj\right\rvert^2 k_l a_{k,j}\otimes a_{k,j}=0.
\end{align}
As a consequence, for each $l,m,n\in \{1,2,3\}$ and $f\in L^2(\T^3_{\eps}),\ g\in H^1(\T^3_{\eps})$ it holds
\begin{align}\label{mirror symmetry product 0}
    \sumkj \langle\partial_l g\skjcomp{l}, f\partial_m \smkjcomp{n}\rangle_{\eps}=0.
\end{align}
\end{lemma}
\begin{proof}
Since by the mirror symmetry property of the covariance function \eqref{mirror symmetric property 1}, the function $Q^{\eps}_0(x)$ is even and smooth, then $\left(\partial_{l} Q^{\eps}_0\right)(0)=0$ and \eqref{property derivatives 5} follows. In order to complete the proof, let us expand the left hand side of \eqref{mirror symmetry product 0}, obtaining by definition of the $\skj$ above
\begin{align*}
\sumkj\langle \partial_l g\skjcomp{l}, f\partial_m \smkjcomp{n}\rangle_{\eps}=\langle f ,\partial_l g\rangle_{\eps}  \sumkj \left\lvert\tkj\right\rvert^2 k_m  (a_{k,j})_l (a_{k,j})_n=0
\end{align*}
due to \eqref{property derivatives 5}.
\end{proof}
On the contrary, as discussed in \cite{flandoli2023boussinesq}, the asymptotic behavior of $\nabla \overline{Q^{\eps}_{\rho}}(0)$ plays a crucial role in order to understand the alpha-term of our model. Therefore, the last result we provide in this section is a convergence property of the matrix 
\begin{align}\label{definition matrix aepsgamma}
    \{{\mathcal{R}_{\eps,\gamma}}^{m,n}\}_{m,n\in\{1,2\}}=\partial_{n}\overline{Q^{\eps,3,m}_{\rho}}(0)
\end{align}
and a uniform bound on $\nabla \overline{Q^{\eps}_{\rho}}(0)$.
\begin{lemma}\label{convergence properties matrix}
It holds 
\begin{align*}
    \norm{\nabla\overline{Q^{\eps}_{\rho}}(0)}_{HS}\leq \frac{2\sqrt{2}\lvert\rho\rvert \zeta^N_{H,\gamma/2} N^{2-\gamma/2} }{C_{1,H}C_{2,H}}. 
\end{align*}
In particular
\begin{align*}
\{{\mathcal{R}_{\eps,\gamma}}^{m,n}\}_{m,n\in\{1,2\}}=\partial_{n}\overline{Q^{\eps,3,m}_{\rho}}(0)= \frac{\pi\rho \zeta_{H,\gamma/2}} {C_{1,H}C_{2,H}}\begin{bmatrix} 
   0 & -1\\ 1 & 0   \end{bmatrix} N^{2-\gamma/2}+O(N^{1-\gamma/2}). 
\end{align*}
\end{lemma}
\begin{proof}
    Thanks to \eqref{property derivatives 3} and \autoref{remark coefficients } it holds
    \begin{align}\label{uniform bound step 1}
    \nabla\overline{Q^{\eps}_{\rho}}(0)=\frac{2\rho}{C_{1,H}C_{2,H}}\sumH \frac{1}{\lvert k_H\rvert^{\gamma/2}}  \left(e_3 \otimes\frac{k_H^{\perp}}{\lvert k_H\rvert}\otimes  \frac{k_H}{\lvert k_H\rvert}-\frac{k_H^{\perp}}{\lvert k_H\rvert}\otimes e_3\otimes \frac{k_H}{\lvert k_H\rvert}\right).  
    \end{align}
Therefore we have
\begin{align}
    \norm{\nabla\overline{Q^{\eps}_{\rho}}(0)}_{HS}& \leq \frac{2\lvert\rho\rvert}{C_{1,H} C_{2,H}}\sumH \frac{1}{\lvert k_H\rvert^{\gamma/2}}\norm{\left(e_3 \otimes\frac{k_H^{\perp}}{\lvert k_H\rvert}\otimes  \frac{k_H}{\lvert k_H\rvert}-\frac{k_H^{\perp}}{\lvert k_H\rvert}\otimes e_3\otimes \frac{k_H}{\lvert k_H\rvert}\right)}_{HS}\notag \\ &= \frac{2\sqrt{2}\lvert\rho\rvert \zeta^N_{H,\gamma/2} N^{2-\gamma/2} }{C_{1,H}C_{2,H}}.
\end{align}
Due to \eqref{uniform bound step 1}, it holds
    \begin{align}\label{step 1 convergence matrix}
\partial_{\cdot}\overline{Q^{\eps,3,\cdot}_{\rho}}(0)&=\frac{2\rho }{C_{1,H}C_{2,H}}\sumH \frac{1}{\lvert k_H\rvert^{\gamma/2}}  \frac{k_H^{\perp}}{\lvert k_H\rvert} \otimes \frac{k_H}{\lvert k_H\rvert}\notag\\ & =\frac{2\rho }{C_{1,H}C_{2,H}N^{\gamma/2-2}}\left(\frac{1}{N^2}\sum_{\substack{k_H\in\Z^2_0\\ 1\leq \frac{\lvert k_H\rvert}{N}\leq 2}} \frac{N^{\gamma/2}}{\lvert k_H\rvert^{\gamma/2}}  \frac{k_H^{\perp}}{\lvert k_H\rvert} \otimes \frac{k_H}{\lvert k_H\rvert}\right).
    \end{align}
    The function $f_{\gamma}(x)=\frac{1}{\lvert x\rvert^{\gamma/2+2}} (x^{\perp}\otimes x) $ is smooth for $1\leq \lvert x \rvert \leq 2$, hence the Riemann sums satisfy
    \begin{align}\label{riemann sum convergence}
        \frac{1}{N^2}\sum_{\substack{k_H\in\Z^2_0\\ 1\leq \frac{\lvert k_H\rvert}{N}\leq 2}} \frac{N^{\gamma/2}}{\lvert k_H\rvert^{\gamma/2}}  \frac{k_H^{\perp}}{\lvert k_H\rvert} \otimes \frac{k_H}{\lvert k_H\rvert}=\int_{1\leq x\leq 2} f_\gamma(x) dx+O(N^{-1}).
    \end{align}
    Since we have
    \begin{align}\label{step 2 convergence matrix}
        \int_{1\leq x\leq 2} f_\gamma(x) dx=\int_1^2 \int_0^{2\pi}  \frac{1}{r^{\gamma/2-1}}\begin{bmatrix}
        -\sin(\theta)\cos(\theta) & -\sin^2(\theta)  \\
        \cos^2(\theta) & \sin(\theta)\cos(\theta)
        \end{bmatrix} dr d\theta=\frac{\zeta_{H,\gamma/2}}{2}\begin{bmatrix} 
   0 & -1\\ 1 & 0   \end{bmatrix} ,
    \end{align}
    combining \eqref{step 1 convergence matrix}, \eqref{riemann sum convergence} and \eqref{step 2 convergence matrix} the thesis follows.
\end{proof}
\begin{remark}\label{rmk: noise motivation}
   We highlight that our noise is an adaptation of the one constructed in \cite{flandoli2023boussinesq} for a 2D-3C model, where a certain helicity of the noise is imposed to achieve an alpha-term in the limit. In that work the authors constructed a random field based on the notion of a `cylinder of vorticity'. A cylinder of vorticity is for them a special 3D vector field $V(x)$ on $\T^2$ that is decomposed as $V_H$, an horizontal component corresponding ideally to a point vortex, and $V_3 e_3$ a vertical component, with the property that $V_H$ is an odd function of $x$, while $V_3$ is even. By randomizing the position of this cylinder, they were able to use the covariance matrix of the field to define a Brownian motion. While we could not give the same description here, since we needed to express the noise in the Fourier basis to exploit crucial simplifications (cfr. \autoref{remark fourier decomposition and splitting products}), our noise retains all the properties of the one used in that work. In particular, by splitting our noise as $W^\eps(t,x)= \overline{W^\eps}(t,x) + \left( W^\eps(t,x)\right)' $, where the first part comprises all the wavenumbers with $k_3=0$, and recalling the further orthogonal decomposition $\overline{W^\eps}(t,x) = \overline{W^{\eps,H}}(t,x) + \overline{W^{\eps,3}}(t,x)e_3 $, we see that this part of the noise correspond exactly to theirs. Indeed, in order to have a nonzero `off diagonal' term $\overline{Q^\eps_{\rho}}(0)$ in the covariance, we had to introduce a correlation between the horizontal and the vertical component, which was inherent in the structure of the cylinder, and breaks mirror symmetry for the horizontal component, which we did with the special choice of the coefficients, that is, we made it so that $\overline{W^{\eps,H}}$ is an odd function of $x$, while $\overline{W^{\eps,3}}$ is even, exactly like in the cylinder case.
   Finally, it is easy to check that our main assumption on the decay of the noise coefficients ($\beta, \gamma \ge 4$, cfr. \autoref{HP noise}) is equivalent to the main assumption of that work, that is, a uniform control on the second derivatives of the covariance matrix, see \cite[Hypothesis 3.3]{flandoli2023boussinesq}.  In \autoref{appB} we compute the mean helicity of our noise and stress its relation with the strength of the alpha-term in the limit PDE. 
\end{remark}

\subsection{Main Results}\label{sec main result}
Following the ideas introduced in \autoref{sec introduction}, we fix $T\in (0,+\infty)$ and we are interested in studying the properties of the following stochastic model on $\T^3_{\eps}$:
\begin{equation}\label{introductory equation stratonovich}
    \begin{cases}
        dB^{\eps}_{t}&=\eta \Delta B^{\eps}_{t}  dt + \sum_{\substack{k\in \Z^{3,N}_0\\ j\in \{1,2\}}}\mathcal{L}_{\skj}B^{\eps}_t  \circ dW^{k,j}_t, \\
        \operatorname{div}B^{\eps}_t&=0, \\
        B^{\eps}_t|_{t=0}&=B^{\eps}_0,
    \end{cases}
\end{equation}
with the coefficients $\skj$ defined in \autoref{subsect description noise}. We start rewriting it in It\^o form. First, let us observe that for each $k\in\Z^{3,N}_0,\ j\in \{1,2\}$ we have due to the linearity of $\mathcal{L}_{\skj}$
\begin{align*}
\mathcal{L}_{\skj}B^{\eps}_t  \circ dW^{k,j}_t&=\mathcal{L}_{\skj}B^{\eps}_t dW^{k,j}_t+\frac{1}{2}\left[ \mathcal{L}_{\skj}B^{\eps}_\cdot, W^{k,j}_\cdot\right]_t \\ & = \mathcal{L}_{\skj}B^{\eps}_t dW^{k,j}_t+\frac{1}{2}\mathcal{L}_{\skj}\left[ B^{\eps}_\cdot, W^{k,j}_\cdot\right]_t\\ & =\mathcal{L}_{\skj}B^{\eps}_t dW^{k,j}_t+\mathcal{L}_{\skj}\mathcal{L}_{\smkj}B^{\eps}_t dt+\rho\one_{\{k_3=0,\ l\neq j\}} \mathcal{L}_{\skj}\mathcal{L}_{\smk{l}}B^{\eps}_t dt.
\end{align*}
Therefore the equation for $B^{\eps}_t$ can be rewritten as
\begin{align*}
    \begin{cases}
        dB^{\eps}_{t}&=\left(\eta \Delta+\sumkj \mathcal{L}_{\skj}\mathcal{L}_{\smkj}+\rho\sumkmeancov \mathcal{L}_{\sk{1}}\mathcal{L}_{\smk{2}}+\mathcal{L}_{\sk{2}}\mathcal{L}_{\smk{1}} \right)  B^{\eps}_{t}  dt\\ & + \sum_{\substack{k\in \Z^{3,N}_0\\ j\in \{1,2\}}}\mathcal{L}_{\skj}B^{\eps}_t   dW^{k,j}_t, \\
        \operatorname{div}B^{\eps}_t&=0, \\
        B^{\eps}_t|_{t=0}&=B^{\eps}_0.
    \end{cases}    
\end{align*}
We are left to study $\sumkj \mathcal{L}_{\skj}\mathcal{L}_{\smkj}+\rho\sumkmeancov \mathcal{L}_{\sk{1}}\mathcal{L}_{\smk{2}}+\mathcal{L}_{\sk{2}}\mathcal{L}_{\smk{1}} $. Its analysis exploits strongly the properties of the covariance function of our noise, i.e. translation invariance and (almost) mirror symmetry and it is contained in the following Lemma. Even if its proof is quite standard, see \cite[Section 3.4.1]{FlaLuoWaseda} we prefer to include it in \autoref{appendix ito strat corr} for the convenience of the reader.
\begin{lemma}\label{lemma ito strat corrector}
If $F:\T^3_{\eps}\rightarrow\R^3$ is a smooth, zero mean divergence free vector field, then it holds   
\begin{align*}
    \sumkj \mathcal{L}_{\skj}\mathcal{L}_{\smkj} F+\rho\sumkmeancov \left(\mathcal{L}_{\sk{1}}\mathcal{L}_{\smk{2}}+\mathcal{L}_{\sk{2}}\mathcal{L}_{\smk{1}}\right)F&=\Lambda^{\eps}F+ \Lambda_{\rho}^{\eps}F,\quad  
\end{align*}
where in the formula above we denoted by $\Lambda^{\eps}$ the differential operator 
\begin{align*}
\Lambda^{\eps}F=\eta^{\eps}_T(\partial^2_{11}+\partial^2_{22})F+\eta^{\eps}_R\partial^2_{33}F
\end{align*}
and by $\Lambda_{\rho}^{\eps}$ the differential operator 
\begin{align*}
\Lambda_{\rho}^{\eps}F=-\sum_{l\in \{1,2\}} \partial_l \overline{Q^{\eps}_{\rho}}(0)\cdot\nabla F_l=-\sum_{\substack{j\in \{1,2,3\}\\l\in \{1,2\}}} \partial_l \overline{Q^{\eps,\cdot, j}_{\rho}}(0)\partial_jF_l.
\end{align*}
\end{lemma}    
Thanks to \autoref{lemma ito strat corrector}, equation \eqref{introductory equation stratonovich} can be rewritten in It\^o form as:
\begin{equation}\label{Introductory equation Ito}
    \begin{cases}
        dB^{\eps}_{t}&=\left(\eta \Delta+\Lambda^{\eps}+\Lambda^{\eps}_{\rho}\right)B^{\eps}_{t}  dt + \sum_{\substack{k\in \Z^{3,N}_0\\ j\in \{1,2\}}}\mathcal{L}_{\skj}B^{\eps}_t   dW^{k,j}_t, \\
        \operatorname{div}B^{\eps}_t&=0, \\
        B^{\eps}_t|_{t=0}&=B^{\eps}_0,
    \end{cases}
\end{equation}
We give now the definition of weak solution for the Stochastic equation of the magnetic field  \eqref{Introductory equation Ito}.

\begin{definition}\label{well posed def}
A stochastic process \begin{align*}
    B^{\eps}\in C_{\mathcal{F}}([0,T];\mathbf{L}^2_{\eps})\cap L^2_{\mathcal{F}}([0,T];\mathbf{H}^1_{\eps})
\end{align*} is a weak solution of equation \eqref{Introductory equation Ito} if, for every $\phi\in \mathbf{H}^2_{\eps}$, we have
\begin{align}\label{weak formulation full system}
    \langle B^{\eps}_t,\phi\rangle_{\eps}
    &=  \langle B^{\eps}_0,\phi\rangle_{\eps}+ \eta \int_0^t  \langle B^{\eps}_s,\Delta\phi\rangle_{\eps} \, ds+\int_0^t  \langle B^{\eps}_s,\Lambda^{\eps}\phi\rangle_{\eps}\, ds+\int_0^t  \langle \Lambda^{\eps}_{\rho}B^{\eps}_s,\phi\rangle_{\eps}\, ds+\sum_{\substack{k\in \Z^{3,N}_0\\ j\in \{1,2\}}}\int_0^t\langle \mathcal{L}_{\skj}B^{\eps}_s,\phi\rangle_{\eps}   dW^{k,j}_s.
\end{align}
for every $t\in [0,T],\ \mathbb{P}-a.s.$
\end{definition}

Due to the fact that equation \eqref{Introductory equation Ito} is linear, the existence and uniqueness of solutions in the sense of \autoref{well posed def} is a standard fact which follows by the abstract theory of \cite[Chapters 3-5]{Flandoli_Book_95}, see also \cite[Section 3.1]{flandoli2022heat}. Indeed the following holds.

\begin{theorem}\label{thm well posed}
    For each $B^{\eps}_0\in \mathbf{L}^2_{\eps}$ there exists a unique weak solution of system \eqref{Introductory equation Ito} in the sense of \autoref{well posed def}.
\end{theorem}

We are now ready to introduce the mean and the fluctuation in the third component of the process $B^{\eps}$ which are the main objects of our analysis. Therefore, let \begin{align*}
    \mB{t}:=M_{\eps}B^{\eps}_t,\ \oB{t}:=N_{\eps}B^{\eps}_t.
\end{align*}
More in general, in order to simplify the notation, in the following we will denote the projection operator $M_{\eps}$ (resp. $N_{\eps}$) by $\overline{\cdot}$ (resp. $\cdot'$) forgetting the dependence from $\eps$.
Thanks to \autoref{thm well posed} and the fact that $M_{\eps}$ and $N_{\eps}$ are projection operator and commute with the operator $\operatorname{div}$, see \autoref{properties of M N}, we have in particular that
\begin{align*}
\overline{B^{\eps}}\in L^2_{\mathcal{F}}(0,T;\mathbf{H}^1_{\eps})\cap C_{\mathcal{F}}(0,T;\mathbf{L}^2_{\eps}),\quad {{B}^{\eps}}'\in L^2_{\mathcal{F}}(0,T;\mathbf{H}^1_{\eps})\cap C_{\mathcal{F}}(0,T;\mathbf{L}^2_{\eps}).\end{align*}
For each $\phi\in \mathbf{H}^2_{\eps}$, considering $\overline{\phi}$ and $\phi'$ as test function in \eqref{weak formulation full system} we obtain that $\mB{t}$ and $\oB{t}$ satisfies
\begin{align}\label{weak formulation mean}
    \langle \mB{t},\phi\rangle_{\eps}
    &=  \langle \mB{0},\phi\rangle_{\eps}+ \eta \int_0^t  \langle \mB{s},\Delta\phi\rangle_{\eps} \, ds+\int_0^t  \langle \mB{s},\Lambda^{\eps}\phi\rangle_{\eps}\, ds+\int_0^t  \langle \Lambda^{\eps}_{\rho}\mB{s},\phi\rangle_{\eps}\, ds\notag\\ &+\sumkmeanj\int_0^t\langle\LL_{\skj}\mB{s},\phi\rangle_{\eps}   dW^{k,j}_s + \sumkoscj\int_0^t\langle\overline{\LL_{\skj}\oB{s}},\phi\rangle_{\eps}   dW^{k,j}_s,
\end{align}
\begin{align}\label{weak formulation oscj}
    \langle \oB{t},{\phi}\rangle_{\eps}
    &=  \langle \oB{0},{\phi}\rangle_{\eps}+ \eta \int_0^t  \langle \oB{s},\Delta{\phi}\rangle_{\eps} \, ds+\int_0^t  \langle \oB{s},\Lambda^{\eps}{\phi}\rangle_{\eps}\, ds+\int_0^t  \langle \Lambda^{\eps}_{\rho}\oB{s},{\phi}\rangle_{\eps}\, ds\notag\\ &+\sumkmeanj\int_0^t\langle\LL_{\skj}\oB{s},{\phi}\rangle_{\eps}   dW^{k,j}_s +\sumkoscj\int_0^t\langle\LL_{\skj}\mB{s},{\phi}\rangle_{\eps}   dW^{k,j}_s\notag\\ &+ \sumkoscj\int_0^t\langle\left(\LL_{\skj}\oB{s}\right)',{\phi}\rangle_{\eps}   dW^{k,j}_s
\end{align}
for every $t\in [0,T],\ \mathbb{P}-a.s.$ Namely $\mB{t}$ and $\oB{t}$ are weak solutions of the following SPDEs:
\begin{equation}\label{Introductory equation Ito mean}
    \begin{cases}
        d\mB{t}&=\left(\eta \Delta+\Lambda^{\eps}+\Lambda^{\eps}_{\rho}\right)\mB{t} dt + \sumkmeanj\LL_{\skj}\mB{t}   dW^{k,j}_t\\ & +\sumkoscj\overline{\LL_{\skj}\oB{t}}  dW^{k,j}_t, \\
        \operatorname{div}\mB{t}&=0, \\
        \mB{t}|_{t=0}&=\mB{0},
    \end{cases}
\end{equation}
\begin{equation}\label{Introductory equation Ito oscillation}
    \begin{cases}
        d\oB{t}&=\left(\eta \Delta+\Lambda^{\eps}+\Lambda^{\eps}_{\rho}\right)\oB{t} dt + \sumkmeanj\LL_{\skj}\oB{t}  dW^{k,j}_t\\ &+\sumkoscj\LL_{\skj}\mB{t}   dW^{k,j}_t+\sumkoscj\left(\LL_{\skj}\oB{t}\right)'  dW^{k,j}_t, \\
        \operatorname{div}\oB{t}&=0, \\
        \oB{t}|_{t=0}&=\oB{0}.
    \end{cases}
\end{equation}
As discussed in \autoref{sec introduction}, we study the convergence of $\mB{t}$ in the simultaneous scaling limit of both separation of scales, i.e. considering a noise which concentrates on smaller and smaller scales as described in \autoref{subsect description noise}, and of the width of our thin layer, namely considering $\eps=\frac{1}{N}\rightarrow 0$ as discussed in \cite{Raugel1, RaugelII,TemamZianeThin3d}. In order to do so, we need to introduce some few objects: we denote by \begin{align*}
    \mA{t}=K_{\eps}[\mB{t}],\quad \oA{t}=K_{\eps}[\oB{t}].
\end{align*}
Thanks to the regularity properties of the Biot-Savart operator, see \autoref{function spaces periodic}, we have that 
\begin{align*}
    \overline{A^{\eps}}\in L^2_{\mathcal{F}}(0,T;\mathbf{H}^2_{\eps})\cap C_{\mathcal{F}}(0,T;\mathbf{H}^1_{\eps}),\quad {{A}^{\eps}}'\in L^2_{\mathcal{F}}(0,T;\mathbf{H}^2_{\eps})\cap C_{\mathcal{F}}(0,T;\mathbf{H}^1_{\eps}),
\end{align*}
moreover \begin{align*}
    \overline{A^{\eps}}=(\overline{A^{H,\eps}}, \overline{A^{3,\eps}})^t=(-\nabla_H^{\perp}(-\Delta)^{-1}\overline{B^{3,\eps}},(-\Delta)^{-1}\nabla^{\perp}_H\cdot \overline{B^{H,\eps}})^t.
    \end{align*}
Lastly we introduce our limit objects. Following the idea first introduced in \cite{galeati2020convergence}, we expect that our limit objects satisfy a PDE with an additional second order operator with intensity related to \begin{align*}
    \operatorname{lim}_{\eps\rightarrow 0}\sum_{\substack{k\in \Z^{3,N}_0\\ j\in \{1,2\}}}\norm{\skj}_{\eps}^2.
\end{align*}
Besides to the fact that the limit above is $0$, in the simultaneous separation of scale, thin layer limit, according to the discussion in \autoref{subsect description noise}, the final deterministic objects keep memory of the operators $\Lambda^{\eps},\ \Lambda^{\eps}_{\rho}$ defined above. Indeed, denoting by $\overline{A^3},\ \overline{B^3}$ the unique weak solutions of the following linear 2D PDEs 
\begin{align}\label{limit solution A3}
&\begin{cases}
\partial_t \overline{A^{3,\gamma}_t}&=(\eta+\eta_T)\Delta \overline{A^{3,\gamma}_t}\quad x\in \T^2,\ t\in (0,T)\\
\overline{A^{3,\gamma}_t}|_{t=0}&=\overline{A^3_0},   
\end{cases}\\
\label{limit solution B3}
&\begin{cases}
\partial_t \overline{B^{3,\gamma}_t}&=(\eta+\eta_T)\Delta \overline{B^{3,\gamma}_t}+\operatorname{div}\left(\mathcal{R}_{\gamma}\nabla^{\perp}\overline{A^{3,\gamma}_t}\right)\quad x\in \T^2,\ t\in (0,T)\\
\overline{B^{3,\gamma}_t}|_{t=0}&=\overline{B^3_0},   
\end{cases}
\end{align}
where $\eta_T=\zeta_{H,2}$ and 
\begin{align*}
\mathcal{R}_{\gamma}=\begin{cases}
   \frac{2\pi\rho\log{2}} {C_{1,H}C_{2,H}}\begin{bmatrix} 
   0 & -1\\ 1 & 0
   \end{bmatrix} & \text{if } \gamma=4\\
   0 & \text{if } \gamma>4,
\end{cases}
\end{align*}
our main result reads as follows:
\begin{theorem}\label{main Theorem}
Assuming \autoref{HP noise} and \begin{align}\label{assumptions initial conditions main thm}
\mBcomp{0}{3}\rightharpoonup \overline{B^3_0} \text{ in }\Dot L^2(\T^2),\quad \mAcomp{0}{3}\rightharpoonup \overline{A^3_0} \text{ in }\Dot L^2(\T^2),\quad \sup_{\eps\in (0,1)}\frac{\norm{\oBcomp{0}{3}}_{\eps}^2}{\eps},\quad \sup_{\eps\in (0,1)}\frac{\norm{\oA{0}}_{\eps}^2}{\eps}<+\infty, 
\end{align}
there exists $N_0$ large enough such that for each $N\geq N_0$, for each  $\theta_1\in (0,1],\ \theta_2\in (0,\theta_1),\ \delta\in (0,\theta_2)$ we have
\begin{align}\label{main thm ineq 1}
\operatorname{sup}_{t\in [0,T]}\expt{\norm{\mBcomp{t}{3}- \overline{B^{3,\gamma}_t}}_{\Dot H^{-\theta_1}}^2}  &\lesssim \begin{cases}
    \norm{\overline{B^3_0}-\mBcomp{0}{3}}^2_{\Dot{H}^{-\theta_1}}+\norm{\overline{A^3_0}-\mAcomp{0}{3}}^2_{\Dot{H}^{-\theta_2}}+\frac{1}{
    N^{\frac{(\theta_2-\delta)(\theta_1-\theta_2)}{(1+\delta)}}}\\   \text{if }\gamma=4,\\
    \norm{\overline{B^3_0}-\mBcomp{0}{3}}^2_{\Dot{H}^{-\theta_1}}+\norm{\overline{A^3_0}-\mAcomp{0}{3}}^2_{\Dot{H}^{-\theta_2}}+\frac{1}{N^{\gamma-4}}+\frac{1}{
    N^{\frac{(\theta_2-\delta)(\theta_1-\theta_2)}{(1+\delta)}}}\\  \text{if }\gamma>4;
    \end{cases}  \\
\label{main thm ineq 2} \operatorname{sup}_{t\in [0,T]}\expt{\norm{\mAcomp{t}{3}-\overline{A^{3,\gamma}_t}}_{\Dot H^{-\theta_2}}^2}&\lesssim \norm{\mAcomp{0}{3}-\overline{A^3_0}}_{\Dot{H}^{-\theta_2}}^2+\frac{1}{N^{\theta_2}}+\frac{1}{
N^{\frac{2(\theta_2-\delta)}{1+\delta}}}.
\end{align}
The hidden constants in \eqref{main thm ineq 1}, \eqref{main thm ineq 2} depend on $\eta,$ $\beta,$ $\gamma,$ $C_V,$ $C_{1,H},$ $C_{2,H},$ $\lvert \rho\rvert$, $T$, $\operatorname{sup}_{\eps\in (0,1)}{\norm{\mBcomp{0}{3}}^2},$      \\ $\operatorname{sup}_{\eps\in (0,1)}{\norm{\mAcomp{0}{3}}^2}$, $\operatorname{sup}_{\eps\in (0,1)}\frac{{\norm{\oBcomp{0}{3}}^2_{\eps}}}{{\eps}},$ $\operatorname{sup}_{\eps\in (0,1)}\frac{{\norm{\oAcomp{0}{3}}^2_{\eps}}}{{\eps}},$ $\theta_1,$ $\theta_2$, $\delta $. 
In particular 
\begin{align*}
\operatorname{sup}_{t\in [0,T]}\expt{\norm{\mBcomp{t}{3}-\overline{B^{3,\gamma}_t}}_{\Dot{H}^{-\theta_1}}^2}+\operatorname{sup}_{t\in [0,T]}\expt{\norm{\mAcomp{t}{3}-\overline{A^{3,\gamma}_t}}_{\Dot{H}^{-\theta_2}}^2}\rightarrow 0\quad \text{as } N\rightarrow +\infty.
\end{align*}
\end{theorem}
\begin{remark}\label{explicit dependence}
The explicit dependence of the hidden constants in \eqref{main thm ineq 1}, \eqref{main thm ineq 2} in terms of $\theta_1,\ \theta_2$ and $\delta$ can be found in \autoref{sec proof main thm}, see \autoref{preliminary convergence} and \autoref{preliminary convergence 2}.     
\end{remark}
\begin{remark}\label{remark HP satisfied}
Our assumption \eqref{assumptions initial conditions main thm} is for example satisfied if $B^{\eps}_0$ can be decomposed as
\begin{align*}
B^{\eps}_0(x_1,x_2,x_3)=\overline{B}_0(x_1,x_2)+{B}'_0\left(x_1,x_2,\frac{x_3}{\eps}\right),
\end{align*}
where, calling ${B^{\eps}_0}'=B_0'\left(\cdot,\cdot,\frac{\cdot}{\eps}\right)$, we have \begin{align*}
    &\overline{B}_0\in \Dot L^2(\T^2; \R^3),\quad {B}'_0\in \Dot L^2(\T^3; \R^3),\quad \int_{\T} B'_0(x_1,x_2,x_3) dx_3=0\ a.e.\ (x_1,x_2)\in \T^2\\  &\operatorname{div} \overline{B}_0=0,\quad \operatorname{div}{B^{\eps}_0}'=0\quad \forall N\in \N.
\end{align*}
\end{remark}
\begin{remark}
Due to \autoref{HP noise}, $\eta_T$ cannot be arbitrarily large in our model. Indeed the following holds
\begin{align*}
    \eta_T=\frac{\zeta_{H,2}}{C_{1,H}^2}<\frac{\zeta_{H,2}}{\zeta_{H,0}}\eta\approx 0.462 \eta.
\end{align*}
\end{remark}
\begin{remark}\label{remark why quantitative}
Contrary to other results related to ours, see \cite{galeati2020convergence, flandoli2021scaling,flandoli2023boussinesq,flandoli20232d,TemamZianeThin3d}, we do not argue by compactness but we provide a quantitative result in the original probability space. Besides from the fact that this makes the result stronger, we are forced to argue in this way. Indeed, we can obtain uniform estimates in a common topology only for $\mBcomp{\cdot}{3},\ \mAcomp{\cdot}{3}$ exploiting the fact that they are independent from the third variable and we can see them as functions from $\T^2$ to $\R^3$. This seems not possible for quantities related to $\oB{\cdot}$. Due to the lack of compactness for quantities related to $\oB{\cdot}$, we cannot find by Skorokhod's representation theorem an auxiliary probability space where either \eqref{weak formulation mean}, \eqref{weak formulation oscj} hold and $(\mBcomp{\cdot}{3},\  \mAcomp{\cdot}{3})$ converge a.s. to some limit objects. Therefore we cannot pass to the limit in equation \eqref{weak formulation mean} showing that the limit objects satisfy \eqref{limit solution A3}, \eqref{limit solution B3}.
\end{remark}
\subsection{Strategy of the Proof}\label{sec strategy proof}
The proof of \autoref{main Theorem} relies on rewriting in mild form $\mBcomp{t}{3},\ \mAcomp{t}{3},\ \overline{M^3_t},\ \overline{A^3_t}$, see \autoref{mild form}, showing, in a quantitative way, that the stochastic convolutions approach $0$ when considering negative Sobolev norms in $\T^2$ and having a good control on the pieces associated to the alpha-term in \eqref{Introductory equation Ito mean}, \eqref{limit solution B3}. This is the object of \autoref{corollary stochastic convolution}, \autoref{preliminary convergence} and \autoref{preliminary convergence 2}. These results (and actually also stronger ones) would be easily true if the following holds:
If $\norm{B_0^{\eps}}_{\eps}^2=O(\eps)$, then also
\begin{align}\label{dream relation}
    \norm{{B^{\eps}}}_{C_{\mathcal{F}}(0,T;\mathbf{L}^2_{\eps})}^2+\norm{{B^{\eps}}}_{L^2_{\mathcal{F}}(0,T;\mathbf{H}^1_{\eps})}^2=O(\eps).
\end{align}
Unfortunately, due to the stochastic stretching terms, equation \eqref{dream relation} seems completely out of reach if one is interested on having a nontrivial operator in the limit $\operatorname{lim}_{\eps\rightarrow 0}\Lambda^{\eps}$, see \cite[Appendix 2]{FlLuo21}. Even if the operator $\Lambda^{\eps}_{\rho}$ and the correlation between the components of the noise $W^{\eps}_t$ makes system \eqref{Introductory equation Ito} more complicated, they do not increase the difficulties in order to obtain \eqref{dream relation}. We rely on the following idea, which we inherited by \cite{Raugel1, TemamZianeThin3d}: since the equation for $\mB{t}$ is a perturbation of a 2D-3C system, we can work, as in \cite{flandoli2023boussinesq}, considering differently the horizontal and the vertical component of the magnetic field $\mB{t}$. Since the third component $\mBcomp{t}{3}$ corresponds, in a vague sense, to the unique non null component associated to the $\operatorname{curl}$ of a 2D divergence free vector field, according to \cite{flandoli2021scaling},\cite{flandoli2021quantitative},\cite{flandoli20232d}, we expect that, if we were able to control the perturbation with respect to the 2D-3C model, it behaves well in term of \eqref{dream relation}, see \cite{flandoli2023boussinesq}. This is not true when considering $\mBcomp{t}{H}$. Indeed, having a good control on the stochastic stretching terms appearing in the equation of $\mBcomp{t}{H}$ is as difficult as proving \eqref{dream relation} and keeping a nontrivial operator in the limit due to the fact that, without exploiting some geometric features of the model, the term 
\begin{align}\label{bad stretching term BH}
    \sumkmeanj\int_0^T \norm{\mB{s}\cdot\nabla\skjcomp{H}}_{\eps}^2 ds,
\end{align}
appearing in the energy estimates, seems to blow-up. The idea is, therefore, of considering the potential field $\mA{t}$ associated to $\mB{t}$. Thanks to \eqref{Biot savart independent from third component}, $\mBcomp{t}{H}=-\nabla^{\perp}_H \mAcomp{t}{3}$ and, formally, due to \autoref{Inversion of Lie derivative} below, it solves
\begin{align*}
d\mAcomp{t}{3}&=\left(\eta \Delta+\Lambda^{\eps}\right)\mAcomp{t}{3} dt + \sumkmeanj\skj\cdot \nabla \mAcomp{t}{3}  dW^{k,j}_t\\& +\sumkoscj Q\left[\overline{\skj\cdot \nabla \oAcomp{t}{3}+\partial_3\skj\cdot\oAcomp{t}{3}}  \right]dW^{k,j}_t.
\end{align*}
Notice that, due to \autoref{remark ineffective on 2D functions}, $\Lambda^{\eps}_{\rho}\mB{t}$ is a vector having the first two components null, therefore it does not affect the equation for $\mBcomp{t}{H}$.    
Now, the bad stretching term \eqref{bad stretching term BH} has disappeared at this level of regularity, and to show that $\mAcomp{t}{3}$ behaves well in term of \eqref{dream relation} we are left to analyze this perturbation of the 2D-3C model. In order to show that the equations for $\mAcomp{t}{3},\ \mBcomp{t}{3}$
are just a regular perturbation of the 2D-3C model we need either some smallness assumptions on the coefficient $\skj,\ k_3\neq 0$ in term of the derivative of the covariance function associated to the noise, see \autoref{subsect description noise}, and some bounds on $\oB{t}$. Also in this case, the analysis of $\oBcomp{t}{3}$ is relatively easy thanks to our assumptions on the noise which allows to control the stochastic stretching terms, see \autoref{subsect description noise}. Lastly we need a good control on $\oBcomp{t}{H}$. This term seems again critical for what concerns the stochastic stretching, indeed a term of the form 
\begin{align}\label{almost bad stretching term BH}
    \sumkmeanj\int_0^T \norm{\oB{s}\cdot\nabla\skjcomp{H}}_{\eps}^2 ds
\end{align}
appears when computing energy estimates. Besides its form analogous to \eqref{bad stretching term BH}, the term \eqref{almost bad stretching term BH} needs a different treatment than the others. For this term, we are able to exploit the special geometry of the thin layer, namely the Poincaré inequality \eqref{poincare thin layer} (with constant $\eps$) holds when considering $\oB{t}$, counterbalancing the gradients of the coefficients. Indeed, applying H\"older inequality to relation \eqref{almost bad stretching term BH}, the explosive behavior of $\norm{\nabla\skjcomp{H}}_{L^{\infty}(\T^3_{\eps})}$ is absorbed thanks to \eqref{poincare thin layer} and we obtain
\begin{align*}
  \sumkmeanj\int_0^T \norm{\oB{s}\cdot\nabla\skjcomp{H}}_{\eps}^2 ds& \leq \left(\frac{\zeta_{H,0}}{C_{1,H}^2}+o(1)\right)\int_0^T \norm{\nabla\oB{s}}^2_{\eps}ds.
\end{align*}
We refer to the analysis of $I^2_t$ in \autoref{ito formula oscillation} below for details. 
Since we already have good estimates on $\int_0^T\norm{\nabla\oBcomp{s}{3}}^2_{\eps}ds$ and \autoref{HP noise} holds, \eqref{almost bad stretching term BH} can be absorbed thanks to the dissipation properties of \eqref{Introductory equation Ito oscillation} obtaining a good control on the stochastic stretching term. Unfortunately this is not enough: energy estimates for $\oBcomp{t}{H}$ involves also terms related to $\int_0^T \norm{\nabla\mBcomp{t}{H}}_{\eps}^2 ds$ which are out of reach as discussed before. For this reason, we replace the analysis of $\oBcomp{t}{H}$ with a term which involves at most $\int_0^T\norm{\nabla\mBcomp{t}{3}}_{\eps}^2 ds$ and $ \int_0^T\norm{\nabla\mAcomp{t}{3}}_{\eps}^2 ds$ when computing its energy estimates and combined with $\oBcomp{t}{3}$ is enough to provide some closeness of $\mBcomp{t}{3}, \mAcomp{t}{3}$ from being a solution of the 2D-3C model. The right choice for our purposes is $\oA{t}=K_{\eps}[\oB{t}]$. The stochastic stretching term appearing in this case is not identically zero as for $\mAcomp{t}{3}$, but can be treated as for \eqref{almost bad stretching term BH}.\\ Compared to \cite{flandoli2023boussinesq}, the equations appearing for our system are more complicated since, even if linear, none of them is closed and they are interconnected in a not trivial way when computing a priori estimates. This is only partially due to the presence of the operator $\Lambda^{\eps}_{\rho}$ and the correlation between the components of the noise $W^{\eps}_t$. We hope that this subsection could help the reader in order to understand the reasons behind the long computations of \autoref{section a priori estimates} and the heuristics for \autoref{main Theorem}.

\section{A Priori Estimates}\label{section a priori estimates}
Recalling that \begin{align*}
\overline{B}^{\eps}\in L^2_{\mathcal{F}}(0,T;\mathbf{H}^1_{\eps})\cap C_{\mathcal{F}}(0,T;\mathbf{L}^2_{\eps})\quad
\text{(resp. }{B^{\eps}}'\in  L^2_{\mathcal{F}}(0,T;\mathbf{H}^1_{\eps})\cap C_{\mathcal{F}}(0,T;\mathbf{L}^2_{\eps})\text{),}
\end{align*}
since $\overline{B}^{\eps}$ in independent on the third variable we have
\begin{align*}
\overline{B}^{\eps}\in L^2_{\mathcal{F}}(0,T;\Dot H^1(\T^2;\R^3))\cap C_{\mathcal{F}}(0,T;\Dot L^2(\T^2;\R^3)).\end{align*}
The goal of \autoref{corollary compactness in space} is to prove some bounds uniform in $\eps>0$ for quantities related to  $\overline{B}^{\eps}$ in a common topology independent from $\eps$. In order to do so, we start obtaining in \autoref{ito formula oscillation} and \autoref{ito formula mean} some pathwise estimates for quantities related to $\overline{B}^{\eps}$ and ${B^{\eps}}'$. Thanks to these we are able to gain in \autoref{proposition a priori estimate 1} some intermediate a priori bounds in topologies depending from $\eps$. These results will imply \autoref{corollary compactness in space}.

As discussed in \autoref{sec strategy proof} we will work with the potential fields associated to $\mB{t}$ and $\oB{t}$ at a certain point. In order to do so we employ the following classical relations:
\begin{proposition}\label{equivalence norms}
Let $X\in \mathbf{L}^2_{\eps},\ Y\in \mathbf{H}^1_{\eps}$. Then 
\begin{align}\label{property 1 norm}
    \norm{X}^2_{\eps}&=\norm{(-\Delta)^{-1/2}\operatorname{curl}X}_{\eps}^2,\quad \norm{\nabla Y}^2_{\eps}=\norm{\operatorname{curl}Y}_{\eps}^2.
\end{align}
In particular if either $X$ and $Y$ are independent from the third variable, it holds
\begin{align}\label{property 2 norm}
    \norm{X^3}^2_{\eps}&=\norm{(-\Delta)^{-1/2}\nabla_{H}^{\perp} X^3}_{\eps}^2,\quad \norm{\nabla Y^3}^2_{\eps}=\norm{\nabla_H^{\perp}Y^3}_{\eps}^2.
\end{align}
\end{proposition}
\begin{proposition}\label{Inversion of Lie derivative}
Let $X:\T^3_{\eps}\rightarrow \R^3$ a smooth, zero mean, divergence free vector field, $Y\in \mathbf{H}^1_{\eps}$. Then
\begin{align}\label{property 1 inversion curl}
X\cdot\nabla Y-Y\cdot\nabla X\in \mathbf{L}^2_{\eps}.    
\end{align}
 Moreover it holds
 \begin{align}\label{property 2 inversion curl}
 X\cdot\nabla Y-Y\cdot\nabla X&=\operatorname{curl}\left(X\cdot \nabla K_{\eps}[Y]+(DX)K_{\eps}[Y]\right)\notag\\ & = 
 \operatorname{curl}\left(Q\left[X\cdot \nabla K_{\eps}[Y]+(DX)K_{\eps}[Y]\right]\right)\notag\\ & = \operatorname{curl}\left(P_{\eps}\left[Q\left[X\cdot \nabla K_{\eps}[Y]+(DX)K_{\eps}[Y]\right]\right]\right).
 \end{align}
 In particular if both $X$ and $Y$ are independent from the third variable, it holds
 \begin{align}\label{property 3 inversion curl}
 X\cdot\nabla Y^H-Y\cdot\nabla X^H=-\nabla^{\perp}_H\left(X\cdot\nabla K^3_{\eps}[Y]\right),  
 \end{align}
 where we denoted by $K^3_{\eps}[Y]$ the third component of $K_{\eps}[Y]$.
\end{proposition}
Now we have all the elements to provide the necessary a priori estimates.
\begin{lemma}\label{ito formula oscillation}
Assuming \autoref{HP noise}, for each $\xi_1,\ \xi_2>0$ we have
\begin{align}\label{ito formula A'}
   d\norm{\oA{t}}_{\eps}^2+2\eta\norm{\nabla \oA{t}}^2_{\eps}dt& \leq 2\sumkj  \langle D\skj\oA{t},\oA{t} \rangle_{\eps}dW^{k,j}_t \notag\\ &  +2\sumkoscj  \langle \skj\cdot \nabla \mA{t}+D\skj\mA{t},\oA{t} \rangle_{\eps}dW^{k,j}_t \notag\\ & +2\left(\frac{\zeta_{H,0}}{C_{1,H}^2}+\frac{\lvert \rho\rvert}{\xi_1}\frac{\left(\zeta_{H,0}+\zeta_{H,2}\right)}{C_{1,H}^2}+O(N^{-1})\right)\norm{\nabla\oA{t}}^2_{\eps}dt\notag\\ & +\left(6\frac{\zeta_{\beta}\zeta_{H,0} }{N^{\beta-2}C_V^2}+O(N^{1-\beta})\right)\lVert \nabla \mA{t}\rVert_{\eps}^2 dt +\left(6\frac{\zeta_{\beta-2}\zeta_{H,0}}{N^{\beta-4}C_V^2}+O(N^{3-\beta})\right)\norm{\mA{t}}_{\eps}^2 dt,
\end{align}
\begin{align}\label{ito formula M3'}
    &d\norm{\oBcomp{t}{3}}^2_{\eps}+2\eta\norm{\nabla \oBcomp{t}{3}}^2_{\eps}dt\notag\\& \leq -2\sum_{\substack{k\in\Z^{3,N}_0\\ j\in \{1,2\}}}  \langle \oB{t}\cdot\nabla\skjcomp{3},\oBcomp{t}{3} \rangle_{\eps}dW^{k,j}_t  +2\sumkoscj  \langle \skj\cdot \nabla \mBcomp{t}{3}-\mB{t}\cdot\nabla\skjcomp{3},\oBcomp{t}{3} \rangle_{\eps}dW^{k,j}_t\notag\\ &+2\frac{\lvert \rho\rvert}{\xi_2}\left(\frac{\zeta_{H,2}}{ C_{1,H}^2}+\sqrt{2}\frac{\zeta_{H,\gamma/2}}{C_{1,H}C_{2,H}}+O(N^{-1})\right)\norm{\nabla \oBcomp{t}{3}}_{\eps}^2 dt \notag\\ & +2\left(\frac{\zeta_{H,\gamma-2}}{N^{\gamma-4}C_{2,H}^2}+2\frac{\zeta_{\beta-2}\zeta_{H,0}}{N^{\beta-4}C_V^2}+\frac{\lvert \rho\rvert\xi_2\zeta_{H,\gamma-2}}{N^{\gamma-4}C_{2,H}^2}+\sqrt{2}\lvert \rho\rvert \xi_2\frac{\zeta_{H,\gamma/2}}{N^{\gamma-4}C_{1,H}C_{2,H}}+O(N^{3-\gamma}+N^{3-\beta})\right)\norm{\nabla\oA{t}}^2_{\eps}dt\notag\\ & +2\sumkoscj\norm{\skj\cdot\nabla\mBcomp{t}{3}}^2_{\eps}dt+\left(4\frac{\zeta_{\beta-2}\zeta_{H,0}}{N^{\beta-4}C_{V}^2}+O(N^{3-\beta})\right)\norm{\mB{t}}^2_{\eps}dt\notag\\ & -2\sumkoscj\norm{\overline{\skj\cdot \nabla \oBcomp{t}{3}-\oB{t}\cdot\nabla\skjcomp{3}}}_{\eps}^2 dt.
\end{align}
\end{lemma}
\begin{proof}
Since by \autoref{equivalence norms} it holds\begin{align}\label{equivalence norm A'}
    \norm{\oA{_t}}_{\eps}^2=\norm{(-\Delta)^{-1/2}\oB{_t}}_{\eps}^2,\quad  \norm{\nabla\oA{_t}}_{\eps}^2=\norm{\oB{_t}}_{\eps}^2,
\end{align} we apply It\^o formula to $\norm{(-\Delta)^{-1/2}\oB{_t}}_{\eps}^2$ obtaining    
\begin{align*}
d\norm{\oA{t}}^2_{\eps} = dW^{'}_t+2(-I_t^1+I_t^2+I_t^3+I_t^4+I^5_t+I^6_t)dt,
\end{align*}
where we denote by \begin{align*}
    dW'_t&=2\sumkmeanj  \langle \mathcal{L}_{\skj} \oB{t},(-\Delta)^{-1}\oB{t} \rangle_{\eps}dW^{k,j}_t  +2\sumkoscj  \langle \mathcal{L}_{\skj} \mB{t},(-\Delta)^{-1}\oB{t} \rangle_{\eps}dW^{k,j}_t \notag\\
& +2\sumkoscj  \langle \left(\mathcal{L}_{\skj}\oB{t}\right)',(-\Delta)^{-1}\oB{t} \rangle_{\eps}dW^{k,j}_t,\\
I^1_t&=\eta \langle \nabla\oB{t},(-\Delta)^{-1}\nabla\oB{t}\rangle_{\eps}+\sumkj\langle \oB{t}, (-\Delta)^{-1}\operatorname{div}(\skj\otimes\smkj \nabla \oB{t})\rangle_{\eps},\\
I^2_t&=\sumkmeanj \langle \LL_{\skj}\oB{t}, (-\Delta)^{-1}\LL_{\smkj}\oB{t}\rangle_{\eps}, \\ I^3_t&= \sumkoscj \langle \LL_{\skj} \mB{t}, (-\Delta)^{-1}\LL_{\smkj}\mB{t}\rangle_{\eps}, \\ I^4_t&=\sumkoscj \langle \left(\LL_{\skj}\oB{t}\right)', (-\Delta)^{-1}\left(\LL_{\smkj}\oB{t}\right)'\rangle_{\eps}, \\ I^5_t& =\sumkoscj \langle \LL_{\skj} \mB{t}, (-\Delta)^{-1}\left(\LL_{\smkj}\oB{t}\right)'\rangle_{\eps},\\ 
I_t^6&=\langle \Lambda^{\eps}_{\rho}\oB{t},(-\Delta)^{-1}\oB{t}\rangle_{\eps}+\rho\sumkmeancov\langle \mathcal{L}_{\sk{1}},(-\Delta)^{-1}\mathcal{L}_{\smk{2}}\rangle_{\eps}+\langle \mathcal{L}_{\sk{2}},(-\Delta)^{-1}\mathcal{L}_{\smk{1}}\rangle_{\eps}.
\end{align*}
In order to study the terms involved in previous relation we recall that by \autoref{Inversion of Lie derivative}
\begin{align*}
\curl(\skj\cdot\nabla\oA{t}+D\skj\oA{t})=\skj\cdot\nabla\oB{t}-\oB{t}\cdot \nabla \skj
\end{align*}
and analogously for the others. First we start rewritng the martingale term. By previous relation, exploiting the fact that $\curl \curl f=-\Delta f$ if $f$ is a divergence free vector field and derivatives and the operator $(-\Delta)^{-1}$ commute in the periodic framework we get 
\begin{align}\label{martingale term A'}
 dW'_t&=    2\sumkj  \langle \LL_{\skj}\oB{t},(-\Delta)^{-1}\oB{t} \rangle_{\eps}dW^{k,j}_t  +2\sumkoscj  \langle \LL_{\skj} \mB{t},(-\Delta)^{-1}\oB{t} \rangle_{\eps}dW^{k,j}_t \notag\\ & =  2\sumkj  \langle \curl\left(\skj\cdot \nabla \oA{t}+D\skj\oA{t}\right),(-\Delta)^{-1}\curl\oA{t} \rangle_{\eps}dW^{k,j}_t  \notag\\
&+2\sumkoscj  \langle \curl\left(\skj\cdot \nabla \mA{t}+D\skj\mA{t}\right),(-\Delta)^{-1}\curl\oA{t} \rangle_{\eps}dW^{k,j}_t \notag\\ & =2\left(\sumkj  \langle \skj\cdot \nabla \oA{t}+D\skj\oA{t},\oA{t} \rangle_{\eps}dW^{k,j}_t  +\sumkoscj  \langle \skj\cdot \nabla \mA{t}+D\skj\mA{t},\oA{t} \rangle_{\eps}dW^{k,j}_t\right)\notag\\ & =2\left(\sumkj  \langle D\skj\oA{t},\oA{t} \rangle_{\eps}dW^{k,j}_t  +\sumkoscj  \langle \skj\cdot \nabla \mA{t}+D\skj\mA{t},\oA{t} \rangle_{\eps}dW^{k,j}_t\right).
\end{align}
Now we can analyze the terms $I^1_t,\ I^2_t,\ I^3_t,\ I^4_t,\ I^5_t,\ I^6_t$. First we observe that due to the fact that derivatives and the operator $(-\Delta)^{-1}$ commute in the periodic framework we easily obtain
\begin{align}\label{estimate I_1}
    I^1_t &= \eta \norm{\nabla\oA{t}}_{\eps}^2+\sumkj\norm{\skj\cdot\nabla\oA{t}}^2_{\eps}.
\end{align}
Exploiting the fact that $\curl \curl f=-\Delta f$ if $f$ is a divergence free vector field we get
\begin{align*}
    I^2_t &=\sumkmeanj \langle \curl(\skj\cdot\nabla\oA{t}+D\skj\oA{t}), (-\Delta)^{-1}\left(\curl(\smkj\cdot\nabla\oA{t}+D\smkj\oA{t})\right)\rangle_{\eps}\\ & =\sumkmeanj \langle \curl\left( P_{\eps}\left[Q\left[\skj\cdot\nabla\oA{t}+D\skj\oA{t})\right]\right]\right), (-\Delta)^{-1}\curl\left( P_{\eps}\left[Q\left[\skj\cdot\nabla\oA{t}+D\skj\oA{t})\right]\right]\right)\rangle_{\eps}\\ & = \sumkmeanj\norm{P_{\eps}\left[Q\left[\skj\cdot\nabla\oA{t}+D\skj\oA{t}\right]\right]}^2_{\eps}\\ & \leq \sumkmeanj\norm{\skj\cdot\nabla\oA{t}+D\skj\oA{t}}^2_{\eps}= \sumkmeanj\norm{\skj\cdot\nabla\oA{t}}^2_{\eps}+\sumkmeanj\norm{D\skj\oA{t}}^2_{\eps},
\end{align*}
where in the last step we exploited \eqref{mirror symmetry product 0}. By H\"older inequality and Poincaré inequality \eqref{poincare thin layer} the last term in the relation above can be estimated by
\begin{align*}
\sumkmeanj\norm{D\skj\oA{t}}^2_{\eps}&  \leq \sumkmeanj (\tkj)^2\lvert k\rvert^2\norm{\oA{t}}^2_{\eps}\\ & \leq \eps^2\norm{\partial_3\oA{t}}^2_{\eps}\left(\sumkmeanjuno  (\tkj)^2\lvert k\rvert^2+\sumkmeanjdue  (\tkj)^2\lvert k\rvert^2\right)\\ & \leq \eps^2\norm{\nabla\oA{t}}^2_{\eps}\left(\sumH  \frac{1}{C^2_{1,H}}+\sumH  \frac{1}{C^2_{2,H}\lvert k_H\rvert^{\gamma-2}}\right)\\ & = \norm{\nabla\oA{t}}^2_{\eps}\left(\frac{\zeta^N_{H,0}}{C_{1,H}^2}+\frac{\zeta^N_{H,\gamma-2}}{N^{\gamma-2}C_{2,H}^2}\right)=\norm{\nabla\oA{t}}^2_{\eps}\left(\frac{\zeta_{H,0}}{C_{1,H}^2}+O(N^{-1})\right).
\end{align*}
Therefore 
\begin{align}\label{estimate I2}
    I^2_t \leq \sumkmeanj\norm{\skj\cdot\nabla\oA{t}}^2_{\eps}+\norm{\nabla\oA{t}}^2_{\eps}\left(\frac{\zeta_{H,0}}{C_{1,H}^2}+O(N^{-1})\right).
\end{align}
Analogously we can show that \begin{align}\label{estimate I_3}
    I^3_t & \leq \sumkoscj\norm{\skj\cdot\nabla\mA{t}}^2_{\eps}+\sumkoscj\norm{D\skj\mA{t}}^2_{\eps}\notag\\ & \leq \sumkoscj \left\lvert\tkj\right\rvert^2 \norm{\nabla \mA{t}}^2_{\eps}+\sumkoscj \left\lvert\tkj\right\rvert^2 \lvert k\rvert^2 \norm{\mA{t}}^2_{\eps}\notag \\ & \leq  \frac{(\zeta_{H,0} N^2+O(N))\lVert \nabla \mA{t}\rVert_{\eps}^2}{C_V^2} \sum_{k_3\in N\Z\setminus\{0\}}\frac{1}{\lvert k_3\rvert^{\beta}}+\frac{(\zeta_{H,0}N^2+O(N))\norm{\mA{t}}_{\eps}^2}{C_V^2}\sum_{k_3\in N\Z\setminus\{0\}}\frac{1}{\lvert k_3\rvert^{\beta-2}}\notag\\ & =\left(2\frac{\zeta_{\beta}\zeta_{H,0} }{N^{\beta-2}C_V^2}+O(N^{1-\beta})\right)\lVert \nabla \mA{t}\rVert_{\eps}^2+\left(2\frac{\zeta_{\beta-2}\zeta_{H,0}}{N^{\beta-4}C_V^2}+O(N^{3-\beta})\right)\norm{\mA{t}}_{\eps}^2.
\end{align}
The analysis of $I^4_t$ is similar recalling that by \autoref{properties of M N} the operator $N_{\eps}$ and the $\curl$ commute. Therefore we have, arguing as above,
\begin{align*}
I^4_t&=\sumkoscj\norm{P_{\eps}\left[Q\left[\left(\skj\cdot\nabla\oA{t}+D\skj\oA{t}\right)'\right]\right]}^2_{\eps}\\ & \leq \sumkoscj\norm{\left(\skj\cdot\nabla\oA{t}+D\skj\oA{t}\right)'}^2_{\eps}\\ & =\sumkoscj\norm{\skj\cdot\nabla\oA{t}+D\skj\oA{t}}^2_{\eps}-\sumkoscj\norm{\overline{\skj\cdot\nabla\oA{t}+D\skj\oA{t}}}^2_{\eps}\\ & \leq \sumkoscj\norm{\skj\cdot\nabla\oA{t}}^2_{\eps}+\sumkoscj\norm{D\skj\oA{t}}^2_{\eps}.
\end{align*}
In the last step we simply neglected the negative term and exploited \eqref{mirror symmetry product 0}. Arguing as when treating $I^2_t$, we can show by H\"older and Poincaré inequality in the thin layer \eqref{poincare thin layer} that \begin{align*}
   \sumkoscj\norm{D\skj\oA{t}}^2_{\eps} & \leq  \eps^2\sum_{\substack{k\in \Z^{3,N}_0\\ k_3\neq 0,\quad j\in\{1,2\}\\ N\leq \lvert k_H\rvert\leq 2N}} \frac{1}{C_V^2 \lvert k\rvert^{\beta-2}}\norm{\nabla \oA{t}}_{\eps}^2\\ & \leq \frac{\left(\zeta_{H,0}+O(N^{-1})\right)\norm{\nabla \oA{t}}_{\eps}^2}{C_V^2}\sum_{k_3\in N\Z\setminus\{0\}}\frac{1}{\lvert k_3\rvert^{\beta-2}}\\ & =\left(2\frac{\zeta_{\beta-2}\zeta_{H,0}}{N^{\beta-2}C_V^2}+O(N^{1-\beta})\right)\norm{\nabla \oA{t}}_{\eps}^2.
\end{align*}
In particular, it holds \begin{align}\label{estimate I_4}
    I^4_t \leq \sumkoscj\norm{\skj\cdot\nabla\oA{t}}^2_{\eps}+\left(2\frac{\zeta_{\beta-2}\zeta_{H,0}}{N^{\beta-2}C_V^2}+O(N^{1-\beta})\right)\norm{\nabla \oA{t}}_{\eps}^2.
\end{align}
For what concerns $I^5_t$, applying Cauchy-Schwarz and Young's inequalities we get 
\begin{align*}
    I^5_t&=\sumkoscj \langle P_{\eps}\left[Q\left[\skj\cdot\nabla\mA{t}+D\skj\mA{t}\right]\right], P_{\eps}\left[Q\left[\left(\smkj\cdot\nabla\oA{t}+D\smkj\oA{t}\right)'\right]\right]\rangle_{\eps}\\ & \leq \frac{I^3_t+I^4_t}{2}\\ & \leq \left(\frac{\zeta_{\beta}\zeta_{H,0} }{N^{\beta-2}C_V^2}+O(N^{1-\beta})\right)\lVert \nabla \mA{t}\rVert_{\eps}^2+\left(\frac{\zeta_{\beta-2}\zeta_{H,0}}{N^{\beta-4}C_V^2}+O(N^{3-\beta})\right)\norm{\mA{t}}_{\eps}^2\\ &+\left(\frac{\zeta_{\beta-2}\zeta_{H,0}}{N^{\beta-2}C_V^2}+O(N^{1-\beta})\right)\norm{\nabla \oA{t}}_{\eps}^2+\frac{1}{2}\sumkoscj\norm{\skj\cdot\nabla\oA{t}}^2_{\eps}.
\end{align*}
Let us analyze the last term. It holds
\begin{align*}
  \frac{1}{2}\sumkoscj\norm{\skj\cdot\nabla\oA{t}}^2_{\eps}& \leq \frac{\norm{\nabla\oA{t}}_{\eps}^2}{2C_V^2}\sum_{\substack{k\in \Z^{3,N}_0\\ k_3\neq 0,\quad j\in\{1,2\}\\ N\leq \lvert k_H\rvert\leq 2N}}\frac{1}{\lvert k\rvert^{\beta}}\\ & \leq \frac{\left(\zeta_{H,0}N^2+O(N)\right)\norm{\nabla\oA{t}}_{\eps}^2}{2C_V^2}\sum_{k_3\in N\Z\setminus\{0\}}\frac{1}{\lvert k_3\rvert^{\beta}}\\ & =\left(\frac{\zeta_{\beta}\zeta_{H,0}}{N^{\beta-2}C_V^2}+O(N^{1-\beta})\right)\norm{\nabla\oA{t}}_{\eps}^2.
\end{align*}
In conclusion it holds
\begin{align}\label{estimate I_5}
I^5_t&\leq \left(\frac{\zeta_{\beta}\zeta_{H,0} }{N^{\beta-2}C_V^2}+O(N^{1-\beta})\right)\lVert \nabla \mA{t}\rVert_{\eps}^2+\left(\frac{\zeta_{\beta-2}\zeta_{H,0}}{N^{\beta-4}C_V^2}+O(N^{3-\beta})\right)\norm{\mA{t}}_{\eps}^2\notag\\ &+\left(2\frac{\zeta_{\beta}\zeta_{H,0}}{N^{\beta-2}C_V^2}+O(N^{1-\beta})\right)\norm{\nabla \oA{t}}_{\eps}^2.    
\end{align}
Lastly we consider $I^6_t$: $\langle \Lambda^{\eps}_{\rho}\oB{t},(-\Delta)^{-1}\oB{t}\rangle_{\eps}$ can be estimated easily due to \autoref{convergence properties matrix} and \eqref{property 1 norm}, obtaining by Cauchy-Schwarz inequality
\begin{align*}
\langle \Lambda^{\eps}_{\rho}\oB{t},(-\Delta)^{-1}\oB{t}\rangle_{\eps}&\leq \frac{2\sqrt{2}\lvert \rho\rvert \zeta^N_{H,\gamma/2} N^{2-\gamma/2} }{C_{1,H}C_{2,H}}\norm{(-\Delta)^{-1/2} \oB{t}}_{\eps}\norm{\nabla \oB{t}}_{H^{-1}(\T^3_{\eps})}\\ & =\frac{2\sqrt{2}\lvert \rho\rvert \zeta^N_{H,\gamma/2} N^{2-\gamma/2} }{C_{1,H}C_{2,H}}\norm{\oA{t}}_{\eps}\norm{\nabla \oA{t}}_{\eps}.
\end{align*}
The other terms can be treated easily by Cauchy-Schwartz inequality and \autoref{Inversion of Lie derivative}. Indeed for each $\xi_1>0$ and $k\in \Z_0^{3,N}$ such that $k_3=0$ we have
\begin{align*}
    \langle \mathcal{L}_{\sk{1}},(-\Delta)^{-1}\mathcal{L}_{\smk{2}}\rangle_{\eps}&=\langle \operatorname{curl}\left(\sk{1}\cdot\nabla \oA{t}+D\sk{1}\oA{t}\right),(-\Delta)^{-1}\operatorname{curl}\left(\smk{2}\cdot\nabla \oA{t}+D\smk{2}\oA{t}\right)\rangle_{\eps}\\ & =\langle \operatorname{curl}\left(P_\eps[Q[\sk{1}\cdot\nabla \oA{t}+D\sk{1}\oA{t}]]\right),(-\Delta)^{-1}\operatorname{curl}\left(P_\eps[Q[\smk{2}\cdot\nabla \oA{t}+D\smk{2}\oA{t}]]\right)\rangle_{\eps}\\ & =\langle \left(P_\eps[Q[\sk{1}\cdot\nabla \oA{t}+D\sk{1}\oA{t}]]\right),\left(P_\eps[Q[\smk{2}\cdot\nabla \oA{t}+D\smk{2}\oA{t}]]\right)\rangle_{\eps}\\ & \leq \frac{\norm{\sk{1}\cdot\nabla \oA{t}+D\sk{1}\oA{t}}_{\eps}^2}{2\xi_1}+ \frac{\xi_1 \norm{\sk{2}\cdot\nabla \oA{t}+D\sk{2}\oA{t}}_{\eps}^2}{2}.
\end{align*}
So far, arguing when treating $I^2_t$, we proved that for each $\xi_1>0$
\begin{align}\label{estimate I6 prefinal}
   I^6_t&\leq \frac{2\sqrt{2}\lvert \rho\rvert \zeta^N_{H,\gamma/2} }{N^{\gamma/2-2}C_{1,H}C_{2,H}}\norm{\oA{t}}_{\eps}\norm{\nabla \oA{t}}_{\eps}+\frac{\lvert \rho\rvert}{\xi_1} \sumkmeanjuno \norm{\skj\cdot \nabla \oA{t}}^2_{\eps}\notag\\ &+ \lvert \rho\rvert\xi_1 \sumkmeanjdue \norm{\skj\cdot \nabla \oA{t}}^2_{\eps}+\norm{\nabla \oA{t}}_{\eps}^2\left(\frac{\lvert \rho\rvert\zeta_{H,0}}{\xi_1 C_{1,H}^2}+O(N^{-1})\right).
\end{align}
Let us analyze $\sumkmeanjuno \norm{\skj\cdot \nabla \oA{t}}^2_{\eps}$ and $\sumkmeanjdue \norm{\skj\cdot \nabla \oA{t}}^2_{\eps}$. We have
\begin{align*}
\sumkmeanjuno\norm{\skj\cdot\nabla\oA{t}}^2_{\eps}& \leq \frac{\norm{\nabla\oA{t}}_{\eps}^2}{C_{1,H}^2}\sumH \frac{1}{\lvert k_H\rvert^2}\\ & =\left(\frac{\zeta_{H,2}}{C_{1,H}^2}+O(N^{-1})\right)\norm{\nabla\oA{t}}_{\eps}^2
\end{align*}
and analogously
\begin{align*}
\sumkmeanjdue\norm{\skj\cdot\nabla\oA{t}}^2_{\eps}& \leq \left(\frac{\zeta_{H,\gamma}}{N^{\gamma-2}C_{1,H}^2}+O(N^{1-\gamma})\right)\norm{\nabla\oA{t}}_{\eps}^2.
\end{align*}
In conclusion, applying Poincaré inequality \eqref{poincare thin layer} to \eqref{estimate I6 prefinal}, it holds
\begin{align}\label{estimate I_6}
I^6_t&\leq \frac{\lvert \rho\rvert}{\xi_1}\left(\frac{\zeta_{H,0}+\zeta_{H,2}}{C_{1,H}^2}+O(N^{-1})\right)\norm{\nabla \oA{t}}_{\eps}^2.    
\end{align}
Combining \eqref{martingale term A'}, \eqref{estimate I_1}, \eqref{estimate I2}, \eqref{estimate I_3}, \eqref{estimate I_4}, \eqref{estimate I_5}, \eqref{estimate I_6} we get relation \eqref{ito formula A'}.\\
Now we move on to the analysis of $\norm{\oBcomp{t}{3}}^2_{\eps}$ which is similar, actually easier, than the previous one. Indeed, by It\^o formula it holds:
\begin{align*}
    d\norm{\oBcomp{t}{3}}^2_{\eps}=d\Tilde{W}'_t+2(-J^1_t+J^2_t+J^3_t+J^4_t+J^5_t+J^6_t)dt,
\end{align*}
where we denote by
\begin{align*}
    d\Tilde{W}'_t &=2\sumkmeanj  \langle \skj\cdot \nabla \oBcomp{t}{3}-\oB{t}\cdot\nabla\skjcomp{3},\oBcomp{t}{3} \rangle_{\eps}dW^{k,j}_t  \notag\\
&+2\sumkoscj  \langle \skj\cdot \nabla \mBcomp{t}{3}-\mB{t}\cdot\nabla\skjcomp{3},\oBcomp{t}{3} \rangle_{\eps}dW^{k,j}_t \notag\\
& +2\sumkoscj  \langle \left(\skj\cdot \nabla \oBcomp{t}{3}-\oB{t}\cdot\nabla\skjcomp{3}\right)',\oBcomp{t}{3} \rangle_{\eps}dW^{k,j}_t,\\
J^1_t&=\eta \norm{\nabla\oBcomp{t}{3}}_{\eps}^2+\sumkj\norm{\skj\cdot\nabla\oBcomp{t}{3}}_{\eps}^2,\\
J^2_t&=\sumkmeanj \langle \skj\cdot \nabla \oBcomp{t}{3}-\oB{t}\cdot\nabla\skjcomp{3}, \smkj\cdot \nabla \oBcomp{t}{3}-\oB{t}\cdot\nabla\smkjcomp{3}\rangle_{\eps},\\J^3_t&= \sumkoscj \langle \skj\cdot \nabla \mBcomp{t}{3}-\mB{t}\cdot\nabla\skjcomp{3}, \smkj\cdot \nabla \mBcomp{t}{3}-\mB{t}\cdot\nabla\smkjcomp{3}\rangle_{\eps},\\ J^4_t&=\sumkoscj \langle \left(\skj\cdot \nabla \oBcomp{t}{3}-\oB{t}\cdot\nabla\skjcomp{3}\right)', \left(\smkj\cdot \nabla \oBcomp{t}{3}-\oB{t}\cdot\nabla\smkjcomp{3}\right)'\rangle_{\eps}, \\ J^5_t& =\sumkoscj \langle \skj\cdot \nabla \mBcomp{t}{3}-\mB{t}\cdot\nabla\skjcomp{3},\left(\smkj\cdot \nabla \oBcomp{t}{3}-\oB{t}\cdot\nabla\smkjcomp{3}\right)'\rangle_{\eps},\\ J^6_t &=-\langle \operatorname{div}_H\left(\mathcal{R}_{\eps,\gamma} \oBcomp{t}{H}\right),\oBcomp{t}{3}\rangle_{\eps}\\ & +\rho\sumkmeancov\langle \sk{1}\cdot\nabla \oBcomp{t}{3},\smk{2}\cdot\nabla \oBcomp{t}{3}-\oB{t}\cdot\nabla\smkcomp{3}{2}\rangle_{\eps}+\langle \sk{2}\cdot\nabla \oBcomp{t}{3}-\oB{t}\cdot\nabla\skcomp{3}{2} ,\smk{1}\cdot\nabla \oBcomp{t}{3}\rangle_{\eps}.
\end{align*}
Since the $\skj$ are divergence free, the martingale term is equal to
\begin{align}\label{martingale term M3}
d\Tilde{W}'_t&=-2\sum_{\substack{k\in\Z^{3,N}_0\\ j\in \{1,2\}}}  \langle \oB{t}\cdot\nabla\skjcomp{3},\oBcomp{t}{3} \rangle_{\eps}dW^{k,j}_t  +2\sumkoscj  \langle \skj\cdot \nabla \mBcomp{t}{3}-\mB{t}\cdot\nabla\skjcomp{3},\oBcomp{t}{3} \rangle_{\eps}dW^{k,j}_t.    
\end{align}
The analysis of $J^2_t,\ J^3_t,\ J^4_t,\ J^5_t$ is similar to the corresponding terms when estimating $\norm{\oA{t}}_{\eps}^2$ and is based on the fact that the mixed products in each inner product are equal to $0$ thanks to \eqref{mirror symmetry product 0}. In order to treat $J_t^2$ let us preliminarily observe that if $k_3=0$, then $\skjcomp{3}\neq 0$ if and only if $j=2$. Therefore we get
\begin{align}\label{estimate J_2}
J_t^2&=\sumkmeanj\norm{\skj\cdot\nabla\oBcomp{t}{3}-\oB{t}\cdot\nabla\skjcomp{3}}^2_{\eps}\notag\\ & = \sumkmeanj\norm{\skj\cdot\nabla\oBcomp{t}{3}}^2_{\eps}+\sumkmeanjdue\norm{\oB{t}\cdot \nabla\skjcomp{3}}^2_{\eps}\notag\\ & \leq   \sumkmeanj\norm{\skj\cdot\nabla\oBcomp{t}{3}}^2_{\eps}+\sumH\frac{1}{C_{2,H}^2\lvert k_H\rvert^{\gamma-2}}\norm{\oB{t}}^2_{\eps}\notag\\
&\leq \sumkmeanj\norm{\skj\cdot\nabla\oBcomp{t}{3}}^2_{\eps}+\left(\frac{\zeta_{H,\gamma-2}}{N^{\gamma-4}C_{2,H}^2}+O(N^{3-\gamma})\right)\norm{\oB{t}}^2_{\eps}\notag\\ & =\sumkmeanj\norm{\skj\cdot\nabla\oBcomp{t}{3}}^2_{\eps}+\left(\frac{\zeta_{H,\gamma-2}}{N^{\gamma-4}C_{2,H}^2}+O(N^{3-\gamma})\right)\norm{\nabla\oA{t}}^2_{\eps},
\end{align}
exploiting relation \eqref{property 1 norm}  in the last equality.
The analysis of $J_t^3$ is analogous to that of $J_t^2$ and lead us to 
\begin{align}\label{estimate J_3}
J^3_t&\leq \sumkoscj\norm{\skj\cdot\nabla\mBcomp{t}{3}}^2_{\eps}+\left(2\frac{\zeta_{\beta-2}\zeta_{H,0}}{N^{\beta-4}C_{V}^2}+O(N^{3-\beta})\right)\norm{\mB{t}}^2_{\eps}. 
\end{align}
For what concerns $J^4_t$, exploiting the fact that $N_{\eps}$ is a projection on $\Dot L^2(\T^3_{\eps})$ we get easily    
\begin{align*}
    J_t^4&=\sumkoscj\norm{\skj\cdot \nabla \oBcomp{t}{3}-\oB{t}\cdot\nabla\skjcomp{3}}_{\eps}^2-\sumkoscj\norm{\overline{\skj\cdot \nabla \oBcomp{t}{3}-\oB{t}\cdot\nabla\skjcomp{3}}}_{\eps}^2 \\ & =\sumkoscj\left( \norm{\skj\cdot \nabla \oBcomp{t}{3}}_{\eps}^2+\norm{\oB{t}\cdot\nabla\skjcomp{3}}_{\eps}^2-\norm{\overline{\skj\cdot \nabla \oBcomp{t}{3}-\oB{t}\cdot\nabla\skjcomp{3}}}_{\eps}^2\right). 
\end{align*}
Let us analyze the second term in the last sum. We have by H\"older inequality and exploiting our definition of the $\skj$
\begin{align*}
\sumkoscj  \norm{\oB{t}\cdot\nabla\skjcomp{3}}_{\eps}^2 &\leq \sum_{\substack{k\in \Z^{3,N}_0\\ k_3\neq 0,\quad j\in\{1,2\}\\ N\leq \lvert k_H\rvert\leq 2N}} \frac{1}{C_V^2 \lvert k\rvert^{\beta-2}}\norm{ \oB{t}}_{\eps}^2\\ & \leq \left(2\frac{\zeta_{\beta-2}\zeta_{H,0}}{N^{\beta-4}C_V^2}+O(N^{3-\beta})\right)\lVert \nabla \oA{t}\rVert_{\eps}^2. 
\end{align*}
In conclusion we get
\begin{align}\label{estimate J_4}
J_t^4& \leq \sumkoscj\left( \norm{\skj\cdot \nabla \oBcomp{t}{3}}_{\eps}^2-\norm{\overline{\skj\cdot \nabla \oBcomp{t}{3}-\oB{t}\cdot\nabla\skjcomp{3}}}_{\eps}^2\right)+\left(2\frac{\zeta_{\beta-2}\zeta_{H,0}}{N^{\beta-4}C_V^2}+O(N^{3-\beta})\right)\lVert \nabla \oA{t}\rVert_{\eps}^2.    
\end{align}
Exploiting the fact that for each $f,g\in L^2(\T^3_{\eps})$ we have $\langle \overline{f},g'\rangle_{\eps}=0$ we get $J^5_t=0$. Indeed, thanks to the mirror symmetry property of the covariance function \eqref{mirror symmetry product 0} and the definition of the $\skj$ we have
\begin{align*}
    J^5_t&=\sumkoscj \langle \skj\cdot \nabla \mBcomp{t}{3}-\mB{t}\cdot\nabla\skjcomp{3},\smkj\cdot \nabla \oBcomp{t}{3}-\oB{t}\cdot\nabla\smkjcomp{3}\rangle_{\eps}\\ & =\sumkoscj \langle \skj\cdot \nabla \mBcomp{t}{3},\smkj\cdot \nabla \oBcomp{t}{3}\rangle_{\eps}+\sumkoscj \langle \mB{t}\cdot\nabla\skjcomp{3},\oB{t}\cdot\nabla\smkjcomp{3}\rangle_{\eps}\\ & =\sumkoscj \left\lvert\tkj\right\rvert^2\langle a_{k,j}\cdot \nabla \mBcomp{t}{3},a_{-k,j}\cdot \nabla \oBcomp{t}{3}\rangle_{\eps}+\sumkoscj \left\lvert\tkj\right\rvert^2 \left(a_{k,j}^3\right)^2\langle \mB{t}\cdot k, \oB{t}\cdot k\rangle\\ & =0.
\end{align*}
Lastly let us treat $J^6_t$: the term $-\langle \operatorname{div}_H\left(\mathcal{R}_{\eps,\gamma} \oBcomp{t}{H}\right),\oBcomp{t}{3}\rangle_{\eps}$ can be estimated easily due to \autoref{convergence properties matrix} and \eqref{property 1 norm} getting by Cauchy-Schwarz inequality
\begin{align*}
-\langle \operatorname{div}_H\left(\mathcal{R}_{\eps,\gamma} \oBcomp{t}{H}\right),\oBcomp{t}{3}\rangle_{\eps}&=  \langle \mathcal{R}_{\eps,\gamma} \oBcomp{t}{H},\nabla_H\oBcomp{t}{3}\rangle_{\eps}\\ & \leq  \frac{2\sqrt{2}\lvert \rho\rvert \zeta^N_{H,\gamma/2} }{N^{\gamma/2-2}C_{1,H}C_{2,H}}\norm{\nabla \oBcomp{t}{3}}_\eps \norm{\nabla \oA{t}}_\eps.
\end{align*}
The other terms can be treated easily by Cauchy-Schwartz inequality obtaining for each $\xi_2>0$ and $k\in \Z_0^{3,N}$ such that $k_3=0$ 
\begin{align*}
\langle \sk{1}\cdot\nabla \oBcomp{t}{3},\smk{2}\cdot\nabla \oBcomp{t}{3}-\oB{t}\cdot\nabla\smkcomp{3}{2}\rangle_{\eps}\leq \frac{\norm{\sk{1}\cdot\nabla \oBcomp{t}{3}}_{\eps}^2}{2\xi_2}+\frac{\xi_2 \norm{\sk{2}\cdot\nabla \oBcomp{t}{3}-\oB{t}\cdot\nabla\skcomp{3}{2}}_{\eps}^2}{2}.    
\end{align*}
So far, arguing when treating $J^2_t$, we proved that for each $\xi_2>0$
\begin{align}\label{estimate J6 prefinal}
   J^6_t&\leq \frac{2\sqrt{2}\lvert \rho\rvert \zeta^N_{H,\gamma/2}  }{N^{\gamma/2-2} C_{1,H}C_{2,H}}\norm{\nabla\oBcomp{t}{3}}_{\eps}\norm{\nabla \oA{t}}_{\eps}+\frac{\lvert \rho\rvert}{\xi_2} \sumkmeanjuno \norm{\skj\cdot \nabla \oBcomp{t}{3}}^2_{\eps}\notag\\ &+\lvert \rho\rvert\xi_2 \sumkmeanjdue \norm{\skj\cdot \nabla \oBcomp{t}{3}}^2_{\eps}+\left(\frac{\lvert \rho\rvert\xi_2\zeta_{H,\gamma-2}}{N^{\gamma-4}C_{2,H}^2}+O(N^{3-\gamma})\right)\norm{\nabla\oA{t}}^2_{\eps}.
\end{align}
Let us analyze $\sumkmeanjuno \norm{\skj\cdot \nabla \oBcomp{t}{3}}^2_{\eps}$ and $\sumkmeanjdue \norm{\skj\cdot \nabla \oBcomp{t}{3}}^2_{\eps}$. Arguing when treating $I^6_t$ we have
\begin{align*}
\sumkmeanjuno\norm{\skj\cdot\nabla\oBcomp{t}{3}}^2_{\eps}& \leq \left(\frac{\zeta_{H,2}}{C_{1,H}^2}+O(N^{-1})\right)\norm{\nabla\oBcomp{t}{3}}_{\eps}^2,\\
\sumkmeanjdue\norm{\skj\cdot\nabla\oBcomp{t}{3}}^2_{\eps}& \leq \left(\frac{\zeta_{H,\gamma}}{N^{\gamma-2}C_{1,H}^2}+O(N^{1-\gamma})\right)\norm{\nabla\oBcomp{t}{3}}_{\eps}^2.
\end{align*}
In conclusion, applying Young's inequality to the first term in \eqref{estimate J6 prefinal}, it holds
\begin{align}\label{estimate J_6}
J^6_t&\leq \frac{\lvert \rho\rvert}{\xi_2}\left(\frac{\zeta_{H,2}}{ C_{1,H}^2}+\sqrt{2}\frac{\zeta_{H,\gamma/2}}{C_{1,H}C_{2,H}}+O(N^{-1})\right)\norm{\nabla \oBcomp{t}{3}}_{\eps}^2\notag\\ &+\lvert \rho\rvert \xi_2\left(\frac{\zeta_{H,\gamma-2}}{N^{\gamma-4}C_{2,H}^2}+\frac{\zeta_{H,\gamma/2}}{N^{\gamma-4}C_{1,H}C_{2,H}}+O(N^{3-\gamma})\right)\norm{\nabla\oA{t}}^2_{\eps}.    
\end{align}
Combining \eqref{martingale term M3},  \eqref{estimate J_2}, \eqref{estimate J_3}, \eqref{estimate J_4}, \eqref{estimate J_6}, relation \eqref{ito formula M3'} follows. 
\end{proof}
\begin{lemma}\label{ito formula mean}
Assuming \autoref{HP noise}, for each $\xi_3>0$ we have
\begin{align}\label{ito formula barA3}
    d\norm{\mAcomp{t}{3}}_{\eps}^2+2\eta\norm{\nabla\mAcomp{t}{3}}_{\eps}^2 dt&\leq 2\sumkoscj\langle \skj\cdot\nabla \oAcomp{t}{3}+\partial_3\skj\cdot\oA{t},\mAcomp{t}{3}\rangle_{\eps}dW^{k,j}_t\notag\\ &+\left(4\frac{\zeta_{\beta}\zeta_{H,0}}{N^{\beta-2}C_V^2}+O(N^{1-\beta})\right)\norm{\nabla \oAcomp{t}{3}}_{\eps}^2dt+\left(4\frac{\zeta_{\beta-2}\zeta_{H,0}}{N^{\beta-4}C_V^2}+O(N^{3-\beta})\right)\norm{\oA{t}}_{\eps}^2dt,
\end{align}
\begin{align}\label{ito formula barM3}
    &d\norm{\mBcomp{t}{3}}_{\eps}^2+2\eta\norm{\nabla \mBcomp{t}{3}}_{\eps}^2+2\sumkoscj\norm{\skj\cdot\nabla\mBcomp{t}{3} }_{\eps}^2 dt\notag\\ & \leq -2\sumkmeanj\langle \mBcomp{t}{H}\cdot\nabla_H\skjcomp{3},\mBcomp{3}{H}\rangle_{\eps}dW^{k,j}_t+2\sumkoscj\langle \skj\cdot\nabla\oBcomp{t}{3}-\oB{t}\cdot\nabla\skjcomp{3},\mBcomp{t}{3}\rangle_{\eps}dW^{k,j}_t\notag\\ &+2\left(\frac{\zeta_{H,\gamma-2}}{N^{\gamma-4}C_{2,H}^2}+\lvert \rho\rvert\xi_3 \frac{\zeta_{H,\gamma-2}}{N^{\gamma-4}C_{2,H}^2}+\sqrt{2}\lvert \rho\rvert \xi_3\frac{\zeta_{H,\gamma/2}}{N^{\gamma-4}C_{1,H}C_{2,H}}+O(N^{3-\gamma})\right)\norm{\nabla \mAcomp{t}{3}}_{\eps}^2 dt\notag\\ & +2\sumkoscj\norm{\overline{\skj\cdot\nabla\oBcomp{t}{3}-\oB{t}\cdot\nabla\skjcomp{3}}}_{\eps}^2dt +2\frac{\lvert \rho\rvert}{\xi_3} \left(\frac{\zeta_{H,2}}{C_{1,H}^2}+\sqrt{2}\frac{\zeta_{H,\gamma/2}}{C_{1,H}C_{2,H}}+O(N^{-1})\right)\norm{\nabla\mBcomp{t}{3}}_{\eps}^2 dt. 
\end{align}
\end{lemma}
\begin{proof}
    First let us observe that thanks to \eqref{Biot savart independent from third component} and \autoref{equivalence norms}
    \begin{align}\label{equivalence norm barv3}
       \mAcomp{t}{3}&=-\nabla_{H}^{\perp}(-\Delta)^{-1}\mBcomp{t}{H},\quad
       \norm{\mAcomp{t}{3}}_{\eps}^2=  \norm{\left(-\Delta\right)^{-1/2}\mBcomp{t}{H}}_{\eps}^2,\quad\norm{\nabla_H\mAcomp{t}{3}}_{\eps}^2= \norm{\mBcomp{t}{H}}_{\eps}^2.
    \end{align}
Thus we apply It\^o formula to $\norm{\left(-\Delta\right)^{-1/2}\mBcomp{t}{H}}_{\eps}$ obtaining due to \autoref{remark ineffective on 2D functions}
\begin{align*}
    d\norm{\mAcomp{t}{3}}_{\eps}^2=d\overline{W}_t+2(-L^1_t+L^2_t+L^3_t+L^4_t)dt,
\end{align*}
where we denote by 
\begin{align*}
    d\overline{W}_t&=2\sumkmeanj\langle \skjcomp{H}\cdot\nabla_H\mBcomp{t}{H}-\mBcomp{t}{H}\cdot\nabla_H\skjcomp{H},(-\Delta)^{-1}\mBcomp{t}{H}\rangle_{\eps}dW^{k,j}_t\\ & +2\sumkoscj\langle \skj\cdot\nabla\oBcomp{t}{H}-\oB{t}\cdot\nabla\skjcomp{H},(-\Delta)^{-1}\mBcomp{t}{H}\rangle_{\eps}dW^{k,j}_t,\\
    L^1_t&=\eta \langle \nabla_H\mBcomp{t}{H},(-\Delta)^{-1}\nabla_H\mBcomp{t}{H}\rangle_{\eps}+\sumkj\langle \mBcomp{t}{H}, (-\Delta)^{-1}\operatorname{div}_H(\skjcomp{H}\otimes\smkjcomp{H} \nabla_H \mBcomp{t}{H})\rangle_{\eps},\\
    L^2_t&=\sumkmeanj\langle \skjcomp{H}\cdot\nabla_H\mBcomp{t}{H}-\mBcomp{t}{H}\cdot\nabla_H\skjcomp{H},(-\Delta)^{-1}\left(\smkjcomp{H}\cdot\nabla_H\mBcomp{t}{H}-\mBcomp{t}{H}\cdot\nabla_H\smkjcomp{H}\right) \rangle_{\eps}\\ L^3_t &=\sumkoscj  \langle \overline{\skj\cdot\nabla\oBcomp{t}{H}-\oB{t}\cdot\nabla\skjcomp{H}} ,(-\Delta)^{-1}\overline{\smkj\cdot\nabla\oBcomp{t}{H}-\oB{t}\cdot\nabla\smkjcomp{H}}  \rangle_{\eps}\\ L^4_t &=\rho\sumkmeancov \langle \skcomp{1}{H}\cdot\nabla \mBcomp{t}{H}-\mBcomp{t}{H}\cdot\nabla\sk{1},(-\Delta)^{-1}\left(\smkcomp{2}{H}\cdot\nabla \mBcomp{t}{H}-\mBcomp{t}{H}\cdot\nabla\smk{2}\right)\rangle_{\eps}\\ & +\rho\sumkmeancov \langle \skcomp{2}{H}\cdot\nabla \mBcomp{t}{H}-\mBcomp{t}{H}\cdot\nabla\sk{2},(-\Delta)^{-1}\left(\smkcomp{1}{H}\cdot\nabla \mBcomp{t}{H}-\mBcomp{t}{H}\cdot\nabla\smk{1}\right)\rangle_{\eps}.
\end{align*}
Arguing analogously to the first part of \autoref{ito formula oscillation} we can treat all the terms above. Thanks to \eqref{property 2 inversion curl} and \eqref{property 3 inversion curl} it holds
\begin{align}
\skjcomp{H}\cdot\nabla_H\mBcomp{t}{H}-\mBcomp{t}{H}\cdot\nabla\skjcomp{H}&=\nabla_H^{\perp}(\skjcomp{H}\cdot \nabla_H \mAcomp{t}{3}) & \text{if } k_3=0\label{formula 1 inversione curl},\\
\skj\cdot\nabla\oBcomp{t}{H}-\oB{t}\cdot\nabla\skjcomp{H}&=\left(\curl\left(\skj\cdot\nabla \oA{t}+D\skj\oA{t}\right)\right)^H& \text{if } k_3\neq 0 \label{formula due inversione curl}.
\end{align}
In particular, considering the mean on the third direction of $\skj\cdot\nabla\oBcomp{t}{H}-\oB{t}\cdot\nabla\skjcomp{H}$, we have
\begin{align}\label{formula tre inversione curl}
 \overline{\skj\cdot\nabla\oBcomp{t}{H}-\oB{t}\cdot\nabla\skjcomp{H}}&=-\nabla_H^{\perp}\left(\overline{\skj\cdot\nabla \oAcomp{t}{3}+\partial_3\skj\cdot\oA{t}} \right) & \text{if } k_3\neq 0.  
\end{align}
Exploiting these relations we can obtain the required estimates. Let us start analyzing $d\overline{W}_t$.
\begin{align}\label{martingale term vbar}
d\overline{W}_t &= 2\sumkmeanj\langle \nabla_H^{\perp}(\skjcomp{H}\cdot \nabla_H \mAcomp{t}{3}),(-\Delta)^{-1}\nabla_H^{\perp}\mAcomp{t}{3}\rangle_{\eps}dW^{k,j}_t\notag\\ & +2\sumkoscj\langle \left(\curl\left(\skj\cdot\nabla \oA{t}+D\skj\oA{t}\right)\right)^H,(-\Delta)^{-1}\nabla_H^{\perp}\mAcomp{t}{3}\rangle_{\eps}dW^{k,j}_t \notag\\ & =2 \sumkmeanj\langle \skjcomp{H}\cdot \nabla_H \mAcomp{t}{3},\mAcomp{t}{3}\rangle_{\eps}dW^{k,j}_t\notag\\ & +2\sumkoscj\langle \skj\cdot\nabla \oAcomp{t}{3}+\partial_3\skj\cdot\oA{t},\mAcomp{t}{3}\rangle_{\eps}dW^{k,j}_t\notag\\ & =2\sumkoscj\langle \skj\cdot\nabla \oAcomp{t}{3}+\partial_3\skj\cdot\oA{t},\mAcomp{t}{3}\rangle_{\eps}dW^{k,j}_t.   
\end{align}
Moving to $L^1_t$ we get easily, integrating by parts, that it holds
\begin{align}\label{estimate L1}
   L^1_t&=\eta \norm{\nabla \mAcomp{t}{3}}_{\eps}^2+\sumkj \norm{\skjcomp{H}\cdot\nabla_H\mAcomp{t}{3}  }_{\eps}^2.
\end{align}
Concerning $L^2_t$ we obtain by \eqref{formula 1 inversione curl}
\begin{align}\label{estimate L2}
    L^2_t&=\sumkmeanj\langle \nabla_H^{\perp}(\skjcomp{H}\cdot \nabla_H \mAcomp{t}{3}),(-\Delta)^{-1}\nabla_H^{\perp}(\smkjcomp{H}\cdot \nabla_H \mAcomp{t}{3})  \rangle_{\eps}\notag\\ & =\sumkmeanj \norm{\skjcomp{H}\cdot\nabla_H \mAcomp{t}{3}}_{\eps}^2.
\end{align}
Thanks to \eqref{formula tre inversione curl} and computations analogous to the ones of \autoref{ito formula oscillation}, we can treat $L^3_t$ obtaining
\begin{align}\label{estimate L3}
    L^3_t&=\sumkoscj  \langle \overline{\skj\cdot\nabla\oBcomp{t}{H}-\oB{t}\cdot\nabla\skjcomp{H}} ,(-\Delta)^{-1}\overline{\smkj\cdot\nabla\oBcomp{t}{H}-\oB{t}\cdot\nabla\smkjcomp{H}}\rangle_{\eps}\notag\\ & = \sumkoscj\langle \nabla_H^{\perp}\left(\overline{\skj\cdot\nabla \oAcomp{t}{3}+\partial_3\skj\cdot\oA{t}} \right), (-\Delta)^{-1}\nabla_H^{\perp}\left(\overline{\smkj\cdot\nabla \oAcomp{t}{3}+\partial_3\smkj\cdot\oA{t}} \right)\rangle_{\eps}\notag\\ & = \sumkoscj \norm{\overline{\skj\cdot\nabla \oAcomp{t}{3}+\partial_3\skj\cdot\oA{t}}}_{\eps}^2\notag\\ & \leq  \sumkoscj \norm{\skj\cdot\nabla \oAcomp{t}{3}+\partial_3\skj\cdot\oA{t}}_{\eps}^2\notag\\ & \leq \left(2\frac{\zeta_{\beta}\zeta_{H,0}}{N^{\beta-2}C_V^2}+O(N^{1-\beta})\right)\norm{\nabla \oAcomp{t}{3}}_{\eps}^2+\left(2\frac{\zeta_{\beta-2}\zeta_{H,0}}{N^{\beta-4}C_V^2}+O(N^{3-\beta})\right)\norm{\oA{t}}_{\eps}^2.
\end{align}
Lastly $L^4_t=0$ thanks to \autoref{Inversion of Lie derivative}. Indeed for each $k\in \Z_0^{3,N}$ such that $k_3=0$ it holds
\begin{align*}
    &\langle \skcomp{1}{H}\cdot\nabla \mBcomp{t}{H}-\mBcomp{t}{H}\cdot\nabla\sk{1},(-\Delta)^{-1}\left(\smkcomp{2}{H}\cdot\nabla \mBcomp{t}{H}-\mBcomp{t}{H}\cdot\nabla\smk{2}\right)\rangle_{\eps}\\ &=\langle \operatorname{curl}\left(\sk{1}\cdot\nabla \mAcomp{t}{3}\right),(-\Delta)^{-1}\operatorname{curl}\left(\smk{2}\cdot\nabla \mAcomp{t}{3}\right)\rangle_{\eps}\\ & =\langle \sk{1}\cdot\nabla \mAcomp{t}{3},\smk{2}\cdot\nabla \mAcomp{t}{3}\rangle_{\eps}=0.
\end{align*}
Relation \eqref{ito formula barA3} then follows combining \eqref{martingale term vbar}, \eqref{estimate L1}, \eqref{estimate L2} and \eqref{estimate L3}. \\
Now we move to the analysis of $\norm{\mBcomp{t}{3}}_{\eps}^2$ which is similar and easier to the previous one. First by It\^o formula we have
\begin{align*}
    d\norm{\mBcomp{t}{3}}_{\eps}^2=d\hat{W}_t+2(-H^1_t+H^2_t+H^3_t+H^4_t)dt,
\end{align*}
where we denote by
\begin{align*}
    d\hat{W}_t&=2\sumkmeanj\langle \skjcomp{H}\cdot\nabla_H\mBcomp{t}{3}-\mBcomp{t}{H}\cdot\nabla_H\skjcomp{3},\mBcomp{t}{3}\rangle_{\eps}dW^{k,j}_t\\ & +2\sumkoscj\langle \skj\cdot\nabla\oBcomp{t}{3}-\oB{t}\cdot\nabla\skjcomp{3},\mBcomp{t}{3}\rangle_{\eps}dW^{k,j}_t,\\
    H^1_t&=\eta \norm{\nabla_H\mBcomp{t}{3}}_{\eps}^2+\sumkj\norm{ \skjcomp{H}\cdot\nabla_H\mBcomp{t}{3}}_{\eps}^2,\\
    H^2_t&=\sumkmeanj\langle \skjcomp{H}\cdot\nabla_H\mBcomp{t}{3}-\mBcomp{t}{H}\cdot\nabla_H\skjcomp{3},\smkjcomp{H}\cdot\nabla_H\mBcomp{t}{3}-\mBcomp{t}{H}\cdot\nabla_H\smkjcomp{3}\rangle_{\eps},\\ H^3_t &=\sumkoscj  \langle \overline{\skj\cdot\nabla\oBcomp{t}{3}-\oB{t}\cdot\nabla\skjcomp{3}} ,\overline{\smkj\cdot\nabla\oBcomp{t}{3}-\oB{t}\cdot\nabla\smkjcomp{3}}  \rangle_{\eps},\\
    H^4_t &=-\langle \operatorname{div}_H\left(\mathcal{R}_{\eps,\gamma} \mBcomp{t}{H}\right),\mBcomp{t}{3}\rangle_{\eps} -\rho\sumkmeancov\langle \sk{1}\cdot\nabla \mBcomp{t}{3},\mB{t}\cdot\nabla\smkcomp{3}{2}\rangle_{\eps}+\langle \mB{t}\cdot\nabla\skcomp{3}{2} ,\smk{1}\cdot\nabla \mBcomp{t}{3}\rangle_{\eps}.
\end{align*}
Analogously to the second part of \autoref{ito formula oscillation} we can treat all the terms above. Let us start from $d\hat{W}_t$. Since the $\skj$ are divergence free we get immediately
\begin{align}\label{estimate martingale Mbar3}
d\hat{W}_t &=-2\sumkmeanj\langle \mBcomp{t}{H}\cdot\nabla_H\skjcomp{3},\mBcomp{t}{3}\rangle_{\eps}dW^{k,j}_t +2\sumkoscj\langle \skj\cdot\nabla\oBcomp{t}{3}-\oB{t}\cdot\nabla\skjcomp{3},\mBcomp{t}{3}\rangle_{\eps}dW^{k,j}_t.    
\end{align}
Thanks to \eqref{mirror symmetry product 0} we get easily by H\"older inequality
\begin{align}\label{estimate H2}
H^2_t &=  \sumkmeanj \norm{\skjcomp{H}\cdot \nabla_H\mBcomp{t}{3}}_{\eps}^2+\sumkmeanj \norm{\mBcomp{t}{H}\cdot \nabla_H\skjcomp{3}}_{\eps}^2\notag\\ & \leq \sumkmeanj \norm{\skjcomp{H}\cdot \nabla_H\mBcomp{t}{3}}_{\eps}^2+\sumkmeanjdue \left\lvert\tkj\right\rvert^2\lvert k\rvert^2 \norm{\mBcomp{t}{H}}_{\eps}^2\notag\\ & \leq \sumkmeanj \norm{\skjcomp{H}\cdot \nabla_H\mBcomp{t}{3}}_{\eps}^2+\sum_{\substack{k_H\in \Z^2_0\\ N\leq\lvert k_H\rvert\leq 2N}}\frac{1}{C_{2,H}^2\lvert k\rvert^{\gamma-2}}\norm{\nabla \mAcomp{t}{3}}_{\eps}^2 \notag \\ & \leq \sumkmeanj \norm{\skjcomp{H}\cdot \nabla_H\mBcomp{t}{3}}_{\eps}^2+\left(\frac{\zeta_{H,\gamma-2}}{N^{\gamma-4}C_{2,H}^2}+O(N^{3-\gamma})\right)\norm{\nabla \mAcomp{t}{3}}_{\eps}^2.
\end{align}
We rewrite $H_t^3$ as 
\begin{align}\label{estimate H3}
H^3_t&=\sumkoscj\norm{\overline{\skj\cdot\nabla\oBcomp{t}{3}-\oB{t}\cdot\nabla\skjcomp{3}}}_{\eps}^2.    
\end{align}
Lastly let us treat $H^4_t$: the term $-\langle \operatorname{div}_H\left(\mathcal{R}_{\eps,\gamma} \mBcomp{t}{H}\right),\mBcomp{t}{3}\rangle_{\eps}$ can be estimated easily due to \autoref{convergence properties matrix} and \eqref{property 2 norm} obtaining by Cauchy-Schwarz inequality
\begin{align*}
-\langle \operatorname{div}_H\left(\mathcal{R}_{\eps,\gamma} \oBcomp{t}{H}\right),\oBcomp{t}{3}\rangle_{\eps}&=  -\langle \mathcal{R}_{\eps,\gamma}\cdot \nabla^{\perp}_H\mAcomp{t}{3},\nabla_H\mBcomp{t}{3}\rangle_{\eps}\\ & \leq  \frac{2\sqrt{2}\lvert \rho\rvert \zeta^N_{H,\gamma/2}  }{N^{\gamma/2-2}C_{1,H}C_{2,H}}\norm{\nabla \mBcomp{t}{3}}_\eps \norm{\nabla \mAcomp{t}{3}}_\eps.
\end{align*}
The other terms can be treated easily by Cauchy-Schwartz inequality. Indeed, for each $\xi_3>0$ and $k\in \Z_0^{3,N}$ such that $k_3=0$ it holds
\begin{align*}
\langle \sk{1}\cdot\nabla \mBcomp{t}{3},\mB{t}\cdot\nabla\smkcomp{3}{2}\rangle_{\eps}\leq \frac{\norm{\sk{1}\cdot\nabla \mBcomp{t}{3}}_{\eps}^2}{2\xi_3}+\frac{\xi_3 \norm{\mBcomp{t}{H}\cdot\nabla_H\skcomp{3}{2}}_{\eps}^2}{2}.    
\end{align*}
In conclusion arguing as when treating $H^2_t$, so far, we have proved that for each $\xi_3>0$
\begin{align}\label{estimate H4 prefinal}
   H^4_t&\leq \frac{2\sqrt{2}\lvert \rho\rvert \zeta^N_{H,\gamma/2} }{N^{\gamma/2-2}C_{1,H}C_{2,H}}\norm{\nabla\mBcomp{t}{3}}_{\eps}\norm{\nabla \mAcomp{t}{3}}_{\eps}+\frac{\lvert \rho\rvert}{\xi_3} \sumkmeanjuno \norm{\skj\cdot \nabla \mBcomp{t}{3}}^2_{\eps}\notag\\ &+\lvert \rho\rvert\xi_3 \left(\frac{\zeta_{H,\gamma-2}}{N^{\gamma-4}C_{2,H}^2}+O(N^{3-\gamma})\right)\norm{\nabla \mAcomp{t}{3}}_{\eps}^2.
\end{align}
Let us analyze $\sumkmeanjuno \norm{\skj\cdot \nabla \mBcomp{t}{3}}^2_{\eps}$. Arguing as in the second part of \autoref{ito formula oscillation} we have
\begin{align*}
\sumkmeanjuno\norm{\skj\cdot\nabla\mBcomp{t}{3}}^2_{\eps}& \leq \left(\frac{\zeta_{H,2}}{C_{1,H}^2}+O(N^{-1})\right)\norm{\nabla\mBcomp{t}{3}}_{\eps}^2.
\end{align*}
In conclusion, applying Young's inequality to the first term of \eqref{estimate H4 prefinal}, it holds
\begin{align}\label{estimate H_4}
H^4_t&\leq \frac{\lvert \rho\rvert}{\xi_3} \left(\frac{\zeta_{H,2}}{C_{1,H}^2}+\sqrt{2}\frac{\zeta_{H,\gamma/2}}{C_{1,H}C_{2,H}}+O(N^{-1})\right)\norm{\nabla\mBcomp{t}{3}}_{\eps}^2\notag\\ &+\lvert \rho\rvert\xi_3 \left(\frac{\zeta_{H,\gamma-2}}{N^{\gamma-4}C_{2,H}^2}+\sqrt{2}\frac{\zeta_{H,\gamma/2}}{N^{\gamma-4}C_{1,H}C_{2,H}}+O(N^{3-\gamma})\right)\norm{\nabla \mAcomp{t}{3}}_{\eps}^2.    
\end{align}
Combining \eqref{estimate martingale Mbar3}, \eqref{estimate H2}, \eqref{estimate H3} and \eqref{estimate H_4} we get relation \eqref{ito formula barM3}.
\end{proof}
\autoref{ito formula oscillation} and \autoref{ito formula mean} imply the following version of \eqref{dream relation}:
\begin{proposition}\label{proposition a priori estimate 1}
Assuming \autoref{HP noise}, for $N>0$ large enough
\begin{align}\label{apriori estimate barA3 CL LH}
\expt{\sup_{t\in [0,T]}\norm{\mAcomp{t}{3}}_{\eps}^2}+\expt{\int_0^T\norm{\nabla\mAcomp{s}{3}}_{\eps}^2 ds} & \lesssim {\norm{B_0^{3,\eps}}_{\eps}^2}+{\norm{\mAcomp{0}{3}}_{\eps}^2}+{\norm{\oA{0}}_{\eps}^2}.\\
\expt{\sup_{t\in [0,T]}\norm{\mBcomp{t}{3}}_{\eps}^2}+\expt{\int_0^T\norm{\nabla\mBcomp{s}{3}}_{\eps}^2 ds} & \lesssim {\norm{B_0^{3,\eps}}_{\eps}^2}+{\norm{\mAcomp{0}{3}}_{\eps}^2}+{\norm{\oA{0}}_{\eps}^2} \label{a priori estimate barB3 CL LH}.\\
\label{apriori estimate A' CL LH}
\expt{\sup_{t\in [0,T]}\norm{\oA{t}}_{\eps}^2}+\expt{\int_0^T\norm{\nabla\oA{s}}_{\eps}^2 ds}&\lesssim {\norm{B_0^{3,\eps}}_{\eps}^2}+{\norm{\mAcomp{0}{3}}_{\eps}^2}+{\norm{\oA{0}}_{\eps}^2}.\\  
\label{apriori estimate B3' CL LH}
\expt{\sup_{t\in [0,T]}\norm{\oBcomp{t}{3}}_{\eps}^2}+\expt{\int_0^T\norm{\nabla\oBcomp{s}{3}}_{\eps}^2 ds}&\lesssim {\norm{B_0^{3,\eps}}_{\eps}^2}+{\norm{\mAcomp{0}{3}}_{\eps}^2}+{\norm{\oA{0}}_{\eps}^2}.  
\end{align}
The hidden constants above all depend only on $\eta,\beta,\gamma,C_V,C_{1,H},C_{2,H},\lvert \rho\rvert,T$.
\end{proposition}
\begin{proof}
In order to simplify the notation, we will denote by $C_{par}$ several constants depending from $\eta$, $\lvert \rho\rvert$, $\gamma$, $\beta$, $C_{1,H}$, $C_{2,H}$, $C_V$, $T$, possibly changing their values line by line but independent from $N$ in this proof. Let us introduce a new parameter $\eta_{aux}$ such that
\begin{align*}
 \frac{\zeta_{H,0}}{C_{1,H}^2}<\eta_{aux}< \eta.   
\end{align*}
Recalling the discussion at the beginning of \autoref{section a priori estimates}, it is easy to show that for each $\eps>0$ the stochastic integrals appearing in \eqref{ito formula A'}, \eqref{ito formula M3'}, \eqref{ito formula barA3}, \eqref{ito formula barM3} are true martingales. 
Let us start considering the  expected value of \eqref{ito formula A'}. We get for each $t\in [0,T]$
\begin{align*}
&\expt{\norm{\oA{t}}_{\eps}^2} +2\eta\expt{\int_0^t\norm{\nabla \oA{s}}^2_{\eps}ds}\notag\\& \leq {\norm{\oA{0}}_{\eps}^2} +\left(6\frac{\zeta_{\beta}\zeta_{H,0} }{N^{\beta-2}C_V^2}+O(N^{1-\beta})\right)\expt{\int_0^t\lVert \nabla \mA{s}\rVert_{\eps}^2ds}+\left(6\frac{\zeta_{\beta-2}\zeta_{H,0}}{N^{\beta-4}C_V^2}+O(N^{3-\beta})\right)\expt{\int_0^t\norm{\mA{s}}_{\eps}^2 ds}\notag\\ & +2\left(\frac{\zeta_{H,0}}{C_{1,H}^2}+\frac{\lvert \rho\rvert}{\xi_1}\frac{\left(\zeta_{H,0}+\zeta_{H,2}\right)}{C_{1,H}^2}+O(N^{-1})\right)\expt{\int_0^t\norm{\nabla\oA{s}}^2_{\eps}ds}.      
\end{align*}
Therefore, if we choose $\xi_1$ large enough, there exists $N_0\in \N$ such that for each $N\geq N_0$ it holds
\begin{align}\label{absorbing estimate}
    \frac{\zeta_{H,0}}{C_{1,H}^2}+\frac{\lvert \rho\rvert}{\xi_1}\frac{\left(\zeta_{H,0}+\zeta_{H,2}\right)}{C_{1,H}^2}+O(N^{-1})\leq \eta_{aux}
\end{align}
and we get 
\begin{align}\label{estimate 1 a priori 1}
\expt{\norm{\oA{t}}_{\eps}^2} +(\eta-\eta_{aux})\expt{\int_0^t\norm{\nabla \oA{s}}^2_{\eps}ds}& \leq {\norm{\oA{0}}_{\eps}^2} +\left(6\frac{\zeta_{\beta}\zeta_{H,0} }{N^{\beta-2}C_V^2}+O(N^{1-\beta})\right)\expt{\int_0^t\lVert \nabla \mA{s}\rVert_{\eps}^2ds}\notag\\&+\left(6\frac{\zeta_{\beta-2}\zeta_{H,0}}{N^{\beta-4}C_V^2}+O(N^{3-\beta})\right)\expt{\int_0^t\norm{\mA{s}}_{\eps}^2 ds}.
\end{align}
Now we take the expected value of \eqref{ito formula barA3} at time $t\in [0,T]$ obtaining by Poincaré inequality \eqref{poincare thin layer}
\begin{align*}
\expt{\norm{\mAcomp{t}{3}}_{\eps}^2}+2\eta\expt{\int_0^t\norm{\nabla\mAcomp{s}{3}}_{\eps}^2 ds}&\leq {\norm{\mAcomp{0}{3}}_{\eps}^2}+4\left(\frac{\zeta_{\beta}\zeta_{H,0}}{N^{\beta-2}C_V^2}+\frac{\zeta_{\beta-2}\zeta_{H,0}}{N^{\beta-2}C_V^2}+O(N^{1-\beta})\right)\expt{\int_0^t \norm{\nabla \oA{s}}_{\eps}^2 ds}.  
\end{align*}
Plugging \eqref{estimate 1 a priori 1} in the relation above we get
\begin{align}\label{estimate 1 a priori 3}
&\expt{\norm{\mAcomp{t}{3}}_{\eps}^2}+2\eta\expt{\int_0^t\norm{\nabla\mAcomp{s}{3}}_{\eps}^2 ds}\notag\\ &\leq {\norm{\mAcomp{0}{3}}_{\eps}^2}+\left(\frac{C_{par}}{N^{\beta-2}}+O(N^{1-\beta})\right)\times\notag\\ & \quad \quad \left( \norm{\oA{0}}_{\eps}^2+\left(\frac{C_{par}}{N^{\beta-2}}+O(N^{1-\beta})\right)\expt{\int_0^t\lVert \nabla \mA{s}\rVert_{\eps}^2ds}+\left(\frac{C_{par}}{N^{\beta-4}}+O(N^{3-\beta})\right)\expt{\int_0^t\norm{\mA{s}}_{\eps}^2 ds}\right).   
\end{align}
In order to simplify the expression above, let us observe that since $\mA{t}$ is independent of the third component, Poincaré inequality with constant equal to $1$ holds. Therefore we have
\begin{align}\label{poincare 2Dtorus}
\expt{\int_0^t\norm{\mA{s}}_{\eps}^2 ds}\leq \expt{\int_0^t\norm{\nabla\mA{s}}_{\eps}^2 ds}.   
\end{align}
Secondly, thanks to \autoref{equivalence norms} it holds
\begin{align}\label{equivalence norm}
\norm{\nabla\mA{t}}_{\eps}^2&=\norm{\mB{t}}_{\eps}^2=\norm{\mBcomp{t}{3}}_{\eps}^2+\norm{\mBcomp{t}{H}}_{\eps}^2=\norm{\mBcomp{t}{3}}_{\eps}^2+\norm{\nabla\mAcomp{t}{3}}_{\eps}^2,\\
\label{poincare norm }
\norm{\mA{t}}_{\eps}^2&=\norm{(-\Delta)^{-1/2}\mB{t}}_{\eps}^2=\norm{(-\Delta)^{-1/2}\mBcomp{t}{3}}_{\eps}^2+\norm{(-\Delta)^{-1/2}\mBcomp{t}{H}}_{\eps}^2\leq \norm{\mBcomp{t}{3}}_{\eps}^2+\norm{\mAcomp{t}{3}}_{\eps}^2.
\end{align}
Combining \eqref{poincare 2Dtorus} and \eqref{equivalence norm} in relation \eqref{estimate 1 a priori 3} we get
\begin{align*}
\expt{\norm{\mAcomp{t}{3}}_{\eps}^2}+2\eta\expt{\int_0^t\norm{\nabla\mAcomp{s}{3}}_{\eps}^2 ds} & \leq {\norm{\mAcomp{0}{3}}_{\eps}^2}+\frac{C_{par}}{N^{2-\beta}}{\norm{\oA{0}}_{\eps}^2}\notag\\ &
+\left(\frac{C_{par}}{N^{\beta-2}}+O(N^{1-\beta})\right)\expt{\int_0^t\lVert \nabla \mAcomp{s}{3}\rVert_{\eps}^2ds}\\ & +\left(\frac{C_{par}}{N^{\beta-2}}+O(N^{1-\beta})\right)\expt{\int_0^t\lVert \nabla \mBcomp{s}{3}\rVert_{\eps}^2ds}.
\end{align*}
Therefore, choosing $N_0\in\N$ large enough such that for each $N\geq N_0$ it holds
\begin{align*}
    \frac{C_{par}}{N^{\beta-2}}+O(N^{1-\beta})\leq \frac{\eta}{2},
\end{align*}
we get
\begin{align}\label{estimate 4 proposition 1}
\expt{\norm{\mAcomp{t}{3}}_{\eps}^2}+\eta\expt{\int_0^t\norm{\nabla\mAcomp{s}{3}}_{\eps}^2 ds} & \leq {\norm{\mAcomp{0}{3}}_{\eps}^2}+\frac{C_{par}}{N^{2-\beta}}\norm{\oA{0}}_{\eps}^2\notag\\ & +\left(\frac{C_{par}}{N^{2-\beta}}+O(N^{1-\beta})\right)\expt{\int_0^t\lVert \nabla \mBcomp{s}{3}\rVert_{\eps}^2ds}.   
\end{align}
 Now we consider the expected value of \eqref{ito formula M3'}:
\begin{align}\label{aux estimate expected value}
&\expt{\norm{\oBcomp{t}{3}}^2_{\eps}}+2\eta\expt{\int_0^t\norm{\nabla \oBcomp{s}{3}}^2_{\eps}ds}+2\sumkoscj \expt{\int_0^t\norm{\overline{\skj\cdot \nabla \oBcomp{s}{3}-\oB{s}\cdot\nabla\skjcomp{3}}}_{\eps}^2 ds}\notag\\& \leq \norm{\oBcomp{0}{3}}^2_{\eps}+2\frac{\lvert \rho\rvert}{\xi_2}\left(\frac{\zeta_{H,2}}{ C_{1,H}^2}+\sqrt{2}\frac{\zeta_{H,\gamma/2}}{C_{1,H}C_{2,H}}+O(N^{-1})\right)\expt{\int_0^t \norm{\nabla \oBcomp{s}{3}}_{\eps}^2 ds} \notag\\ & +2\left(\frac{\zeta_{H,\gamma-2}}{N^{\gamma-4}C_{2,H}^2}+2\frac{\zeta_{\beta-2}\zeta_{H,0}}{N^{\beta-4}C_V^2}+\frac{\lvert \rho\rvert \xi_2}{N^{\gamma-4}}\left(\frac{\zeta_{H,\gamma-2}}{C_{2,H}^2}+\sqrt{2}\frac{\zeta_{H,\gamma/2}}{C_{1,H}C_{2,H}}\right)+O(N^{3-\gamma}+N^{3-\beta})\right)\expt{\int_0^t \norm{\nabla\oA{s}}^2_{\eps}ds}\notag\\ & +2\sumkoscj \expt{\int_0^t\norm{\skj\cdot\nabla\mBcomp{s}{3}}^2_{\eps}ds}+\left(4\frac{\zeta_{\beta-2}\zeta_{H,0}}{N^{\beta-4}C_{V}^2}+O(N^{3-\beta})\right)\expt{\int_0^t\norm{\mB{s}}^2_{\eps}ds}.
\end{align} 
Considering $\xi_2>0$ large enough, we can find $N_0\in \N$ such that for each $N\geq N_0$ it holds 
\begin{align*}
\frac{\lvert \rho\rvert}{\xi_2}\left(\frac{\zeta_{H,2}}{ C_{1,H}^2}+\sqrt{2}\frac{\zeta_{H,\gamma/2}}{C_{1,H}C_{2,H}}+O(N^{-1})\right)\leq \frac{\eta}{2}.
\end{align*}
Therefore \eqref{aux estimate expected value} reduces to
\begin{align*}
&\expt{\norm{\oBcomp{t}{3}}^2_{\eps}}+\eta\expt{\int_0^t\norm{\nabla \oBcomp{s}{3}}^2_{\eps}ds}+2\sumkoscj \expt{\int_0^t\norm{\overline{\skj\cdot \nabla \oBcomp{s}{3}-\oB{s}\cdot\nabla\skjcomp{3}}}_{\eps}^2 ds}\notag\\& \leq \norm{\oBcomp{0}{3}}^2_{\eps}+2\sumkoscj \expt{\int_0^t\norm{\skj\cdot\nabla\mBcomp{s}{3}}^2_{\eps}ds}+\left(\frac{C_{par}}{N^{\beta-4}}+O(N^{3-\beta})\right)\expt{\int_0^t\norm{\mB{s}}^2_{\eps}ds} \notag\\ & +\left(C_{par}\left(\frac{1}{N^{\gamma-4}}+\frac{1}{N^{\beta-4}}\right)+O(N^{3-\gamma}+N^{3-\beta})\right)\expt{\int_0^t \norm{\nabla\oA{s}}^2_{\eps}ds}\notag. 
\end{align*}
Combining \eqref{estimate 1 a priori 1} and the relation above we get
\begin{align}\label{estimate 2 a priori 1}
&\expt{\norm{\oBcomp{t}{3}}^2_{\eps}}+\eta\expt{\int_0^t\norm{\nabla \oBcomp{s}{3}}^2_{\eps}ds}+2\sumkoscj\expt{\int_0^t \norm{\overline{\skj\cdot \nabla \oBcomp{s}{3}-\oB{s}\cdot\nabla\skjcomp{3}}}_{\eps}^2 ds}\notag\\ & \leq {\norm{\oBcomp{0}{3}}^2_{\eps}}  +2\sumkoscj\expt{\int_0^t\norm{\skj\cdot\nabla\mBcomp{s}{3}}^2_{\eps}ds}+\left(\frac{C_{par}}{N^{\beta-4}}+O(N^{3-\beta})\right)\expt{\int_0^t\norm{\mB{s}}^2_{\eps}ds}\notag\\ & +\left(C_{par}\left(\frac{1}{N^{\gamma-4}}+\frac{1}{N^{\beta-4}}\right)+O(N^{3-\gamma}+N^{3-\beta})\right)\times\notag\\ & \quad\quad \left({\norm{\oA{0}}_{\eps}^2} +\left(\frac{C_{par}}{N^{\beta-2}}+O(N^{1-\beta})\right)\expt{\int_0^t\lVert \nabla \mA{s}\rVert_{\eps}^2ds}+\left(\frac{C_{par}}{N^{\beta-4}}+O(N^{3-\beta})\right)\expt{\int_0^t\norm{\mA{s}}_{\eps}^2 ds}\right).
\end{align}
Combining \eqref{estimate 2 a priori 1}, \eqref{poincare 2Dtorus}, \eqref{equivalence norm} and \eqref{poincare norm } we obtain 
\begin{align}\label{formula tremenda B3'}
&\expt{\norm{\oBcomp{t}{3}}^2_{\eps}}+\eta\expt{\int_0^t\norm{\nabla \oBcomp{s}{3}}^2_{\eps}ds}+2\sumkoscj\expt{\int_0^t\norm{\overline{\skj\cdot \nabla \oBcomp{s}{3}-\oB{s}\cdot\nabla\skjcomp{3}}}_{\eps}^2 ds}\notag\\ & \leq {\norm{\oBcomp{0}{3}}^2_{\eps}}+ C_{par}\left(\frac{1}{N^{\gamma-4}}+\frac{1}{N^{\beta-4}}\right){\norm{\oA{0}}_{\eps}^2} +2\sumkoscj\expt{\int_0^t\norm{\skj\cdot\nabla\mBcomp{s}{3}}^2_{\eps}ds}\notag\\ 
&+\left(C_{par}\left(\frac{1}{N^{\gamma-4}}+\frac{1}{N^{\beta-4}}\right)+O(N^{3-\beta}+N^{3-\gamma})\right)\times\notag\\&\quad \quad \quad\left(\expt{\int_0^t \norm{\mBcomp{s}{3}}_{\eps}^2 ds}+\expt{\int_0^t \norm{\nabla\mAcomp{s}{3}}_{\eps}^2 ds}+\expt{\int_0^t \norm{\mAcomp{s}{3}}_{\eps}^2 ds}\right).    
\end{align}
Therefore, thanks to \eqref{estimate 4 proposition 1} we get
\begin{align}\label{formula finale B3'}
&\expt{\norm{\oBcomp{t}{3}}^2_{\eps}}+2\eta\expt{\int_0^t\norm{\nabla \oBcomp{s}{3}}^2_{\eps}ds}+2\sumkoscj\expt{\int_0^t \norm{\overline{\skj\cdot \nabla \oBcomp{s}{3}-\oB{s}\cdot\nabla\skjcomp{3}}}_{\eps}^2 ds}\notag\\ & \leq {\norm{\oBcomp{0}{3}}^2_{\eps}}+C_{par}\left(\frac{1}{N^{\gamma-4}}+\frac{1}{N^{\beta-4}}\right) \left({\norm{\oA{0}}_{\eps}^2}+{\norm{\mAcomp{0}{3}}_{\eps}^2} \right)  +2\sumkoscj\expt{\int_0^t\norm{\skj\cdot\nabla\mBcomp{s}{3}}^2_{\eps}ds}\notag\\ 
&+\left(C_{par}\left(\frac{1}{N^{\gamma-4}}+\frac{1}{N^{\beta-4}}\right)+O(N^{3-\gamma}+N^{3-\beta})\right)\expt{\int_0^t \norm{\mBcomp{s}{3}}_{\eps}^2 ds}\notag\\ & +\left( \frac{C_{par}}{N^{2-\beta}}+O(N^{1-\beta}) \right)\expt{\int_0^t\lVert \nabla \mBcomp{s}{3}\rVert_{\eps}^2ds}.
\end{align}
Lastly we take the expected value of \eqref{ito formula barM3} obtaining 
\begin{align}\label{auxiliary expected value 2}
&\expt{\norm{\mBcomp{t}{3}}_{\eps}^2}+2\eta\expt{\int_0^t\norm{\nabla \mBcomp{s}{3}}_{\eps}^2 ds}+2\sumkoscj\expt{\int_0^t\norm{\skj\cdot\nabla\mBcomp{s}{3} }_{\eps}^2 ds}\notag\\ &\leq \norm{\mBcomp{0}{3}}_{\eps}^2+ 2\left(\frac{\zeta_{H,\gamma-2}}{N^{\gamma-4}C_{2,H}^2}+\lvert \rho\rvert\xi_3 \frac{\zeta_{H,\gamma-2}}{N^{\gamma-4}C_{2,H}^2}+\sqrt{2}\lvert \rho\rvert \xi_3\frac{\zeta_{H,\gamma/2}}{N^{\gamma-4}C_{1,H}C_{2,H}}+O(N^{3-\gamma})\right)\expt{\norm{\nabla \mAcomp{s}{3}}_{\eps}^2 ds}\notag\\ & +2\frac{\lvert \rho\rvert}{\xi_3} \left(\frac{\zeta_{H,2}}{C_{1,H}^2}+\sqrt{2}\frac{\zeta_{H,\gamma/2}}{C_{1,H}C_{2,H}}+O(N^{-1})\right)\expt{\int_0^t \norm{\nabla\mBcomp{s}{3}}_{\eps}^2 ds}\notag\\ & +2\sumkoscj\expt{\int_0^t\norm{\overline{\skj\cdot\nabla\oBcomp{s}{3}-\oB{s}\cdot\nabla\skjcomp{3}}}_{\eps}^2ds}.
\end{align}
Considering $\xi_3>0$ large enough we can find $N_0\in \N$ such that for each $N\geq N_0$ it holds 
\begin{align*}
\frac{\lvert \rho\rvert}{\xi_3} \left(\frac{\zeta_{H,2}}{C_{1,H}^2}+\sqrt{2}\frac{\zeta_{H,\gamma/2}}{C_{1,H}C_{2,H}}+O(N^{-1})\right) \leq \frac{\eta}{2}.
\end{align*}
Therefore \eqref{auxiliary expected value 2} reduces to
\begin{align*}
&\expt{\norm{\mBcomp{t}{3}}_{\eps}^2}+\eta\expt{\int_0^t\norm{\nabla \mBcomp{s}{3}}_{\eps}^2 ds}+2\sumkoscj\expt{\int_0^t\norm{\skj\cdot\nabla\mBcomp{s}{3} }_{\eps}^2 ds}\notag\\ & \leq \norm{\mBcomp{0}{3}}_{\eps}^2+\left(\frac{C_{par}}{N^{\gamma-4}}+O(N^{3-\gamma})\right)\expt{\norm{\nabla \mAcomp{s}{3}}_{\eps}^2 ds} +2\sumkoscj\expt{\int_0^t\norm{\overline{\skj\cdot\nabla\oBcomp{s}{3}-\oB{s}\cdot\nabla\skjcomp{3}}}_{\eps}^2ds}.
\end{align*}
Combining relations \eqref{estimate 4 proposition 1} and \eqref{formula finale B3'} in the formula above we obtain
\begin{align*}
   \expt{\norm{\mBcomp{t}{3}}_{\eps}^2}+\eta\expt{\int_0^t\norm{\nabla \mBcomp{s}{3}}_{\eps}^2ds }& \leq {\norm{\mBcomp{0}{3}}_{\eps}^2}+{\norm{\oBcomp{0}{3}}^2_{\eps}}+C_{par}\left(\frac{1}{N^{\gamma-4}}+\frac{1}{N^{\beta-4}}\right) \left({\norm{\oA{0}}_{\eps}^2}+{\norm{\mAcomp{0}{3}}_{\eps}^2} \right)\notag\\ 
&+\left(C_{par}\left(\frac{1}{N^{\gamma-4}}+\frac{1}{N^{\beta-4}}\right)+O(N^{3-\gamma}+N^{3-\beta})\right)\expt{\int_0^t \norm{\mBcomp{s}{3}}_{\eps}^2 ds}\notag\\ &+\left(\frac{C_{par}}{N^{\beta-2}}+O(N^{1-\beta})\right)\expt{\int_0^t\lVert \nabla \mBcomp{s}{3}\rVert_{\eps}^2ds}.
\end{align*}
Choosing $N_0\in \N$ large enough such that for each $N\geq N_0$ it holds
\begin{align*}
\frac{C_{par}}{N^{\beta-2}}+O(N^{1-\beta}) \leq \frac{\eta}{2}.  
\end{align*}
Analogously, the term which multiplies $\expt{\int_0^t\lVert \nabla \mBcomp{s}{3}\rVert_{\eps}^2ds}$ and the initial conditions can be bounded uniformly in $N$ depending only from $\eta, \beta,\gamma, C_V,C_{1,H}, C_{2,H},\lvert \rho\rvert, T$. Therefore previous expression simplifies and becomes, due to \autoref{properties of M N},
\begin{align}\label{final form BarB3}
\expt{\norm{\mBcomp{t}{3}}_{\eps}^2}+\eta\expt{\int_0^t\norm{\nabla \mBcomp{s}{3}}_{\eps}^2ds } & \lesssim_{\eta,\beta,\gamma,C_V,C_{1,H},C_{2,H},\lvert \rho\rvert,T} {\norm{\mBcomp{0}{3}}_{\eps}^2}+{\norm{\oBcomp{0}{3}}^2_{\eps}}\notag\\ & +\left(\frac{1}{N^{\beta-4}}+\frac{1}{N^{\gamma-4}}\right)\left({\norm{\mAcomp{0}{3}}_{\eps}^2}+{\norm{\oA{0}}_{\eps}^2}+\expt{\int_0^t \norm{\mBcomp{s}{3}}_{\eps}^2 ds}\right).     
\end{align}
Applying Grownall Lemma to \eqref{final form BarB3}
we obtain that for each $t\in [0,T]$ it holds
\begin{align}\label{estimate BarB3}
 \expt{\norm{\mBcomp{t}{3}}_{\eps}^2}+&\expt{\int_0^t\norm{\nabla \mBcomp{s}{3}}_{\eps}^2ds }\lesssim_{\eta,\beta,\gamma,C_V,C_{1,H},C_{2,H},\lvert \rho\rvert,T} {\norm{B^{3,\eps}_0}_{\eps}^2}+{\norm{\mAcomp{0}{3}}_{\eps}^2}+{\norm{\oA{0}}_{\eps}^2}.   
\end{align}
Exploiting the estimates implied by \eqref{estimate BarB3} in \eqref{estimate 4 proposition 1} and \eqref{formula finale B3'}, we get 
\begin{align}
\label{estimate barA3}
 \expt{\norm{\mAcomp{t}{3}}_{\eps}^2}+\expt{\int_0^t\norm{\nabla \mAcomp{s}{3}}_{\eps}^2ds }& \lesssim_{\eta,\beta,\gamma,C_V,C_{1,H},C_{2,H},\lvert \rho\rvert,T} {\norm{B^{3,\eps}_0}_{\eps}^2}+{\norm{\mAcomp{0}{3}}_{\eps}^2}+{\norm{\oA{0}}_{\eps}^2},\\
\label{estimate B3'}
 \expt{\norm{\oBcomp{t}{3}}_{\eps}^2}+\expt{\int_0^t\norm{\nabla \oBcomp{s}{3}}_{\eps}^2ds }& \lesssim_{\eta,\beta,\gamma,C_V,C_{1,H},C_{2,H},\lvert \rho\rvert,T} {\norm{B^{3,\eps}_0}_{\eps}^2}+{\norm{\mAcomp{0}{3}}_{\eps}^2}+{\norm{\oA{0}}_{\eps}^2}.    
\end{align}
Combining \eqref{equivalence norm}, \eqref{poincare norm }, \eqref{estimate BarB3} and \eqref{estimate barA3} we get 
\begin{align}
\label{estimate A'}
 \expt{\norm{\oA{t}}_{\eps}^2}+&\expt{\int_0^t\norm{\nabla \oA{s}}_{\eps}^2ds }\lesssim_{\eta,\beta,\gamma,C_V,C_{1,H},C_{2,H},\lvert \rho\rvert,T} {\norm{B^{3,\eps}_0}_{\eps}^2}+{\norm{\mAcomp{0}{3}}_{\eps}^2}+{\norm{\oA{0}}_{\eps}^2}.   
\end{align}
We omit the easy details in order to obtain \eqref{estimate barA3}, \eqref{estimate B3'} and \eqref{estimate A'}. Having already shown \eqref{estimate BarB3}, \eqref{estimate barA3}, \eqref{estimate B3'}, \eqref{estimate A'}, in order to conclude the proof we need only to put the supremum in time inside the expected value. Since this is a standard procedure we do it only $\oA{t}$, the other being analogous. Considering the expected value of the supremum in time of \eqref{ito formula A'} for $\xi_1=1$ we obtain, thanks to \eqref{estimate BarB3} and \eqref{estimate barA3} 
\begin{align}\label{supremum in time A' step 1}
\expt{\operatorname{sup}_{t\in [0,T]}\norm{\oA{t}}_{\eps}^2} & \leq {\norm{\oA{0}}_{\eps}^2}+2\left(\frac{\zeta_{H,0}}{C_{1,H}^2}+\lvert \rho\rvert\frac{\left(\zeta_{H,0}+\zeta_{H,2}\right)}{C_{1,H}^2}+O(N^{-1})\right)\expt{\int_0^T\norm{\nabla\oA{s}}^2_{\eps}ds} \notag\\ & +\left(6\frac{\zeta_{\beta}\zeta_{H,0} }{N^{\beta-2}C_V^2}+O(N^{1-\beta})\right)\expt{\int_0^T\lVert \nabla \mA{s}\rVert_{\eps}^2ds}\notag\\ & +\left(6\frac{\zeta_{\beta-2}\zeta_{H,0}}{N^{\beta-4}C_V^2}+O(N^{3-\beta})\right)\expt{\int_0^T\norm{\mA{s}}_{\eps}^2 ds} +\expt{\operatorname{sup}_{t\in [0,T]}\lvert W'_t\rvert}\notag\\ & \leq C_{par} \left({\norm{B_0^{3,\eps}}_{\eps}^2}+{\norm{\mAcomp{0}{3}}_{\eps}^2}+{\norm{\oA{0}}_{\eps}^2}\right)+\expt{\operatorname{sup}_{t\in [0,T]}\lvert W'_t\rvert}, 
\end{align}
$W'_t$ being defined in the first part of \autoref{ito formula oscillation}. Therefore we are left to analyze $\expt{\operatorname{sup}_{t\in [0,T]}\lvert W'_t\rvert}$ which can be treated by Burkholder-Davis-Gundy inequality, H\"older and Young inequalities obtaining
\begin{align}\label{supremum in time A' 2}
\expt{\operatorname{sup}_{t\in [0,T]}\lvert W'_t\rvert}& \lesssim_{\lvert \rho\rvert} \expt{\left(\sumkj \int_0^T \langle D\skj\oA{t},\oA{t} \rangle^2_{\eps}dt\right)^{1/2}}\notag\\ &  + \expt{\left(\sumkoscj \int_0^T  \langle \skj\cdot \nabla \mA{t}+D\skj\mA{t},\oA{t} \rangle^2_{\eps}dt\right)^{1/2}}\notag\\ & \leq \expt{\operatorname{sup}_{t\in [0,T]}\norm{\oA{t}}_\eps\left(\sumkj \int_0^T \norm{ D\skj\oA{t}}^2_{\eps}dt\right)^{1/2}}\notag\\ &  + \expt{\operatorname{sup}_{t\in [0,T]}\norm{\oA{t}}_\eps\left(\sumkoscj \int_0^T  \norm{ \skj\cdot \nabla \mA{t}+D\skj\mA{t} }^2_{\eps}dt\right)^{1/2}}\notag\\ & \leq \frac{1}{2}\expt{\operatorname{sup}_{t\in [0,T]}\norm{\oA{t}}_{\eps}^2}+ C_{par}\expt{\sumkj \int_0^T \norm{ D\skj\oA{t}}^2_{\eps}dt}\notag\\ &  + C_{par}\expt{\sumkoscj \int_0^T  \norm{ \skj\cdot \nabla \mA{t}+D\skj\mA{t}}^2_{\eps}dt}\notag\\ & \leq  C_{par} \left({\norm{B_0^{3,\eps}}_{\eps}^2}+{\norm{\mAcomp{0}{3}}_{\eps}^2}+{\norm{\oA{0}}_{\eps}^2}\right)+\frac{1}{2}\expt{\operatorname{sup}_{t\in [0,T]}\norm{\oA{t}}_{\eps}^2},
\end{align}
where in the last step we combine the computations in order to obtain \eqref{estimate I2}, \eqref{estimate I_3}, \eqref{estimate I_4} and the bounds \eqref{estimate BarB3}, \eqref{estimate barA3}, \eqref{estimate A'}. Combining relations \eqref{supremum in time A' step 1} and \eqref{supremum in time A' 2}, the a priori bound \eqref{apriori estimate A' CL LH} follows. 
\end{proof}
As a corollary of \eqref{apriori estimate barA3 CL LH}, \eqref{a priori estimate barB3 CL LH} we get 
\begin{corollary}\label{corollary compactness in space}
Assuming \autoref{HP noise} and there exists $N_0\in \N$ such that for each $N\geq N_0$ it holds  \begin{align}\label{HP passage to the limit}
{\norm{B_0^{3,\eps}}_{\eps}^2}+{\norm{\mAcomp{0}{3}}_{\eps}^2}+{\norm{\oA{0}}_{\eps}^2}\lesssim \eps,
\end{align} then there exists $N_1\in \N$ large enough such that for each $N\geq N_1$ we have   
\begin{align}\label{space compactness mean}
\expt{\sup_{t\in [0,T]}\left(\norm{\mBcomp{t}{3}}^2+\norm{\mAcomp{t}{3}}^2 \right)+\int_0^T\norm{\nabla\mBcomp{s}{3}}^2 ds +\int_0^T\norm{\nabla\mAcomp{s}{3}}^2 ds } \lesssim_{} 1,
\end{align}
where the hidden constants depend only on $\eta,\beta,\gamma,C_V,C_{1,H},C_{2,H},\lvert \rho\rvert,T$

\end{corollary}
\begin{proof}
    The thesis follows immediately combining \autoref{proposition a priori estimate 1} and assumption \eqref{HP passage to the limit}.
\end{proof}

\section{Proof of \autoref{main Theorem}}\label{sec proof main thm}
Let $\ke:=\eta+\eta^{\eps}_T\geq \eta>0$ and let us introduce the following notation
\begin{align*}
Z^{1,\eps}_t&=\sumkmeanj\int_0^t e^{(t-s)\ke\Delta}\left(\skj\cdot\nabla\mBcomp{s}{3}-\mB{s}\cdot\nabla\skjcomp{3}\right)dW^{k,j}_s\\ &+\sumkoscj\int_0^t  e^{(t-s)\ke\Delta} \overline{\skj\cdot\nabla\oBcomp{s}{3}-\oB{s}\cdot\nabla\skjcomp{3}} dW^{k,j}_s,  \\
Z^{2,\eps}_t&=\sumkmeanj\int_0^t e^{(t-s)\ke\Delta}\skj\cdot\nabla\mAcomp{s}{3}dW^{k,j}_s\\ &+\sumkoscj\int_0^t  e^{(t-s)\ke\Delta}Q\left[\overline{\skj\cdot\nabla\oAcomp{s}{3}+\partial_3\skj\cdot \oA{t}}\right] dW^{k,j}_s.
\end{align*}
The stochastic integrals above are well defined thanks to the regularity properties of $\mB{t},\ \oB{t},\mA{t},\oA{t}$ and the summability properties of the coefficients in \autoref{HP noise}.
The first step in order to prove \autoref{main Theorem} is to rewrite $\mBcomp{t}{3}$ and $\mAcomp{t}{3}$ in mild form. Indeed, recalling the definition of $\mathcal{R}_{\eps,\gamma}$, see \eqref{definition matrix aepsgamma}, the following lemma holds true:
\begin{lemma}\label{mild form}
The following mild formulas hold true
\begin{align}
\label{mild A^3}  
\mAcomp{t}{3}&=e^{t\ke\Delta}\mAcomp{0}{3}+Z^{2,\eps}_t,\\ \label{mild B^3}
   \mBcomp{t}{3}&=e^{t\ke\Delta}\mBcomp{0}{3}+Z^{1,\eps}_t+\int_0^t e^{(t-s)\kappa_{\eps}\Delta}\operatorname{div}\left(\mathcal{R}_{\eps,\gamma}\nabla^{\perp}\mAcomp{s}{3}\right) ds.
\end{align}
\end{lemma}
\begin{proof}
Let us take $\phi=(-\nabla^{\perp}_H(-\Delta)^{-1}\psi_h,0),\ \psi_h=e^{ih\cdot x}$ as test function in \eqref{weak formulation mean} obtaining, since all the functions appearing in the integrals are independent from $x_3$, that it holds
\begin{align}
     \langle \mBcomp{t}{H},-\nabla^{\perp}_H(-\Delta)^{-1}\psi_h\rangle
    &=  \langle \mBcomp{0}{H},-\nabla^{\perp}_H(-\Delta)^{-1}\psi_h\rangle+ \ke \int_0^t  \langle \mBcomp{s}{H},\nabla_{H}^{\perp}\psi_h\rangle \, ds\notag\\ &+\sumkmeanj\int_0^t\langle\skj\cdot \nabla \mBcomp{s}{H}-\mB{s}\cdot\nabla\skjcomp{H},-\nabla^{\perp}_H(-\Delta)^{-1}\psi_h\rangle   dW^{k,j}_s\notag\notag\\ & + \sumkoscj\int_0^t\langle\overline{\skj\cdot \nabla \oBcomp{s}{H}-\oB{s}\cdot\nabla\skjcomp{H}},-\nabla^{\perp}_H(-\Delta)^{-1}\psi_h\rangle   dW^{k,j}_s\label{mild form A3 step 1}.
\end{align}
Recalling that, by \eqref{Biot savart independent from third component} and \autoref{Inversion of Lie derivative}, it holds \begin{align*}
 \mB{t}&=\operatorname{curl}\mA{t}\\
 \sumkmeanj \skj\cdot \nabla \mB{t}-\mB{t}\cdot\nabla\skj&=\sumkmeanj\operatorname{curl}(\skj\cdot \nabla \mA{t}+D\skj \mA{t})\\
 \sumkoscj \skj\cdot \nabla \oB{t}-\oB{t}\cdot\nabla\skj&=\sumkoscj\operatorname{curl}(\skj\cdot \nabla \oA{t}+D\skj \oA{t}),
\end{align*}
looking at the horizontal component, we have in particular that
\begin{align}
 \label{formula 1 mild form A3}\mBcomp{t}{H}&=-\nabla_{H}^{\perp}\mAcomp{t}{3}\\
 \label{formula 2 mild form A3}\sumkmeanj \skj\cdot \nabla \mBcomp{t}{H}-\mB{t}\cdot\nabla\skjcomp{H}&=-\sumkmeanj \nabla_{H}^{\perp}(\skj\cdot \nabla \mAcomp{t}{3})\\
 \sumkoscj \overline{\skj\cdot \nabla \oBcomp{t}{H}-\oB{t}\cdot\nabla\skjcomp{H}}&=-\sumkoscj\nabla_H^{\perp}\left(\overline{\skj\cdot\nabla \oAcomp{t}{3}+\partial_3\skj\cdot\oA{t}} \right)\notag\\ & =-\sumkoscj\nabla_H^{\perp}Q\left(\overline{\skj\cdot\nabla \oAcomp{t}{3}+\partial_3\skj\cdot\oA{t}} \right). \label{formula 3 mild form A3}   
\end{align}
Now both $\skj\cdot \nabla \mAcomp{t}{3}$ an $Q\left(\overline{\skj\cdot\nabla \oAcomp{t}{3}+\partial_3\skj\cdot\oA{t}} \right)$ are zero mean. Substituting \eqref{formula 1 mild form A3}, \eqref{formula 2 mild form A3}, \eqref{formula 3 mild form A3} in \eqref{mild form A3 step 1} and integrating by parts we obtain
\begin{align*}
     \langle \mAcomp{t}{3},\psi_h\rangle
    &=  \langle \mAcomp{0}{3},\psi_h\rangle- \ke\lvert h\rvert^2 \int_0^t  \langle \mAcomp{s}{3},\psi_h\rangle \, ds\\ &+\sumkmeanj\int_0^t\langle\skj\cdot \nabla \mAcomp{s}{3},\psi_h\rangle   dW^{k,j}_s\notag\\ & + \sumkoscj\int_0^t\langle Q\left[\overline{\skj\cdot \nabla \oAcomp{s}{3}+\partial_3\skj\cdot\oA{s}}\right],\psi_h\rangle   dW^{k,j}_s.
\end{align*}
Therefore, applying It\^o formula to $e^{t\ke\lvert h\rvert^2}\langle \mAcomp{t}{3},\psi_h\rangle$ we have
\begin{align}\label{fourier mild form barA3}
\langle \mAcomp{t}{3},\psi_h\rangle&=e^{-t\ke\lvert h\rvert^2}   \langle \mAcomp{0}{3},\psi_h\rangle+\sumkmeanj\int_0^t e^{-(t-s)\ke\lvert h\rvert^2}\langle\skj\cdot \nabla \mAcomp{s}{3},\psi_h\rangle   dW^{k,j}_s\notag\\ & + \sumkoscj\int_0^t e^{-(t-s)\ke\lvert h\rvert^2}\langle Q\left[\overline{\skj\cdot \nabla \oAcomp{s}{3}+\partial_3\skj\cdot\oA{s}}\right],\psi_h\rangle   dW^{k,j}_s\quad \mathbb{P}-a.s.\quad \forall t\in [0,T].
\end{align}
We can find $\Gamma\subseteq \Omega$ of full probability such that the above equality holds for all $t\in [0,T]$ and all $h\in \Z^2_0$. But this is exactly \eqref{mild A^3} written in Fourier modes.\\
Let us consider for $h\in \Z^2_0$ $\phi=(0,0,\psi_h)^t,\ \psi_h=e^{ih\cdot x}$ as test function in \eqref{weak formulation mean} obtaining, since all the functions appearing in the integrals are independent from $x_3$, that it holds
\begin{align*}
     \langle \mBcomp{t}{3},\psi_h\rangle
    &=  \langle \mBcomp{0}{3},\psi_h\rangle- \ke \lvert h\rvert^2\int_0^t  \langle \mBcomp{s}{3},\psi_h\rangle \, ds+\int_0^t \langle \operatorname{div}\left(\mathcal{R}_{\eps,\gamma}\nabla^{\perp}\mAcomp{s}{3}\right),\psi_h\rangle ds \notag\\ &+\sumkmeanj\int_0^t\langle\skj\cdot \nabla \mBcomp{s}{3}-\mB{s}\cdot\nabla\skjcomp{3},\psi_h\rangle   dW^{k,j}_s\notag\\ & + \sumkoscj\int_0^t\langle\overline{\skj\cdot \nabla \oBcomp{s}{3}-\oB{s}\cdot\nabla\skjcomp{3}},\psi_h\rangle   dW^{k,j}_s.
\end{align*}
Therefore, applying It\^o formula to $e^{t\ke\lvert h\rvert^2}\langle \mBcomp{t}{3},\psi_h\rangle$ we obtain 
\begin{align}\label{fourier mild form barB3}
\langle \mBcomp{t}{3},\psi_h\rangle&=e^{-t\ke\lvert h\rvert^2}   \langle \mBcomp{0}{3},\psi_h\rangle+ \int_0^t e^{-(t-s)\kappa_{\eps}\lvert h\rvert^2}\langle \operatorname{div}\left(\mathcal{R}_{\eps,\gamma}\nabla^{\perp}\mAcomp{s}{3}\right),\psi_h\rangle ds \notag\\ & +\sumkmeanj\int_0^t e^{-(t-s)\ke\lvert h\rvert^2}\langle\skj\cdot \nabla \mBcomp{s}{3}-\mB{s}\cdot\nabla\skjcomp{3},\psi_h\rangle   dW^{k,j}_s\notag\\ & + \sumkoscj\int_0^t e^{-(t-s)\ke\lvert h\rvert^2}\langle\overline{\skj\cdot \nabla \oBcomp{s}{3}-\oB{s}\cdot\nabla\skjcomp{3}},\psi_h\rangle   dW^{k,j}_s\quad \mathbb{P}-a.s.\quad \forall t\in [0,T].
\end{align}
We can then find $\Gamma\subseteq \Omega$ of full probability such that the above equality holds for all $t\in [0,T]$ and all $h\in \Z^2_0$. But this is exactly \eqref{mild B^3} written in Fourier modes.
\end{proof}
Secondly we need to study some properties of the stochastic convolutions $Z^{1,\eps}_t,Z^{2,\eps}_t$. \autoref{Lemma 1 Stochastic Convolution} and \autoref{lemma 2 Stochastic Convolution} below are the analogous of \cite[Lemma 2.5]{flandoli2021quantitative} in our framework. In particular it is important to point out that \cite[Assumption 2.4]{flandoli2021quantitative} is false in our case and we have to deal also with the stochastic stretching, which was neglected in previous results. 
\begin{lemma}\label{Lemma 1 Stochastic Convolution}
For each $\delta>0$ and $N\in \N$ large enough such that \autoref{proposition a priori estimate 1} holds, then
\begin{align}\label{estimate 1 conv 1}
 \operatorname{sup}_{t\in [0,T]}\expt{\norm{Z^{1,\eps}_t}_{\Dot{H}^{-\delta}}^2}& \lesssim_{}  \frac{1}{\eps\delta}\left({\norm{B_0^{3,\eps}}_{\eps}^2}+{\norm{\mAcomp{0}{3}}_{\eps}^2}+{\norm{\oA{0}}_{\eps}^2}\right).   \\
\label{estimate 1 conv 2}
 \operatorname{sup}_{t\in [0,T]}\expt{\norm{Z^{2,\eps}_t}_{\Dot{H}^{-\delta}}^2}& \lesssim_{} \frac{1}{\eps\delta}\left({\norm{B_0^{3,\eps}}_{\eps}^2}+{\norm{\mAcomp{0}{3}}_{\eps}^2}+{\norm{\oA{0}}_{\eps}^2}\right).   
\end{align}
The hidden constants in the inequality above depend only on $\eta,\beta,\gamma,C_V,C_{1,H},C_{2,H},\lvert \rho\rvert,T$.
\end{lemma}
\begin{proof}
By Burkholder-Davis-Gundy inequality and \autoref{Properties semigroup} we have for each $t\in [0,T]$
\begin{align}\label{estimate 1 conv 1 step 1}
 \expt{\norm{Z^{1,\eps}_t}_{\Dot{H}^{-\delta}}^2}& \lesssim_{\lvert \rho\rvert} \sumkmeanj\expt{\int_0^t \norm{e^{(t-s)\ke\Delta}\left(\skj\cdot\nabla\mBcomp{s}{3}-\mB{s}\cdot\nabla\skjcomp{3}\right)}_{\Dot{H}^{-\delta}}^2 ds}\notag\\ &+\sumkoscj\expt{\int_0^t  \norm{e^{(t-s)\ke\Delta} \overline{\skj\cdot\nabla\oBcomp{s}{3}-\oB{s}\cdot\nabla\skjcomp{3}}}_{\Dot{H}^{-\delta}}^2 ds} \notag\\ &  \lesssim_{\eta}
  \sumkmeanj\expt{\int_0^t \frac{1}{(t-s)^{1-\delta}}\norm{\skj\cdot\nabla\mBcomp{s}{3}}_{\Dot{H}^{-1}}^2+\norm{\mB{s}\cdot\nabla\skjcomp{3}}^2 ds}\notag\\ &+\sumkoscj\expt{\int_0^t  \frac{1}{(t-s)^{1-\delta}} \norm{\overline{\skj\cdot\nabla\oBcomp{s}{3}}}_{\Dot{H}^{-1}}^2+\norm{\overline{\oB{s}\cdot\nabla\skjcomp{3}}}^2ds}\notag\\ & =\frac{1}{\eps}\sumkmeanj\expt{\int_0^t \frac{1}{(t-s)^{1-\delta}}\norm{\skj\cdot\nabla\mBcomp{s}{3}}_{\Dot{H}^{-1}(\T^{3}_{\eps})}^2+\norm{\mB{s}\cdot\nabla\skjcomp{3}}_{\eps}^2 ds}\notag\\ &+\frac{1}{\eps}\sumkoscj\expt{\int_0^t  \frac{1}{(t-s)^{1-\delta}}\norm{\overline{\skj\cdot\nabla\oBcomp{s}{3}}}_{\Dot{H}^{-1}(\T^{3}_{\eps})}^2 +\norm{\overline{\oB{s}\cdot\nabla\skjcomp{3}}}_{\eps}^2ds}\notag\\ &  = V^{1,\eps}_t+V^{2,\eps}_t+V^{3,\eps}_t+V^{4,\eps}_t.
\end{align}
Now let us estimate $V^{1,\eps}_t,\ V^{2,\eps}_t,\ V^{3,\eps}_t,\ V^{4,\eps}_t $. First we have by H\"older inequality since the $\skj$ are divergence free 

\begin{align}\label{estimate 1 conv 1 step 2}
V^{1,\eps}_t+V^{3,\eps}_t & \leq \frac{1}{\eps}\expt{\sumkmeanjuno {\int_0^t \frac{1}{(t-s)^{1-\delta}} \norm{\skj\mBcomp{s}{3}}^2_\eps ds}+\sumkoscj{\int_0^t \frac{1}{(t-s)^{1-\delta} } \norm{\skj\oBcomp{s}{3}}^2_\eps ds}}\notag\\ & \lesssim_{\beta,C_{1,H},C_V} \frac{1}{\eps}\left({\int_0^t \frac{1}{(t-s)^{1-\delta}} \expt{\norm{\mBcomp{s}{3}}^2_\eps} ds}+\frac{1}{N^{\beta-2}}{\int_0^t \frac{1}{(t-s)^{1-\delta} }} \expt{\norm{\oBcomp{s}{3}}^2_\eps} ds\right)\notag\\ & \lesssim_{\eta,\beta,\gamma,C_V,C_{1,H},C_{2,H},\lvert \rho\rvert,T} \frac{1}{\eps \delta}\left({\norm{B_0^{3,\eps}}_{\eps}^2}+{\norm{\mAcomp{0}{3}}_{\eps}^2}+{\norm{\oA{0}}_{\eps}^2}\right), 
\end{align}
where in the last step we exploit \eqref{a priori estimate barB3 CL LH} and \eqref{apriori estimate B3' CL LH}. For what concerns the other two terms, noting that if $k_3=0$ then $\skjcomp{3}\neq 0$ if and only if $j=2$, thanks to Holder inequality and \eqref{apriori estimate barA3 CL LH}, \eqref{apriori estimate A' CL LH} we get
\begin{align}\label{estimate 1 conv 1 step 3}
    V^{2,\eps}_t+V^{3,\eps}_t &\lesssim_{\beta, \gamma,C_{2,H},C_V} \frac{1}{\eps} \int_0^t \frac{1}{N^{\gamma-4}}\expt{\norm{\mBcomp{s}{H}}_{\eps}^2}+\frac{1}{N^{\beta-4}} \expt{\norm{\oB{s}}^2_{\eps}}ds\notag \\ & = \frac{1}{\eps} \int_0^t \frac{1}{N^{\gamma-4}}\expt{\norm{\nabla\mAcomp{s}{3}}_{\eps}^2}+\frac{1}{N^{\beta-4}} \expt{\norm{\nabla\oA{s}}^2_{\eps}}ds\notag\\ & \lesssim_{\eta,\beta,\gamma,C_V,C_{1,H},C_{2,H},\lvert \rho\rvert,T} \frac{1}{\eps}\left( {\norm{B_0^{3,\eps}}_{\eps}^2}+{\norm{\mAcomp{0}{3}}_{\eps}^2}+{\norm{\oA{0}}_{\eps}^2}\right).
\end{align}
Combining \eqref{estimate 1 conv 1 step 1}, \eqref{estimate 1 conv 1 step 2}, \eqref{estimate 1 conv 1 step 3} relation \eqref{estimate 1 conv 1} follows. The proof of \eqref{estimate 1 conv 2} is similar and we omit the details.
\end{proof}
\begin{lemma}\label{lemma 2 Stochastic Convolution}
For each $\delta>0$ and $N\in \N$ large enough such that \autoref{proposition a priori estimate 1} holds, then
\begin{align}\label{estimate 2 conv 1}
 \operatorname{sup}_{t\in [0,T]}\expt{\norm{Z^{1,\eps}_t}_{\Dot{H}^{-1-2\delta}}^2}\lesssim_{}   &\frac{1}{\eps\delta^2N^2}\left({\norm{B_0^{3,\eps}}_{\eps}^2}+{\norm{\mAcomp{0}{3}}_{\eps}^2}+{\norm{\oA{0}}_{\eps}^2}\right), \\ 
\label{estimate 2 conv 2}
 \operatorname{sup}_{t\in [0,T]}\expt{\norm{Z^{2,\eps}_t}_{\Dot{H}^{-1-2\delta}}^2}\lesssim_{}   &\frac{1}{\eps\delta^2N^2}\left({\norm{B_0^{3,\eps}}_{\eps}^2}+{\norm{\mAcomp{0}{3}}_{\eps}^2}+{\norm{\oA{0}}_{\eps}^2}\right),   
\end{align}
where the hidden constants depend only on $\eta,\beta,\gamma,C_V,C_{1,H},C_{2,H},\lvert \rho\rvert,T$.
\end{lemma}
\begin{proof}
By Burkholder-Davis-Gundy inequality and \autoref{Properties semigroup} we have for each $t\in [0,T]$
\begin{align}\label{estimate 2 conv 1 step 1}
 \expt{\norm{Z^{1,\eps}_t}_{\Dot{H}^{-1-2\delta}}^2}& \lesssim_{\lvert \rho\rvert} \sumkmeanj\expt{\int_0^t \norm{e^{(t-s)\ke\Delta}\left(\skj\cdot\nabla\mBcomp{s}{3}-\mB{s}\cdot\nabla\skjcomp{3}\right)}_{\Dot{H}^{-1-2\delta}}^2 ds}\notag\\ &+\sumkoscj\expt{\int_0^t  \norm{e^{(t-s)\ke\Delta} \overline{\skj\cdot\nabla\oBcomp{s}{3}-\oB{s}\cdot\nabla\skjcomp{3}}}_{\Dot{H}^{-1-2\delta}}^2 ds} \notag\\ &  \lesssim_{\eta}
  \sumkmeanj\expt{\int_0^t \frac{1}{(t-s)^{1-\delta}}\norm{\skj\cdot\nabla\mBcomp{s}{3}}_{\Dot{H}^{-2-\delta}}^2+\norm{\mB{s}\cdot\nabla\skjcomp{3}}_{\Dot{H}^{-1-2\delta}}^2 ds}\notag\\ &+\sumkoscj\expt{\int_0^t  {\norm{\overline{\skj\cdot\nabla\oBcomp{s}{3}}}_{\Dot{H}^{-1}}^2}+\norm{\overline{\oB{s}\cdot\nabla\skjcomp{3}}}_{\Dot{H}^{-1-2\delta}}^2ds}\notag\\ & =\sumkmeanj\expt{\int_0^t \frac{1}{(t-s)^{1-\delta}}\norm{\skj\cdot\nabla\mBcomp{s}{3}}_{\Dot{H}^{-2-\delta}}^2+\norm{\mB{s}\cdot\nabla\skjcomp{3}}_{\Dot{H}^{-1-2\delta}}^2 ds}\notag\\ &+\sumkoscj\expt{\int_0^t  \frac{1}{\eps}\norm{\overline{\skj\cdot\nabla\oBcomp{s}{3}}}_{\Dot{H}^{-1}(\T^{3}_{\eps})}^2 +\norm{\overline{\oB{s}\cdot\nabla\skjcomp{3}}}_{\Dot{H}^{-1-2\delta}}^2ds}\notag\\ &  = Y^{1,\eps}_t+Y^{2,\eps}_t+Y^{3,\eps}_t+Y^{4,\eps}_t.
\end{align}    
The term $Y^{3,\eps}_t$ can be simply treated as $V^{3,\eps}_t$ in the proof of \autoref{Lemma 1 Stochastic Convolution} obtaining 
\begin{align}\label{estimate 2 conv 1 step 2}
    Y^{3,\eps}_t \lesssim_{\eta,\beta,\gamma,C_V,C_{1,H},C_{2,H},\lvert \rho\rvert,T} \frac{1}{\eps N^{\beta-2}}\left({\norm{B_0^{3,\eps}}_{\eps}^2}+{\norm{\mAcomp{0}{3}}_{\eps}^2}+{\norm{\oA{0}}_{\eps}^2}\right).
\end{align}
Let us consider now the others which are the more difficult ones. We will use in the following, even if not specified later, that all the vector fields considered are divergence free, moreover  for $k_3=0$ we $\skjcomp{H}\neq 0$ if and only if $j=1$ and $\skjcomp{3}\neq 0$ if and only if $j=2$. Let us start considering $Y^{1,\eps}_t$ obtaining, thanks to the definition of Sobolev norms in the two dimensional torus 
\begin{align}\label{estimate 2 conv 1 step 3}
 Y^{1,\eps}_t& \lesssim \sumkmeanjuno\expt{\int_0^t \frac{1}{(t-s)^{1-\delta}} \norm{\skjcomp{H} \mBcomp{s}{3}}_{\Dot{H}^{-1-\delta}}^2 ds}\notag\\ & \lesssim   \frac{1}{N^2}\sum_{k\in \Z^2_0}\expt{\int_0^t \frac{1}{(t-s)^{1-\delta}} \norm{ e^{ik\cdot x} \mBcomp{s}{3}}_{\Dot{H}^{-1-\delta}}^2 ds}\notag\\ & =  \frac{1}{2\pi N^2}{\int_0^t \frac{1}{(t-s)^{1-\delta}}\sum_{k\in \Z^2_0}\sum_{j\in \Z^2_0}\frac{1}{\lvert j\rvert^{2+2\delta}} \expt{\langle \mBcomp{s}{3}e^{ik\cdot x},e^{-ij\cdot x} \rangle^2} ds}\notag\\ & =\frac{1}{2\pi N^2}{\int_0^t \frac{1}{(t-s)^{1-\delta}}\sum_{k\in \Z^2_0} \expt{\langle \mBcomp{s}{3},e^{-ik\cdot x} \rangle^2}\sum_{j\in \Z^2_0}\frac{1}{\lvert j\rvert^{2+2\delta}}ds}\notag\\ & \lesssim \frac{1}{\delta N^2} \int_0^t \frac{1}{(t-s)^{1-\delta}}\expt{\norm {\mBcomp{s}{3}}^2} ds \notag\\ & \lesssim_{\eta,\beta,\gamma,C_V,C_{1,H},C_{2,H},\lvert \rho\rvert,T} \frac{1}{\eps\delta^2 N^2} \left({\norm{B_0^{3,\eps}}_{\eps}^2}+{\norm{\mAcomp{0}{3}}_{\eps}^2}+{\norm{\oA{0}}_{\eps}^2}\right),
\end{align}
where in the last step we exploited \eqref{a priori estimate barB3 CL LH}. Let us consider $Y^{2,\eps}_t$ obtaining, thanks to \eqref{equivalence norm barv3}
\begin{align}\label{estimate 2 conv 1 step 4}
Y^{2,\eps}_t & \lesssim \sumkmeanjdue \int_0^t \expt{\norm{\skjcomp{3}\mBcomp{s}{H}}^2} ds \notag \\ & \lesssim_{C_{2,H}} \sum_{\substack{k\in\Z^2_0\\ N\leq \lvert k\rvert \leq 2N}} \frac{1}{\lvert k\rvert^{\gamma}} \int_0^T \expt{\norm{\mBcomp{s}{H}}^2} ds\notag \\ & \lesssim_{\gamma} \frac{1}{N^{\gamma-2}} \int_0^T \expt{\norm{\nabla\mAcomp{s}{3}}^2} ds\notag\\ &  \lesssim_{\eta,\beta,\gamma,C_V,C_{1,H},C_{2,H},\lvert \rho\rvert,T} \frac{1}{\eps N^{\gamma-2}} \left({\norm{B_0^{3,\eps}}_{\eps}^2}+{\norm{\mAcomp{0}{3}}_{\eps}^2}+{\norm{\oA{0}}_{\eps}^2}\right),
\end{align}
where in the last step we exploited \eqref{apriori estimate barA3 CL LH}. Lastly we treat $Y^{4,\eps}_t$, obtaining thanks to \eqref{equivalence norm A'} and the fact that $\operatorname{div}$ and $M^{\eps}$ commute
\begin{align}\label{estimate 2 conv 1 step 5}
    Y^{4,\eps}_t& \lesssim\sumkoscj\expt{\int_0^t  \norm{\overline{\oB{s}\skjcomp{3}}}^2ds} \notag\\
     &= \frac{1}{\eps}\sumkoscj\expt{\int_0^t  \norm{\overline{\oB{s}\skjcomp{3}}}_{\eps}^2ds} \notag\\ & \leq \frac{1}{\eps}\sumkoscj\expt{\int_0^t  \norm{{\oB{s}\skjcomp{3}}}_{\eps}^2ds} \notag
     \\ & \lesssim_{\beta, C_V} \frac{1}{\eps N^{\beta-2}} \expt{\int_0^t  \norm{\oB{s}}_{\eps}^2ds }\notag\\ & \lesssim_{\eta,\beta,\gamma,C_V,C_{1,H},C_{2,H},\lvert \rho\rvert,T} \frac{1}{\eps N^{\beta-2}} \left({\norm{B_0^{3,\eps}}_{\eps}^2}+{\norm{\mAcomp{0}{3}}_{\eps}^2}+{\norm{\oA{0}}_{\eps}^2}\right).
\end{align}
Relation \eqref{estimate 2 conv 1} follows combining \eqref{estimate 2 conv 1 step 1}, \eqref{estimate 2 conv 1 step 2}, \eqref{estimate 2 conv 1 step 3}, \eqref{estimate 2 conv 1 step 4}, \eqref{estimate 2 conv 1 step 5}. The proof of \eqref{estimate 2 conv 2} is similar and we omit the details.
\end{proof}
Interpolating between \autoref{Lemma 1 Stochastic Convolution} and \autoref{lemma 2 Stochastic Convolution} we obtain:
\begin{corollary}\label{corollary stochastic convolution}
For each $\theta\in (0,1]$ and $\delta\in (0,\theta]$ it holds
\begin{align*}
\operatorname{sup}_{t\in [0,T]}\expt{\norm{Z^{1,\eps}_t}_{\Dot{H}^{-\theta}}^2} \lesssim_{}   &\frac{1}{\eps\delta^{\frac{1+\theta}{1+\delta}}N^{\frac{2(\theta-\delta)}{1+\delta}}}\left({\norm{B_0^{3,\eps}}_{\eps}^2} +{\norm{\mAcomp{0}{3}}_{\eps}^2}+{\norm{\oA{0}}_{\eps}^2}\right).   \\
 \operatorname{sup}_{t\in [0,T]}\expt{\norm{Z^{2,\eps}_t}_{\Dot{H}^{-\theta}}^2}\lesssim_{}   &\frac{1}{\eps\delta^{\frac{1+\theta}{1+\delta}}N^{\frac{2(\theta-\delta)}{1+\delta}}}\left({\norm{B_0^{3,\eps}}_{\eps}^2}+{\norm{\mAcomp{0}{3}}_{\eps}^2}+{\norm{\oA{0}}_{\eps}^2}\right).    
\end{align*}
The hidden constants in the inequality above depend only on  $\eta,\beta,\gamma,C_V,C_{1,H},C_{2,H},\lvert \rho\rvert,T$.
\end{corollary}
Moreover, we recall that by classical theory of evolution equations, see for example \cite{Lunardi}, or arguing as in \autoref{mild form}, the unique weak solutions of \eqref{limit solution B3}, \eqref{limit solution A3} can be written in mild form as
\begin{align*}
\overline{A_t^{3,\gamma}}=e^{t(\eta+\eta_T)\Delta}\overline{A^3_0}, \quad \overline{B_t^{3,\gamma}}=e^{t(\eta+\eta_T)\Delta}\overline{B^3_0}+\int_0^t e^{(t-s)(\eta+\eta_T)\Delta}\operatorname{div}\left(\mathcal{R}_{\gamma}\nabla^{\perp}\overline{A^3_s}\right) ds.
\end{align*}
Now we introduce some intermediate functions between $(\mBcomp{t}{3},\mAcomp{t}{3})$ and $(\overline{B_t^3},\overline{A_t^3})$. Let $\left(\widehat{B_t^{3,\eps}},\widehat{A_t^{3,\eps}}\right)$ the unique weak solution of the linear system
\begin{align}\label{intermediate systems A}
&\begin{cases}
\partial_t \widehat{A^{3,\eps}_t}&=(\eta+\eta^{\eps}_T)\Delta \widehat{A^{3,\eps}_t}\quad x\in \T^2,\ t\in (0,T)\\
\widehat{A^{3,\eps}_t}|_{t=0}&=\mAcomp{0}{3},   
\end{cases}\\
\label{intermediate systems B}
&\begin{cases}
\partial_t \widehat{B^{3,\eps}_t}&=(\eta+\eta^{\eps}_T)\Delta \widehat{B^{3,\gamma}_t}+\operatorname{div}\left(\mathcal{R}_{\eps,\gamma}\nabla^{\perp}\widehat{A^{3,\eps}_t}\right)\quad x\in \T^2,\ t\in (0,T)\\
\widehat{B^{3,\eps}_t}|_{t=0}&=\mBcomp{0}{3}.   
\end{cases}    
\end{align}
Again, by classical theory of evolution equations, see for example \cite{Lunardi}, or arguing as in \autoref{mild form}, the unique weak solutions of \eqref{intermediate systems A}, \eqref{intermediate systems B} can be written in mild form as
\begin{align*}
\widehat{A_t^{3,\eps}}=e^{t(\eta+\eta^{\eps}_T)\Delta}\mAcomp{0}{3}, \quad \widehat{B_t^{3,\eps}}=e^{t(\eta+\eta^{\eps}_T)\Delta}\mBcomp{0}{3}+\int_0^t e^{(t-s)(\eta+\eta_T^{\eps})\Delta} \operatorname{div}\left(\mathcal{R}_{\eps,\gamma}\nabla^{\perp}\widehat{A^{3,\eps}_s}\right) ds.
\end{align*}
Now we prove that $(\widehat{B^{3,\eps}_t},\widehat{A^{3,\eps}_t})$
and $(\overline{B^{3,\gamma}_t},\overline{A^{3,\gamma}_t})$ are close.
\begin{lemma}\label{preliminary convergence}
Under the same assumptions of \autoref{main Theorem}, for each  $\theta_1\in (0,1],\ \theta_2\in (0,\theta_1)$ we have
\begin{align}
\operatorname{sup}_{t\in [0,T]}\norm{\widehat{A^{3,\eps}_t}-\overline{A^{3,\gamma}_t}}_{\Dot{H}^{-\theta_2}}&\lesssim_{\eta,T}\lvert \eta^{\eps}_T-\eta_T\rvert^{\theta_2/2}+\norm{\overline{v^3_0}-\mAcomp{0}{3}}_{\Dot{H}^{-\theta_2}}, \label{intermediate estimate A}   \\
\operatorname{sup}_{t\in [0,T]}\norm{\widehat{B^{3,\eps}_t}-\overline{B^{3,\gamma}_t}}_{\Dot{H}^{-\theta_1}}&\lesssim_{\eta,T} \begin{cases}
    \frac{\lvert\eta^{\eps}_T-\eta_T\rvert^{\theta_2/2}+\norm{\overline{A^3_0}-\mAcomp{0}{3}}_{\Dot{H}^{-\theta_2}}}{\theta_1-\theta_2}+\norm{\overline{B^3_0}-\mBcomp{0}{3}}_{\Dot{H}^{-\theta_1}}+\frac{1}{\theta_1 N}\\  \text{if }\gamma=4,\\
    \frac{\lvert\eta^{\eps}_T-\eta_T\rvert^{\theta_2/2}+\norm{\overline{A^3_0}-\mAcomp{0}{3}}_{\Dot{H}^{-\theta_2}}}{\theta_1-\theta_2}+\norm{\overline{B^3_0}-\mBcomp{0}{3}}_{\Dot{H}^{-\theta_1}}+\frac{1}{\theta_1 N^{\gamma/2-2}} \\ \text{if }\gamma>4.
\end{cases}\label{intermediate estimate B}    
\end{align}
\end{lemma}
\begin{proof}
 First let us observe that by assumptions
\begin{align*}
\mBcomp{0}{3}\rightharpoonup \overline{B^3_0} \text{ in }\Dot L^2(\T^2),\quad \mAcomp{0}{3}\rightharpoonup \overline{A^3_0} \text{ in }\Dot L^2(\T^2)   
\end{align*}
 we have in particular that the families $\left\{\norm{\mBcomp{0}{3}}^2\right\}_{\eps\in (0,1)},\ \left\{\norm{\mAcomp{0}{3}}^2\right\}_{\eps\in (0,1)}$ are bounded. Therefore
\begin{align}\label{assumptions initial conditions preliminary lemma}
    \operatorname{sup}_{\eps\in (0,1)}\frac{{\norm{\mBcomp{0}{3}}_{\eps}^2}}{\eps}=\operatorname{sup}_{\eps\in (0,1)}{\norm{\mBcomp{0}{3}}}^2<+\infty,\quad \mBcomp{0}{3}\rightarrow \overline{B^3_0} \text{ in }\Dot{H}^{-\theta_2}(\T^2)\quad \forall \theta_2>0,\notag\\
    \operatorname{sup}_{\eps\in (0,1)}\frac{{\norm{\mAcomp{0}{3}}_{\eps}^2}}{\eps}=\operatorname{sup}_{\eps\in (0,1)}{\norm{\mAcomp{0}{3}}}^2<+\infty,\quad \mAcomp{0}{3}\rightarrow \overline{A^3_0} \text{ in }\Dot{H}^{-\theta_2}(\T^2)\quad \forall \theta_2>0.
\end{align}    
Moreover, since by \autoref{convergence properties matrix} $\mathcal{R}_{\eps,\gamma}\rightarrow \mathcal{R}_{\eps},$ then in particular we have
\begin{align}\label{uniform bound matrix}
    \operatorname{sup}_{\eps\in (0,1)}\norm{\mathcal{R}_{\eps,\gamma}}_{HS}<+\infty.
\end{align}
Lastly, by simple energy estimates on \eqref{intermediate systems A} we have for each $\eps\in (0,1)$
\begin{align}\label{uniform bound auxiliary A}
    \operatorname{sup}_{t\in [0,T]}\norm{\widehat{A^{3,\eps}_t}}^2+\int_0^T \norm{\nabla\widehat{A^{3,\eps}_s}}^2 ds\lesssim_{\eta} \operatorname{sup}_{\eps\in (0,1)}{\norm{\mAcomp{0}{3}}}^2<+\infty.
\end{align}
The convergence of $\widehat{A^{\eps,3}_t}$ to $\overline{A^{3,\gamma}_t}$ then follows by triangle inequality, \autoref{Properties semigroup} and \eqref{assumptions initial conditions preliminary lemma}. Indeed for each $t\in [0,T]$ it holds
\begin{align*}
    \norm{\widehat{A^{3,\eps}_t}-\overline{A^{3,\gamma}_t}}_{\Dot{H}^{-\theta_2}}&\leq  \norm{\left(e^{t(\eta+\eta_T^{\eps})\Delta}-e^{t(\eta+\eta_T)\Delta}\right)\overline{A^3_0}}_{\Dot{H}^{-\theta_2}}+\norm{e^{t(\eta+\eta^{\eps}_T)\Delta}\left(\mAcomp{0}{3}-\overline{A^3_0}\right)}_{\Dot{H}^{-\theta_2}}\\ & \lesssim_{\eta,T}\lvert \eta^{\eps}_T-\eta_T\rvert^{\theta_2/2}\norm{\overline{A^3_0}}+\norm{\overline{A^3_0}-\mAcomp{0}{3}}_{\Dot{H}^{-\theta_2}}\\ & \lesssim \lvert \eta^{\eps}_T-\eta_T\rvert^{\theta_2/2}+\norm{\overline{A^3_0}-\mAcomp{0}{3}}_{\Dot{H}^{-\theta_2}}.
\end{align*}
Thanks to \eqref{intermediate estimate A} we can obtain also \eqref{intermediate estimate B}. Indeed thanks to triangle inequality for each $t\in [0,T]$ we have
\begin{align*}
\norm{\widehat{B^{3,\eps}_t}-\overline{B^{3,\gamma}_t}}_{\Dot{H}^{-\theta_1}}& \leq \left(\norm{\left(e^{t(\eta+\eta_T^{\eps})\Delta}-e^{t(\eta+\eta_T)\Delta}\right)\overline{B^3_0}}_{\Dot{H}^{-\theta_1}}+\norm{e^{t(\eta+\eta^{\eps}_T)\Delta}\left(\mBcomp{0}{3}-\overline{B^3_0}\right)}_{\Dot{H}^{-\theta_1}}\right)\\& +\norm{\int_0^t \left(e^{(t-s)(\eta+\eta_T^{\eps})\Delta}-e^{(t-s)(\eta+\eta_T)\Delta}\right) \operatorname{div}\left(\mathcal{R}_{\eps,\gamma}\nabla^{\perp}\widehat{A^{3,\eps}_s}\right) ds}_{\Dot{H}^{-\theta_1}}  \\ & + \norm{\int_0^t e^{(t-s)(\eta+\eta_T)\Delta} \operatorname{div}\left(\left(\mathcal{R}_{\eps,\gamma}-\mathcal{R}_{\gamma}\right)\nabla^{\perp}\widehat{A^{3,\eps}_s}\right) ds}_{\Dot{H}^{-\theta_1}}\\ & + \norm{\int_0^t e^{(t-s)(\eta+\eta_T)\Delta} \operatorname{div}\left(\mathcal{R}_{\gamma}\left(\nabla^{\perp}\widehat{A^{3,\eps}_s}-\nabla^{\perp}\overline{A^{3,\gamma}_s}\right)\right) ds}_{\Dot{H}^{-\theta_1}}=I_1+I_2+I_3+I_4. 
\end{align*}
$I_1$ can be treated as above obtaining:
\begin{align}\label{intermediate B step 1}
I_1\lesssim_{\eta,T}\lvert \eta^{\eps}_T-\eta_T\rvert^{\theta_1/2}+\norm{\overline{B^3_0}-\mBcomp{0}{3}}_{\Dot{H}^{-\theta_1}}.    
\end{align}
Thanks to \autoref{Properties semigroup}, \eqref{uniform bound matrix} and \eqref{uniform bound auxiliary A} it holds
\begin{align}\label{intermediate B step 2}
I_2 & \lesssim \lvert \eta^{\eps}_T-\eta_T\rvert^{\theta_1/2}\int_0^t\lvert t-s\rvert^{\theta_1/2}\norm{e^{(t-s)\eta\Delta}\operatorname{div}\left(\mathcal{R}_{\eps,\gamma}\nabla^{\perp}\widehat{A^{3,\eps}_s}\right)} ds\notag\\ & \lesssim_{\eta}  \lvert\eta^{\eps}_T-\eta_T\rvert^{\theta_1/2}\int_0^t \frac{1}{(t-s)^{1-\theta_1/2}}\norm{\widehat{A^{3,\eps}_s}} ds\notag\\ & \lesssim_{T} \frac{1}{\theta_1} \lvert\eta^{\eps}_T-\eta_T\rvert^{\theta_1/2}.    
\end{align}
Due to \autoref{Properties semigroup} and \eqref{uniform bound auxiliary A} we can treat $I_3$ obtaining:
\begin{align}\label{intermediate B step 3}
  I_3&\lesssim_{\eta,\theta_1} \int_0^T \frac{1}{(t-s)^{1-\theta_1/2}} \norm{\operatorname{div}\left(\left(\mathcal{R}_{\eps,\gamma}-\mathcal{R}_{\gamma}\right)\nabla^{\perp}\widehat{A^{3,\eps}_s}\right)}_{\Dot{H}^{-2}} ds \notag\\ & \lesssim \norm{\mathcal{R}_{\eps,\gamma}-\mathcal{R}_{\gamma}}_{HS}\int_0^T \frac{1}{(t-s)^{1-\theta_1/2}} \norm{\widehat{A^{3,\eps}_s}} ds \notag\\ & \lesssim_{T}\frac{1}{\theta_1}\norm{\mathcal{R}_{\eps,\gamma}-\mathcal{R}_{\gamma}}_{HS}.
\end{align}
In particular, combining \eqref{intermediate B step 3} and \autoref{convergence properties matrix} we have that 
\begin{align}\label{intermediate B step 4}
    I_3\lesssim_{\eta,T}\begin{cases}
        \frac{1}{\theta_1 N} & \text{if } \gamma=4,\\
        \frac{1}{\theta_1 N^{\gamma/2-2}} & \text{if } \gamma>4. 
    \end{cases}
\end{align}
Lastly we treat $I_4$ combining \autoref{Properties semigroup} and \eqref{intermediate estimate A}:
\begin{align}\label{intermediate B step 5}
 I_4 & \lesssim_{\eta} \int_0^t \frac{1}{(t-s)^{1+\theta_2/2-\theta_1/2}} \norm{\widehat{A^{3,\eps}_s}-\overline{A^{3,\gamma}_s}}_{\Dot{H}^{-\theta_2}}ds\notag\\ & \lesssim_{T} \frac{1}{\theta_1-\theta_2}\operatorname{sup}_{t\in [0,T]}  \norm{\widehat{A^{3,\eps}_t}-\overline{A^{3,\gamma}_t}}_{\Dot{H}^{-\theta_2}}\notag\\ & \lesssim_{\eta,T} \frac{\lvert \eta^{\eps}_T-\eta_T\rvert^{\theta_2/2}+\norm{\overline{A^3_0}-\mAcomp{0}{3}}_{\Dot{H}^{-\theta_2}}}{\theta_1-\theta_2} .
\end{align}
Combining \eqref{intermediate B step 1},\eqref{intermediate B step 2},\eqref{intermediate B step 4},\eqref{intermediate B step 5} we get the thesis.
\end{proof}
Secondly we provide a quantitative result on the closeness of 
$(\widehat{B^{3,\eps}_t},\widehat{A^{3,\eps}_t})$
and $(\mBcomp{t}{3},\mAcomp{t}{3}).$
\begin{lemma}\label{preliminary convergence 2}
Under the same assumptions of \autoref{main Theorem}, there exists $N_0$ large enough such that for each $N\geq N_0$, for each  $\theta_1\in (0,1],\ \theta_2\in (0,\theta_1),\ \delta\in (0,\theta_2)$ we have
\begin{align}
\operatorname{sup}_{t\in [0,T]}\expt{\norm{\widehat{A^{3,\eps}_t}-\mAcomp{t}{3}}^2_{\Dot{H}^{-\theta_2}}}&\lesssim_{}   \frac{1}{\delta^{\frac{1+\theta_2}{1+\delta}}N^{\frac{2(\theta_2-\delta)}{1+\delta}}}, \label{intermediate estimate 2 A}   \\
\operatorname{sup}_{t\in [0,T]}\expt{\norm{\widehat{B^{3,\eps}_t}-\mBcomp{t}{3}}_{\Dot{H}^{-\theta_1}}^2}&\lesssim_{}   \frac{1}{\theta_2^{2-\theta_1+\theta_2}\delta^{\frac{(\theta_1-\theta_2)(1+\theta_2)}{2(1+\delta)}}N^{\frac{(\theta_2-\delta)(\theta_1-\theta_2)}{(1+\delta)}}},
\label{intermediate estimate 2 B}    
\end{align}    
where the hidden constants depend only on $\eta,\beta,\gamma,C_V,C_{1,H},C_{2,H},\lvert \rho\rvert, T$.
\end{lemma}
\begin{proof}
From \autoref{mild form}, we already know that $\mBcomp{t}{3}$ and $\mAcomp{t}{3}$ can be rewritten in mild form as 
\begin{align*}
   \mBcomp{t}{3}=e^{t(\eta+\eta_T^{\eps})\Delta}\mBcomp{0}{3}+Z_{t}^{1,\eps}+\int_0^t e^{(t-s)(\eta+\eta^{\eps}_T)\Delta}\operatorname{div}\left(\mathcal{R}_{\eps,\gamma}\nabla^{\perp}\mAcomp{s}{3}\right) ds,\quad \mAcomp{t}{3}=e^{t(\eta+\eta_T^{\eps})\Delta}\mAcomp{0}{3}+Z_{t}^{2,\eps}.
\end{align*}
Therefore denoting by 
\begin{align*}
    D^{\eps}_t=\mBcomp{t}{3}-\widehat{B^{3,\eps}_t},\quad d^{\eps}_t=\mAcomp{t}{3}-\overline{A^3_t},
\end{align*}
we have
\begin{align*}
    D^{\eps}_t=Z_{t}^{1,\eps}+\int_0^t e^{(t-s)(\eta+\eta^{\eps}_T)\Delta}\operatorname{div}\left(\mathcal{R}_{\eps,\gamma}\nabla^{\perp}d^{\eps}_s\right) ds,\quad d^{\eps}_s=Z_{t}^{2,\eps}.
\end{align*}    
Due to \autoref{corollary stochastic convolution} we have immediately \eqref{intermediate estimate 2 A}. Now we move to the analysis of $D^{\eps}_t$. $Z^{1,\eps}_t$ can be treated easily by \autoref{corollary stochastic convolution} obtaining 
\begin{align}\label{intermediate D step 1}
\operatorname{sup}_{t\in [0,T]}\expt{\norm{Z^{1,\eps}_t}^2_{\Dot{H}^{-\theta_1}}}&\lesssim_{\eta,\beta,\gamma,C_V,C_{1,H},C_{2,H},\lvert \rho\rvert, T}   \frac{1}{\delta^{\frac{1+\theta_1}{1+\delta}}N^{\frac{2(\theta_1-\delta)}{1+\delta}}}.    
\end{align}
In order to treat the deterministic convolution, we first argue by interpolation obtaining thanks to \autoref{Properties semigroup}, \eqref{uniform bound matrix}
\begin{align}\label{intermediate D step 2}
 &\norm{\int_0^t e^{(t-s)(\eta+\eta^{\eps}_T)\Delta}\operatorname{div}\left(\mathcal{R}_{\eps,\gamma}\nabla^{\perp}d^{\eps}_s\right) ds}^2_{\Dot{H}^{-\theta_1}}\notag\\ &\leq \norm{\int_0^t e^{(t-s)(\eta+\eta^{\eps}_T)\Delta}\operatorname{div}\left(\mathcal{R}_{\eps,\gamma}\nabla^{\perp}d^{\eps}_s\right) ds}^{\theta_1-\theta_2}_{\Dot{H}^{-2-\theta_2}}   \norm{\int_0^t e^{(t-s)(\eta+\eta^{\eps}_T)\Delta}\operatorname{div}\left(\mathcal{R}_{\eps,\gamma}\nabla^{\perp}d^{\eps}_s\right) ds}^{2-\theta_1+\theta_2}_{\Dot{H}^{-\theta_2}}\notag\\ &   \lesssim_{\eta} \left(\int_0^t \norm{d^{\eps}_s}_{\Dot{H}^{-\theta_2}}ds \right)^{\theta_1-\theta_2}\left(\int_0^t \frac{1}{(t-s)^{1-\theta_2/2}} \norm{\widehat{A^{3,\eps}_s}-\mAcomp{s}{3}} ds\right)^{2-\theta_1+\theta_2}\notag\\ & \lesssim_{T} \frac{1}{\theta_2^{2-\theta_1+\theta_2}}\left(\int_0^t \norm{d^{\eps}_s}_{\Dot{H}^{-\theta_2}}^2ds \right)^{\frac{\theta_1-\theta_2}{2}}\operatorname{sup}_{t\in [0,T]}\left(\norm{\mAcomp{t}{3}}+\norm{\widehat{A^{3,\eps}_s}}\right)^{2-\theta_1+\theta_2}.
\end{align}
Combining \eqref{intermediate D step 2} and \eqref{intermediate estimate 2 A}, \eqref{corollary compactness in space}, \eqref{uniform bound auxiliary A} we obtain an estimate of the deterministic convolution. Indeed, by Holder inequality, for each $t\in [0,T]$ it holds:
\begin{align}\label{intermediate D step 3}
&\expt{\norm{\int_0^t e^{(t-s)(\eta+\eta^{\eps}_T)\Delta}\operatorname{div}\left(\mathcal{R}_{\eps,\gamma}\nabla^{\perp}d^{\eps}_s\right) ds}^2_{\Dot{H}^{-\theta_1}}} \notag\\ & \lesssim_{\eta,T}\frac{1}{\theta_2^{2-\theta_1+\theta_2}} \expt{\left(\int_0^T \norm{d^{\eps}_s}_{\Dot{H}^{-\theta_2}}^2ds \right)^{\frac{\theta_1-\theta_2}{2}}\operatorname{sup}_{t\in [0,T]}\left(\norm{\mAcomp{t}{3}}+\norm{\widehat{A^{3,\eps}_s}}\right)^{2-\theta_1+\theta_2}}\notag\\ & \leq \frac{1}{\theta_2^{2-\theta_1+\theta_2}} \expt{\int_0^T \norm{d^{\eps}_s}_{\Dot{H}^{-\theta_2}}^2ds }^{\frac{\theta_1-\theta_2}{2}}\expt{\operatorname{sup}_{t\in [0,T]}\left(\norm{\mAcomp{t}{3}}^2+\norm{\widehat{A^{3,\eps}_s}}^2\right)}^{\frac{2-\theta_1+\theta_2}{2}}\notag \\ & \lesssim_{\eta,\beta,\gamma,C_V,C_{1,H},C_{2,H},\lvert \rho\rvert, T}   \frac{1}{\theta_2^{2-\theta_1+\theta_2}\delta^{\frac{(\theta_1-\theta_2)(1+\theta_2)}{2(1+\delta)}}N^{\frac{(\theta_2-\delta)(\theta_1-\theta_2)}{(1+\delta)}}}.
\end{align}
Combining \eqref{intermediate D step 1} and \eqref{intermediate D step 3} the thesis follows.
\end{proof}

Now we are ready to prove \autoref{main Theorem}.
\begin{proof}[Proof of \autoref{main Theorem}]
Recalling that, due to the properties of Riemann sums,
\begin{align*}
    \lvert \eta^{\eps}_T-\eta_T\rvert\lesssim \frac{1}{N},
\end{align*}then the thesis follows immediately by triangle inequality combining \autoref{preliminary convergence} and \autoref{preliminary convergence 2}. 
\end{proof}

\begin{acknowledgements}
    We thank Professor Dejun Luo for useful discussions and valuable insights into the subject.
\end{acknowledgements}
\begin{funding}
    The research of the second author is funded by the European Union (ERC, NoisyFluid, No. 101053472). Views and opinions expressed are however those of the authors only and do not necessarily reflect those of the European Union or the European Research Council. Neither the European Union nor the granting authority can be held responsible for them. 
\end{funding}
\appendix 

\section{Proof of \autoref{lemma ito strat corrector}}\label{appendix ito strat corr}
We start showing that 
\begin{align*}
    \sumkj \mathcal{L}_{\skj}\mathcal{L}_{\smkj} F&=\Lambda^{\eps}F.
\end{align*}
First let us observe that it holds
\begin{align*}
 \sumkj \mathcal{L}_{\skj}\mathcal{L}_{\smkj} F&= \sumkj \operatorname{div}(\skj\otimes \smkj \nabla F)-R\\& =\operatorname{div}\left(Q^{\eps}_0(0) \nabla F\right)-R\\ & = \Lambda^{\eps}F-R,
\end{align*}
where we denoted by 
\begin{align*}
R=\sumkj\skj\cdot\nabla\left(F\cdot\nabla\smkj\right)+\left(\smkj\cdot\nabla F\right)\cdot\nabla\skj-(F\cdot\nabla\smkj)\cdot\nabla\skj.
\end{align*}
Therefore, it is enough to show that 
\begin{align}\label{claim ito strat corrector}
R=0.
\end{align}
First we observe that 
\begin{align*}
   \skj\cdot\nabla\left(F\cdot\nabla\smkj\right)=(\skj\cdot\nabla F)\cdot\nabla\smkj+\skj\cdot \nabla^2 \smkj F. 
\end{align*}
Therefore 
\begin{align}\label{relation 1}
    R&=\left(\sumkj  (\skj\cdot\nabla F)\cdot\nabla\smkj+\left(\smkj\cdot\nabla F\right)\cdot\nabla\skj\right)+\left(\sumkj \skj\cdot \nabla^2 \smkj F-(F\cdot\nabla\smkj)\cdot\nabla\skj \right)\notag\\ & =  R_1+R_2.
\end{align}
$R_1$ is easily $0$ due to \eqref{property derivatives 5}. In order to show that $R_2=0$, we observe that, recalling the definition of $\skj$ and expanding the sum, it holds for each $k\in \Z^{3,N}_0,\ j\in \{1,2\}$
\begin{align}\label{relation 2}
\skj\cdot \nabla^2 \smkj F-(F\cdot\nabla\smkj)\cdot\nabla\skj&= -2\sum_{l,m\in\{1,2,3\}}\left\lvert\tkj\right\rvert^2 a_{k,j} k_l k_m(a_{k,j})_l F_m \notag\\ & =2\skj\cdot \nabla^2 \smkj F\notag\\ & =2\sum_{m\in \{1,2,3\}}\skj\cdot \nabla\partial_m\smkj F_m. 
\end{align}
Let us show that for each $m\in \{1,2,3\}$ it holds $\sumkj\skj\cdot \nabla\partial_m\smkj=0$.
Due to the fact that $\skj$ are divergence free, we have
\begin{align}\label{relation 3}
\sumkj\skj\cdot \nabla\partial_m\smkj&=\operatorname{div}\left( \sumkj \skj\otimes \partial_m \smkj\right)\notag\\ & =
i\operatorname{div}\left( \sumkj \left\lvert\tkj\right\rvert^2 k_m a_{k,j}\otimes a_{k,j}\right)=0.
\end{align}
Our claim \eqref{claim ito strat corrector} then follows combining \eqref{relation 1}, \eqref{relation 2}, \eqref{relation 3}.\\
Now let us prove that
\begin{align*}
    \rho\sumkmeancov \left(\mathcal{L}_{\sk{1}}\mathcal{L}_{\smk{2}}+\mathcal{L}_{\sk{2}}\mathcal{L}_{\smk{1}}\right)F&=\Lambda_{\rho}^{\eps}F.
\end{align*}
First let us observe that it holds
\begin{align*}
    \sumkmeancov \left(\mathcal{L}_{\sk{1}}\mathcal{L}_{\smk{2}}+\mathcal{L}_{\sk{2}}\mathcal{L}_{\smk{1}}\right)F&=\sumkmeancov \operatorname{div}\left((\sk{1}\otimes \smk{2}+\sk{2}\otimes \smk{1}) \nabla F\right)-\Tilde{R}\\ &=\operatorname{div}\left(\overline{Q^{\eps}_{\rho}}(0) \nabla F\right)-\Tilde{R},
\end{align*}
where we denoted by
\begin{align*}
    \Tilde{R}&=\sumkmeancov\sk{1}\cdot\nabla\left(F\cdot\nabla\smk{2}\right)+\left(\smk{1}\cdot\nabla F\right)\cdot\nabla\sk{2}-(F\cdot\nabla\smk{1})\cdot\nabla\sk{2}\\ &+\sumkmeancov\sk{2}\cdot\nabla\left(F\cdot\nabla\smk{1}\right)+\left(\smk{2}\cdot\nabla F\right)\cdot\nabla\sk{1}-(F\cdot\nabla\smk{2})\cdot\nabla\sk{1}.
\end{align*}
Since $\overline{Q^{\eps}_{\rho}}(0)=0$, we are left to show that 
\begin{align}\label{claim off diagonal}
    -\rho\Tilde{R}=\Lambda^{\eps}_{\rho}F.
\end{align}
First we observe that
\begin{align*}
\sk{1}\cdot\nabla\left(F\cdot\nabla\smk{2}\right)+\sk{2}\cdot\nabla\left(F\cdot\nabla\smk{1}\right)&=\left(\sk{1}\cdot\nabla F\right)\cdot\nabla\smk{2}+\sk{1}\cdot\nabla^2 \smk{2} F\\ &+\left(\sk{2}\cdot\nabla F\right)\cdot\nabla\smk{1}+\sk{2}\cdot\nabla^2 \smk{1} F.
\end{align*}
Therefore 
\begin{align}\label{step 1 claim off diagonal}
    \Tilde{R}&=\left(\sumkmeancov \left(\smk{1}\cdot\nabla F\right)\cdot\nabla\sk{2}+\left(\smk{2}\cdot\nabla F\right)\cdot\nabla\sk{1}+\left(\sk{1}\cdot\nabla F\right)\cdot\nabla\smk{2}+\left(\sk{2}\cdot\nabla F\right)\cdot\nabla\smk{1}\right)\notag\\ & + \left(\sumkmeancov \sk{1}\cdot\nabla^2-(F\cdot\nabla\smk{1})\cdot\nabla\sk{2} \smk{2} F+\sk{2}\cdot\nabla^2 \smk{1} F-(F\cdot\nabla\smk{2})\cdot\nabla\sk{1}\right)\notag\\ &= \Tilde{R}_1+\Tilde{R}_2.
\end{align}
$\Tilde{R}_2=0$ arguing as above when showing that $R_2=0.$ We have for each $k$ such that $k_3=0$
\begin{align}\label{step 2 claim off diagonal}
    \left(\smk{1}\cdot\nabla F\right)\cdot\nabla\sk{2}=\left(\sk{1}\cdot\nabla F\right)\cdot\nabla\smk{2},\quad \left(\smk{2}\cdot\nabla F\right)\cdot\nabla\sk{1}=\left(\sk{2}\cdot\nabla F\right)\cdot\nabla\smk{1}
\end{align}
due to our choice of the $\skj$. Therefore 
\begin{align*}
    \Tilde{R}=2\sumkmeancov \left(\sk{1}\cdot\nabla F\right)\cdot\nabla\smk{2}+\left(\sk{2}\cdot\nabla F\right)\cdot\nabla\smk{1}.
\end{align*}
Since, due to \eqref{property derivatives 3}, we have for each $l\in \{1,2,3\}$
\begin{align*}
    \partial_l \overline{Q^{\eps}_{\rho}}(0)=2i\rho \sumkmeancov \left(\tkjcomp{1}{\tkjcomp{2}}^* a_{k,1}\otimes a_{k,2}+\tkjcomp{2}{\tkjcomp{1}}^* a_{k,2}\otimes a_{k,1}\right) k_l,
\end{align*}
it follows that 
\begin{align}\label{step 3 claim off diagonal}
  -\rho \Tilde{R}=-\sum_{l\in \{1,2\}} \partial_l \overline{Q^{\eps}_{\rho}}(0)\cdot\nabla F_l.
\end{align}
Relation \eqref{claim off diagonal} then is a consequence of \eqref{step 1 claim off diagonal}, \eqref{step 2 claim off diagonal}, \eqref{step 3 claim off diagonal}. 
\begin{remark}\label{remark ineffective on 2D functions}
Due to \autoref{lemma ito strat corrector}, recalling the definition of $\mathcal{R}_{\eps,\gamma}$ in \eqref{definition matrix aepsgamma}, it follows that if $F:\T^2\rightarrow  \R^3$ we have
\begin{align*}
    \Lambda_{\rho}^\eps F&=-\sum_{l\in \{1,2\}}(0,0,\partial_l \overline{Q^{\eps,3,H}_{\rho}}(0)\cdot\nabla_H F_l)^t\\ & =2\rho\sum_{l\in \{1,2\}}\sumkmeancov (0,0,\tkjcomp{1}\tkjcomp{2} k_l a_{k,1}\cdot\nabla_H F_l)^t\\ & =2\rho\sum_{l\in \{1,2\}}\sumkmeancov (0,0,\tkjcomp{1}\tkjcomp{2}\operatorname{div}\left(k_l a_{k,1}F_l\right))^t\\ & = -\left(0,0,\operatorname{div}\left(\mathcal{R}_{\eps,\gamma} F_H\right)\right)^t.
\end{align*}
More in general if $F:\T^3_{\eps}\rightarrow \R^3$ we have
\begin{align*}
    \Lambda^{\eps}_{\rho}F&=-\sum_{l\in\{1,2\}}(\partial_l \overline{Q^{\eps,H,3}_{\rho}}(0)\partial_3 F_l,\partial_l \overline{Q^{\eps,3,H}_{\rho}}(0)\cdot\nabla_H F_l)^t\\ & =-\left(\sum_{l\in\{1,2\}}\partial_3(\partial_l \overline{Q^{\eps,H,3}_{\rho}}(0) F_l),\operatorname{div}_H\left(\mathcal{R}_{\eps,\gamma} F_H\right)\right)^t.
\end{align*}
\end{remark}
\section{Some physical remarks}\label{appB}
In this section we discuss some physical properties of our random vector field connected with the appearance of the alpha-term, in particular we show that the helicity of our vector field approach a constant value in the limit, despite the noise is weakly converging to zero.
Let us compute the time average of the helicity of our random velocity field (in the following we will omit the dependence on time and space)
\begin{align*}
    \H^{\eps,\gamma}:= \frac{1}{2T}\expt{W^\eps(T,x)\cdot (\grad \times W^\eps(T, x))}.
\end{align*}
As before we split our noise in mean and fluctuations along the thin direction $W^\eps = \overline{W^\eps}+ {W^\eps}'$, and split the mean part in horizontal and vertical components $\overline{W^\eps}= \overline{W^{\eps,H}} + \overline{W^{\eps,3}}e_3$. Observe that \begin{align*}
    \grad \times \overline{W^\eps} = (-\grad^\perp \overline{W^{\eps,3}}, \grad^\perp \cdot \overline{W^{\eps,H}}),
\end{align*} where with a small abuse of notation we have indicated with $\overline{W^{\eps,3}}$ the scalar quantity $\overline{W^\eps}\cdot e_3$
\begin{align*}
    \H^{\eps,\gamma} = \frac{1}{2T}\expt{ -\overline{W^{\eps,H}} \cdot \grad^\perp \overline{W^{\eps,3}} + \overline{W^{\eps,3}}\grad^\perp \cdot \overline{W^{\eps,H}} + {W^\eps}'\cdot (\grad \times \overline{W^\eps}) + \overline{W^\eps}\cdot (\grad \times {W^\eps}') + {W^\eps}'\cdot (\grad \times {W^\eps}')}.
\end{align*}
Notice that the mixed terms are zero by the independence of the Brownian motions, thus we are left with the first two terms, and the last one. We start by the first two, we have, by our choices of the noise coefficients,  
\begin{align*}
    -\overline{W^{\eps,H}} \cdot \grad^\perp \overline{W^{\eps,3}} &= -\sum_{\substack{k, k'\in \Z^2_0,\\ N\le |k|,|k'|\le 2N}} \theta^{\eps}_{k,1}\theta^{\eps}_{k',2} a_k^1\cdot (i{k'}^\perp)e^{i(k+k')\cdot x} W^{k,1}W^{k', 2}\\
    &= \frac{1}{C_{1,H}C_{2,H}}\sum_{\substack{k, k'\in \Z^2_0,\\ N\le |k|,|k'|\le 2N}}  \frac{k^\perp\cdot (k')^\perp}{|k|^{2}|k'|^{\gamma/2}} e^{i(k+k')\cdot x}W^{k,1}W^{k', 2}.
\end{align*}
Taking the expected value and remembering that $\expt{W_T^{k,1}W_T^{k',2}} = 2T\rho \delta_{k+k'} $ we get 
\begin{equation*}
    \expt{-\overline{W^{\eps,H}} \cdot \grad^\perp \overline{W^{\eps,3}}} = -\frac{2\rho T}{C_{1,H}C_{2,H}} \sum_{\substack{k\in \Z^2_0 \\ N\le |k|\le 2N}} \frac{1}{|k|^{\gamma/2}}.
\end{equation*}
Analogously, we have 
\begin{align*}
    \overline{W^{\eps,3}} (\grad^\perp \cdot\overline{W^{\eps, H}}) &= -\frac{1}{C_{1,H}C_{2,H}}\sum_{\substack{k, k'\in \Z^2_0,\\ N\le |k|,|k'|\le 2N}}  \frac{1}{|k|^{\gamma/2}} e^{i(k+k')\cdot x}W^{k,2}W^{k', 1}. 
\end{align*}
From which 
\begin{align*}
    \expt{ \overline{W^{\eps,3}} (\grad^\perp \cdot\overline{W^{\eps, H}})}= - \frac{2\rho T}{C_{1,H}C_{2,H}}\sum_{\substack{k\in \Z^2_0 \\ N\le |k|\le 2N}} \frac{1}{|k|^{\gamma/2}}.
\end{align*}
Finally for the last term 
\begin{align*}
    {W^\eps}'\cdot (\grad \times {W^\eps}')= \sum_{\substack{k, k'\in \Z^{3, \eps}_0,\\ k_3, k'_3\neq 0 \\ N\le |k_H|,|k_H'|\le 2N\\ i,j\in\{1,2\}}}\theta^{\eps}_{i,k}\theta^{\eps}_{k,k'} a_k^i\cdot (ik'\times a_{k'}^j) e^{i(k+k')\cdot x}W^{k,i}W^{k', j}.
\end{align*}
Taking expectation, we see that only the terms with $k'=-k$ and $i=j$ survive but then, $a_k^i\cdot (k\times a_{-k}^i)= a_k^i\cdot (k\times a_{k}^i)= 0$. 
In the end $\H^{\eps,\gamma}$ does not vanish in the limit if and only if $\gamma=4$. Indeed, we have
\begin{align}\label{elicity convergence}
    \H^{\eps,\gamma}=-\frac{2\rho\zeta^{N}_{H,\gamma/2}}{N^{\gamma/2-2}C_{1,H}C_{2,H}}\rightarrow \H^{\gamma}=\begin{cases}
        -\frac{4\pi\rho\log 2}{C_{1,H}C_{2,H}} &\text{if } \gamma=4\\
        0 &\text{if } \gamma>4.
    \end{cases}
\end{align}
Due to \eqref{elicity convergence} it follows that
\begin{align}\label{relation alpha effect elicity}
\mathcal{R}_{\eps,\gamma}=-\frac{\H^{\eps,\gamma}}{2}\begin{bmatrix} 
   0 & -1\\ 1 & 0   \end{bmatrix},\ 
\mathcal{R}_{\gamma}=-\frac{\H^{\gamma}}{2}\begin{bmatrix} 
   0 & -1\\ 1 & 0   \end{bmatrix}.
\end{align}
Finally, due to \eqref{relation alpha effect elicity}, denoting by 
\begin{align*}
 \overline{B^{\gamma}_t}=(-\nabla^{\perp} \overline{A^{3,\gamma}_t}, \overline{B^{3,\gamma}_t})^t, \quad  \mathcal{A}^{\gamma}=-\frac{\mathcal{H}^{\gamma}}{2}\begin{bmatrix} 
   1& 0 & 0\\ 0 &1 & 0\\ 0 & 0 & 0   \end{bmatrix},   
\end{align*}
 our limit equations \eqref{limit solution A3}, \eqref{limit solution B3} read as
\begin{align}\label{limit system vector form}
    \begin{cases}
        \partial_t \overline{B^{\gamma}_t} &=(\eta+\eta_T)\Delta \overline{B^{\gamma}_t}+\nabla\times (\mathcal{A}^{\gamma}\overline{B^{\gamma}_t})\quad x\in \T^2,\ t\in (0,T)\\ 
        \operatorname{div}\overline{B^{\gamma}_t}& =0\\
        \overline{B^{\gamma}_t}|_{t=0}&= \overline{B}_0.
    \end{cases}
\end{align}
\bibliography{biblio}{}
\bibliographystyle{plain}

\end{document}